\documentclass[a4paper,11pt]{article} 

\newcommand\authormark[1]{\textsuperscript{#1}}
\usepackage[numbers,sort&compress]{natbib}
\usepackage{opticameet3}

\usepackage[colorlinks=true,
            bookmarks=false,
            citecolor=blue,
            urlcolor=blue]{hyperref}

\usepackage{amsmath,amssymb}
\usepackage[colorlinks=true,bookmarks=false,citecolor=blue,urlcolor=blue]{hyperref} 
\usepackage[utf8]{inputenc}
\usepackage[T1]{fontenc}
\usepackage{booktabs}
\usepackage{multirow}
\usepackage{amsthm} 
\usepackage{amssymb} 
\usepackage{caption}
\usepackage{subcaption}
\newtheorem{theorem}{Theorem}[section]   
\newtheorem{lemma}{Lemma}[section]       
\newtheorem{example}{Example}
\newtheorem{definition}{Definition}[section] 
\newtheorem{assumption}{Assumption}[section]
\newtheorem{corollary}{Corollary}[section]
\newtheorem{remark}{Remark}[section]

\usepackage{graphicx}
\usepackage{amsmath}
\usepackage[version=4]{mhchem}
\usepackage{siunitx}
\usepackage{longtable,tabularx}
\usepackage{xcolor} 

\renewcommand{\u}{\mathbf{u}} 
\newcommand{\B}{\mathbf{b}} 
\renewcommand{\v}{\mathbf{v}} 
\newcommand{\C}{\mathbf{w}} 
\renewcommand{\L}{\mathbf{L}}

\numberwithin{equation}{section}
\numberwithin{definition}{section}

\begin{document}


\title{Finite Element Approximation of Rotating MHD-Boussinesq Flows
with Nonlinear Power-Law Viscosity}






\author{Akash Ashirbad Panda,\authormark{1}
Evana Islam Sarkar,\authormark{2}}

\address{
\authormark{1}Department of Mathematics, 
Indian Institute of Technology Indore, Madhya Pradesh-452020, India
\\
\authormark{2}Department of Mathematics, School of Basic Sciences,
Indian Institute of Technology Bhubaneswar, Odisha-752050, India
}

\email{
\authormark{1}akashpanda@iiti.ac.in, \ 
\authormark{2}s24ma09008@iitbbs.ac.in
}

\begin{abstract}
{\small{
We propose and analyse a fully discrete finite element approximation of the
three-dimensional incompressible rotating MHD--Boussinesq system with nonlinear
power-law ($p$-Laplace) viscosity. Conforming elements are used for the velocity
and temperature, $\mathbf H(\operatorname{curl})$-conforming N\'ed\'elec elements
for the magnetic induction, and a semi-implicit backward Euler scheme in time.
The discretisation preserves the principal energy structure: the Lorentz and
induction terms cancel exactly, while the Coriolis force remains energy neutral.
This yields well-posedness and mesh-independent energy stability.

The nonlinear power-law diffusion creates two main analytical difficulties. Since a uniform
time-derivative estimate for the discrete velocity is unavailable, strong compactness is
obtained through a time-translate argument and Simon's compactness criterion. The weak limit
of the nonlinear diffusion is then identified by a Minty argument based on a discrete energy
$\limsup$ inequality. Compactness and passage to the limit hold throughout the
shear-thickening regime, while identification of a genuine weak solution is obtained in the
Newtonian case and, in three dimensions, above an explicit threshold dictated by the
Lorentz force rather than by convection. We further establish a~priori error estimates
without any mesh-time coupling condition. Three-dimensional numerical experiments confirm
the stability and convergence behaviour predicted by the analysis.
}}
\end{abstract}

\vspace{0.5cm}


\medskip


\section{Introduction}\label{sec-introduction}

Magnetoconvection in electrically conducting fluids arises in applications ranging from
planetary dynamo action to liquid-metal cooling and plasma confinement: buoyancy drives the
fluid motion, the motion induces a magnetic field, and the resulting Lorentz force feeds back
on the flow. In many such regimes the effective viscosity is non-Newtonian; shear-dependent
behaviour may arise from microstructure, temperature-dependent material parameters, or
turbulence closures, and is commonly described by a power law. In the momentum equation this
leads to a $p$-Laplacian diffusion operator, recovering the Newtonian model for $p=2$ and
describing shear-thickening behaviour for $p>2$. In rotating systems, the fluid dynamics are
additionally influenced by the Coriolis force, which redistributes momentum while remaining
energy neutral.

\subsection{The rotating MHD-Boussinesq model}

Let $\mathbb{D}\subset\mathbb{R}^{3}$ be a bounded Lipschitz polyhedron,
assumed simply connected with connected boundary $\partial\mathbb D$.
For a final time $T>0$, set
\[
\mathbb D_T:=\mathbb D\times(0,T).
\]
We consider the nonstationary incompressible rotating MHD-Boussinesq system
with constant density and $p$-Laplacian viscosity, $p\geq2$:
\begin{align}
\label{1st equation on MHD-BSn}
\u_{t}
-\nu\operatorname{div}\!\big(|\nabla\u|^{p-2}\nabla\u\big)
+(\u\cdot\nabla)\u
+2\boldsymbol{\omega}_{r}\times\u
+\nabla P^{\ast}
+\B\times(\nabla\times\B)
&=\theta e_{3}+\mathbf f
&&\mbox{in }\mathbb D_T,
\\
\label{2nd equation on MHD-BSn}
\B_{t}
+\kappa\nabla\times(\nabla\times\B)
-\nabla\times(\u\times\B)
&=\mathbf g
&&\mbox{in }\mathbb D_T,
\\
\label{3rd equation on MHD-BSn}
\theta_{t}
-\sigma\Delta\theta
+\u\cdot\nabla\theta
&=\u\cdot e_{3}+\psi
&&\mbox{in }\mathbb D_T,
\\
\operatorname{div}\u&=0
&&\mbox{in }\mathbb D_T,
\\
\operatorname{div}\B&=0
&&\mbox{in }\mathbb D_T.
\end{align}
The system is supplemented with the initial conditions
\begin{align}
\label{initial condition for MHD bsn}
\u(x,0)=\u^{0},\qquad
\B(x,0)=\B^{0},\qquad
\theta(x,0)=\theta^{0},
\qquad x\in\mathbb D,
\end{align}
and the boundary conditions
\begin{align}
\label{boundary condition for MHD bsn}
\u=0,\qquad
\mathbf n\times(\nabla\times\B)=0,\qquad
\mathbf n\cdot\B=0,\qquad
\theta=0
\qquad\mbox{on }\partial\mathbb D.
\end{align}
Here $\mathbf n$ denotes the outward unit normal and
$\operatorname{div}\B^{0}=0$.
\noindent
The unknowns $\u$, $\B$, $P^{\ast}$, and $\theta$ denote, respectively,
the fluid velocity, magnetic induction, modified pressure, and temperature
fluctuation. The parameters $\nu,\kappa,\sigma>0$ denote the viscosity,
magnetic diffusivity, and thermal diffusivity, while
$\mathbf f$, $\mathbf g$, and $\psi$ are the body force, applied current, and heat source.
We assume that $\mathbf g$ is solenoidal so that the magnetic divergence constraint is
propagated in time. The momentum equation contains the buoyancy force $\theta e_3$ and the
Lorentz force $\B\times(\nabla\times\B)$, whereas the induction equation couples the magnetic
field and velocity through $\nabla\times(\u\times\B)$. The temperature equation contains
the reciprocal Boussinesq coupling $\u\cdot e_3$. Such systems arise naturally in
magnetoconvection \cite{hashizume2006numerical,smoletntsev2010mhd} and are closely related to
Rayleigh-B\'enard convection in electrically conducting fluids
\cite{marchioro1994mathematical,mulone2003necessary}.
\\\noindent
The prescribed angular velocity of the rotating frame is denoted by $\boldsymbol{\omega}_{r}\in\mathbb{R}^{3},$ and is assumed constant in space and time. For rotation about the vertical axis, we take
\[
\boldsymbol{\omega}_{r}=\Omega_{r}e_{3},
\qquad \Omega_{r}\geq 0,
\]
where $e_{3}=(0,0,1)^{\top}$. The momentum equation then contains the Coriolis force $2\boldsymbol{\omega}_{r}\times\u .$ The associated centrifugal term can be incorporated into the modified pressure since
\[
\boldsymbol{\omega}_{r}\times
\big(\boldsymbol{\omega}_{r}\times x\big)
=
-\nabla\left(\frac12
\left|\boldsymbol{\omega}_{r}\times x\right|^{2}\right).
\]
For vertical rotation,
\[
2\boldsymbol{\omega}_{r}\times\u
=
2\Omega_{r}
\begin{pmatrix}
-u_{2}\\
u_{1}\\
0
\end{pmatrix}.
\]
Moreover, the Coriolis force is energy neutral:
\begin{equation}\label{eq-coriolis-energy}
\big(\boldsymbol{\omega}_{r}\times\v,\v\big)=0
\qquad
\forall\,\v\in\mathbf L^{2}(\mathbb D),
\end{equation}
and, in particular,
\[
\big(2\boldsymbol{\omega}_{r}\times\u,\u\big)=0.
\]
Thus rotation modifies the flow dynamics without directly producing or dissipating
kinetic energy, a structure that is retained by the fully discrete scheme.
\\ \noindent
Two structural conventions are important. First, the magnetic boundary conditions in
\eqref{boundary condition for MHD bsn} play different roles:
$\mathbf n\cdot\B=0$ is an essential condition and, together with
$\operatorname{div}\B=0$, is incorporated into the solenoidal space $\mathcal K$ introduced
in Section~\ref{subsec-spaces}, whereas
$\mathbf n\times(\nabla\times\B)=0$ is a natural boundary condition recovered from the weak
formulation. Thus the tangential trace $\mathbf n\times\B$ is not prescribed, and the natural
trial and test space for the magnetic field is
$\mathcal W=\mathbf H(\operatorname{curl};\mathbb D)$.
Second, the nonlinear diffusion is formulated in terms of $\nabla\u$ rather than the symmetric
gradient $\mathbf D\u$. This standard simplification preserves the monotonicity and coercivity
needed in the analysis; the corresponding symmetric-gradient formulation may be treated using
Korn's inequality in $\mathbf W^{1,p}$.

\subsection{State of the art}\label{subsec-literature}

\paragraph{MHD-Boussinesq systems.}
In two dimensions, global well-posedness was established by Bian
\cite{bian2016initial}, and for temperature-dependent coefficients by Bian and Gui
\cite{bian2016boussinesq,bian2017initial}; Ye \cite{ye2021global} treated the case of vanishing
thermal diffusion, while a first-order semi-implicit Euler scheme was analysed in
\cite{MR4783950}. Finite element approximations of stationary thermally coupled MHD systems
were studied in \cite{Meir1995,MeirSchmidt2001,Cheng2009,TabataTagami2005}. For the
evolutionary problem, Ravindran \cite{Ravindran2019} derived error estimates for a
Crank-Nicolson scheme, while Qiu \cite{Qiu2020} considered a fully discrete semi-implicit
Euler method.

These approaches employ Lagrange elements for the magnetic field and therefore rely on
$\B\in\mathbf H^{1}(\mathbb D)$, a regularity property that may fail on nonconvex polyhedra
\cite{CostabelDauge1998,CostabelDauge2000}. This motivates the use of
$\mathbf H(\operatorname{curl})$-conforming N\'ed\'elec edge elements. For the isothermal
MHD system, edge-element formulations were studied in
\cite{Dauge1989,CostabelDauge2000,Prohl2008}; related results for variable-density,
discontinuous Galerkin, and fully discrete MHD approximations can be found in
\cite{LiQiuYang2022,BanasProhl2010,LiuWalkington2007,Walkington2005,LiLi2024,ding2022convergence}.

\paragraph{Finite element methods for power-law diffusion.}

 The nonlinear diffusion associated with the $p$-Laplacian has motivated a substantial body of work on the stability and accuracy of finite element approximations. Barrett and Liu \cite{BarrettLiu1994} analysed backward Euler finite element approximations using linear elements and established error estimates for the parabolic $p$-Laplacian. For nonlinear parabolic systems with $p$-structure, quasi-norm techniques were developed in \cite{DieningEbmeyerRuzicka2007,BerselliRuzicka2022}, and subsequent work by Breit et al. \cite{BreitDieningStornWichmann2021} removed restrictions coupling the spatial mesh size to the temporal step. In the setting of incompressible non-Newtonian fluids, Berselli et al. \cite{BerselliDieningRuzicka2015} established optimal quasi-norm error bounds for the $p$-Navier–Stokes equations. Error analysis for fully discrete approximations of the (p)-Stokes problem was presented in \cite{EcksteinRuzicka2018}, while higher-order temporal discretisation strategies were investigated in \cite{Hansen2007,EmmrichWroblewska2013}. These studies provide analytical tools for treating nonlinear constitutive laws and motivate their extension to coupled electromagnetic and thermal systems.
\paragraph{Research gap.}
Although numerical methods for MHD-Boussinesq equations and finite element techniques for power-law fluids have been studied extensively, their combination introduces additional analytical difficulties. In particular, the three-dimensional rotating MHD-Boussinesq system considered here couples nonlinear viscous diffusion with magnetic induction, buoyancy-driven temperature evolution, and Coriolis forces. To the best of our knowledge, a fully discrete finite element analysis addressing this combination, with an $\mathbf H(\operatorname{curl};\mathbb D))$-conforming approximation of the magnetic field, has not been established in the cited literature.

A closely related starting point is the analysis of the three-dimensional incompressible MHD system with $p$-Laplacian viscosity by Le, Panda, and Sarkar ~\cite{LePandaSarkar2026}, which addresses stability, convergence, and a priori error estimation. The present study develops this framework for rotating magnetoconvection by incorporating the temperature equation, the Boussinesq coupling, and the Coriolis force. These additional couplings require corresponding extensions of the discrete energy estimates, compactness arguments, and error analysis. At the same time, the skew-symmetry of the rotational contribution provides an important structural property for retaining the energy estimates used in the analysis.
\subsection{Contributions}\label{subsec-contributions}

We construct and analyse a fully discrete method combining conforming finite elements for
$\u$ and $\theta$, N\'ed\'elec edge elements for $\B$, and a semi-implicit backward Euler
time discretisation. The main contributions are as follows.

\begin{enumerate}

\item[(i)] \emph{A structure-preserving semi-implicit scheme.}
The Lorentz force and induction coupling are evaluated using the same magnetic field
$\B_h^{m-1}$, so that the triple-product identity \eqref{eq-triple-product} yields exact
cancellation in the discrete energy balance. The Coriolis term is treated at the current
time level and satisfies
\[
\big(\boldsymbol{\omega}_{r}\times\u_h^m,\u_h^m\big)=0.
\]
Together with the skew-symmetric forms $a_1$ and $a_2$, these properties give existence,
uniqueness, and energy stability without any coupling condition between $\Delta t$ and $h$.
Moreover, since $\nabla S_h\subset\mathcal W_h$, the constraint
$(\B_h^m,\nabla s_h)=0$ provides the discrete counterpart of
$\operatorname{div}\B=0$ and $\mathbf n\cdot\B=0$.

\item[(ii)] \emph{Compactness without a velocity time-derivative bound.}
The $p$-Laplacian prevents a uniform bound on
$\partial_t\u_{h,\Delta t}$, so the classical Aubin-Lions argument is not directly available
for the velocity. Instead, we establish a time-translate estimate in which the Lorentz term is
controlled through the Hodge mapping \eqref{eq-hodge}, while the linear Coriolis contribution
is controlled directly by the uniform velocity bound, and then apply Simon's compactness
criterion.

\item[(iii)] \emph{Identification of the nonlinear flux.}
The weak limit of
$|\nabla\u_{h,\Delta t}|^{p-2}\nabla\u_{h,\Delta t}$
is identified by Minty's monotonicity method using an energy $\limsup$ inequality derived from
the discrete energy identity. Compactness and passage to the limit in all remaining terms hold
for every $p\ge2$, while identification of the nonlinear flux is obtained for
$p=2$ \ or \ $p\ge\frac83.$
In three dimensions, the additional restriction is dictated by the Lorentz force rather than
the convective term; the Coriolis force introduces no further restriction because it is linear
and energy neutral.

\item[(iv)] \emph{Unconditional error estimates.}
Under Assumption~\ref{Assumption} and for a pointwise divergence-free velocity pair, we derive
the estimate
\[
\mathcal O\!\left(\Delta t+h^{\frac{lp}{2(p-1)}}\right)
\]
in $L^\infty(L^2)$ for the velocity, magnetic induction, and temperature, reducing to the
classical
$\mathcal O(\Delta t+h^l)$ rate when $p=2$. The analysis employs a linear Stokes projection,
so that the $p$-Laplacian contribution appears explicitly as a consistency term, while the
pointwise divergence-free property allows a quadratic Gr\"onwall argument. The additional
Coriolis contribution is linear and is controlled without changing the resulting rate.

\item[(v)] \emph{Numerical validation.}
Three-dimensional experiments with manufactured solutions illustrate energy stability,
first-order temporal convergence, and the expected spatial convergence behaviour in both the
Newtonian and shear-thickening regimes.

\end{enumerate}

\subsection{Organization of the paper}\label{subsec-organization}

Section~\ref{sec-functional-setting} introduces the functional setting and weak formulation,
while Section~\ref{sec-fem} presents the finite element spaces and fully discrete scheme.
Sections~\ref{sec-existence-stability}, \ref{sec-convergence}, and \ref{sec-error} establish,
respectively, well-posedness and stability, convergence to a weak solution, and a~priori error
estimates. Section~\ref{sec-numerics} presents the numerical experiments, and
Section~\ref{sec-conclusion} concludes the paper.

\vspace{0.2cm}

\section{Functional Setting and Mathematical Formulation of the Problem}\label{sec-functional-setting}

This section specifies the functional framework for analysing the
rotating MHD-Boussinesq system
\eqref{1st equation on MHD-BSn}-\eqref{3rd equation on MHD-BSn}
under the initial and boundary conditions
\eqref{initial condition for MHD bsn}-\eqref{boundary condition for MHD bsn}.
We define the spaces required for the velocity, magnetic induction,
and temperature, and collect the inequalities and compactness results
used in the subsequent analysis. We also establish the algebraic
identities associated with the convective, rotational, and
electromagnetic terms that are needed for the energy estimates.

\subsection{Standing assumptions and notation}\label{subsec-notation}

\paragraph{Domain, time interval, and flow index.}

We work on a bounded Lipschitz polyhedral domain
$\mathbb D\subset\mathbb R^3$ that is simply connected and has
a connected boundary $\partial\mathbb D$. The outward unit normal
is denoted by $\mathbf n$, and the time interval is $(0,T)$,
where $T>0$. We use the notation
$\mathbb D_T=\mathbb D\times(0,T)$.
The assumed topology of the domain is used in the magnetic-field
estimates, particularly in the Poincar\'e-type inequality
\eqref{eq-poincare-curl}. The power-law exponent is restricted to
$p\geq 2$, with $p=2$ representing the Newtonian case and $p>2$
the shear-thickening case. We write
$p'=p/(p-1)$ for the conjugate exponent. The rotation vector
$\boldsymbol{\omega}_r\in\mathbb R^3$ is assumed to be constant
in space and time.

\paragraph{Function spaces and notation.}
For $j\ge1$ and $1\le q\le\infty$, $W^{j,q}(\mathbb D)$ denotes the usual Sobolev space,
with norm $\|\cdot\|_{j,q}$, and $H^j(\mathbb D)=W^{j,2}(\mathbb D)$.
A non-integer first index refers to the corresponding Sobolev-Slobodeckij space.
For a Banach space $X$, $L^q(0,T;X)$ denotes the usual Bochner space, and
\[
\|v\|_{L^\infty(0,T;X)}
=
\operatorname*{ess\,sup}_{0\le t\le T}\|v(t)\|_X.
\]
We write $\mathcal C^j(\mathbb D)$ for $j$-times continuously differentiable functions
and $\mathcal C_0^\infty(\mathbb D)$ for smooth compactly supported functions; see
\cite{GiraultRaviart1986}.
\\ \noindent
Vector- and tensor-valued quantities and spaces are written in boldface; for instance,
$\u=(u_1,u_2,u_3)$ and
$\mathbf L^2(\mathbb D)=(L^2(\mathbb D))^3$, with analogous notation for
$\mathbf W^{j,q}$ and $\mathbf H^j$. The $L^2$ inner product is denoted by $(\phi,\psi)=\int_{\mathbb D}\phi\,\psi\,dx,$ with corresponding norm $\|\cdot\|_{L^2}$; the same notation is used componentwise for
vector and tensor fields. We also write
$\|\mathbf v\|_{j,q}:=\|\mathbf v\|_{\mathbf W^{j,q}}$ and set
$e_3=(0,0,1)^{\top}$. Throughout, $C>0$ denotes a generic constant independent of
$h$ and $\Delta t$, whose value may change from line to line.

\subsection{Function spaces}\label{subsec-spaces}

We now define the spaces for the three unknown fields in accordance
with their respective boundary conditions and divergence constraints \eqref{boundary condition for MHD bsn}.
The velocity space incorporates the no-slip condition, the magnetic
space accounts for the curl and normal-trace requirements, and the
temperature space incorporates its homogeneous Dirichlet condition.

\vspace{0.1cm}

\noindent
{\bf Velocity and pressure.}
The no-slip condition and the $p$-Laplace structure lead to
\begin{equation}\label{eq-spaces-velocity}
\mathcal X=\mathbf W_0^{1,p}(\mathbb D),\qquad
\mathcal X_0=
\left\{
\mathbf v\in\mathcal X:
\operatorname{div}\mathbf v=0
\text{ in }\mathbb D
\right\},
\end{equation}
together with
\[
Q=L_0^2(\mathbb D)
=
\left\{
q\in L^2(\mathbb D):
\int_{\mathbb D}q\,dx=0
\right\}.
\]
We equip $\mathcal X$ and $\mathcal X_0$ with the norm
$\|\nabla\cdot\|_{\mathbf L^p}$, equivalent to the full
$\mathbf W^{1,p}$-norm on $\mathbf W_0^{1,p}(\mathbb D)$. The pressure space $Q$ is used
for the discrete inf-sup condition; the discrete scheme is posed on a divergence-free
velocity space and therefore does not contain the pressure explicitly.
\\ \noindent
{\bf Magnetic induction.}
Since $\mathbf n\cdot\B=0$ is imposed essentially whereas
$\mathbf n\times(\nabla\times\B)=0$ is a natural boundary condition, we set
\begin{equation}\label{eq-spaces-magnetic}
\mathcal W
=
\mathbf H(\operatorname{curl};\mathbb D)
=
\left\{
\mathbf w\in\mathbf L^2(\mathbb D):
\nabla\times\mathbf w\in\mathbf L^2(\mathbb D)
\right\},
\end{equation}
with norm
\[
\|\mathbf w\|_{\mathcal W}
=
\left(
\|\mathbf w\|_{\mathbf L^2}^2
+
\|\nabla\times\mathbf w\|_{\mathbf L^2}^2
\right)^{1/2}.
\]
No tangential boundary condition is imposed on $\mathcal W$.
\\ \noindent
We further introduce
\begin{align}\label{eq-spaces-div}
\mathbf H(\operatorname{div};\mathbb D)
&=
\left\{
\mathbf w\in\mathbf L^2(\mathbb D):
\operatorname{div}\mathbf w\in L^2(\mathbb D)
\right\},
\notag\\
\mathbf H_0(\operatorname{div};\mathbb D)
&=
\left\{
\mathbf w\in\mathbf H(\operatorname{div};\mathbb D):
\mathbf w\cdot\mathbf n=0
\text{ on }\partial\mathbb D
\right\}.
\end{align}
The solenoidal magnetic space is
\begin{equation}\label{eq-space-K}
\mathcal K
=
\left\{
\mathbf b\in
\mathcal W\cap\mathbf H_0(\operatorname{div};\mathbb D):
\operatorname{div}\mathbf b=0
\right\}.
\end{equation}
Thus $\mathcal K$ incorporates both
$\operatorname{div}\B=0$ and $\mathbf n\cdot\B=0$. The normal trace is understood in
$H^{-1/2}(\partial\mathbb D)$. Since $\mathbb D$ is simply connected with connected boundary,
$\mathcal K$ contains no nontrivial harmonic fields; consequently,
$\|\nabla\times\cdot\|_{\mathbf L^2}$ is a norm on $\mathcal K$.

\vspace{0.1cm}

\noindent
{\bf Temperature.}
The homogeneous Dirichlet condition on the temperature leads to
\begin{equation}\label{eq-spaces-temperature}
\mathcal Y=H^1(\mathbb D),
\qquad
\mathcal Y_0=H_0^1(\mathbb D).
\end{equation}

\vspace{0.1cm}

\noindent
{\bf Solenoidal spaces.}
We shall also use
\begin{align}\label{eq-spaces-test}
\mathcal J
&=
\left\{
\mathbf w\in\mathcal C_0^\infty(\mathbb D):
\operatorname{div}\mathbf w=0
\right\},
\notag\\
\mathcal G
&=
\left\{
\mathbf w\in\mathcal C^\infty(\overline{\mathbb D}):
\operatorname{div}\mathbf w=0
\text{ in }\mathbb D,\quad
\mathbf w\cdot\mathbf n=0
\text{ on }\partial\mathbb D
\right\},
\notag\\
\mathbb H
&=
\left\{
\mathbf v\in\mathbf L^2(\mathbb D):
\operatorname{div}\mathbf v=0
\text{ in }\mathbb D,\quad
\mathbf v\cdot\mathbf n=0
\text{ on }\partial\mathbb D
\right\}.
\end{align}
Equivalently, $\mathbb H$ is the closure of $\mathcal J$ in
$\mathbf L^2(\mathbb D)$. The spaces $\mathcal J$ and $\mathcal G$ are used as smooth
test spaces for the momentum and induction equations, respectively; in particular,
$\mathcal G$ is dense in $\mathcal K$.

\subsection{Embedding, Poincar\'e, and compactness results}
\label{subsec-embeddings}

We first recall the compact embedding for curl-div spaces on Lipschitz polyhedra
\cite{AmroucheBernardiDaugeGirault1998,GiraultRaviart1986}.

\begin{lemma}\label{Embedding lemma}
There exists $\delta_1>0$ such that
\[
\mathcal V(\mathbb D)
:=
\mathcal W\cap\mathbf H_0(\operatorname{div};\mathbb D)
\hookrightarrow\hookrightarrow
\mathbf L^{3+\delta_1}(\mathbb D),
\]
where
\[
\|\mathbf v\|_{\mathcal V(\mathbb D)}
=
\left(
\|\nabla\times\mathbf v\|_{\mathbf L^2}^2
+
\|\operatorname{div}\mathbf v\|_{\mathbf L^2}^2
\right)^{1/2}.
\]
\end{lemma}
\noindent
Since $\mathcal K\subset\mathcal V(\mathbb D)$ and
$\operatorname{div}\mathbf v=0$ for $\mathbf v\in\mathcal K$, this yields
\[
\|\mathbf v\|_{\mathbf L^{3+\delta_1}}
\le
C\|\nabla\times\mathbf v\|_{\mathbf L^2},
\qquad
\mathbf v\in\mathcal K,
\]
which provides the additional integrability needed to control the Lorentz force.
\\ \noindent
We shall also use the following Poincar\'e and Sobolev estimates on bounded Lipschitz
polyhedra \cite{Monk2003}:
\begin{align}
\label{eq-poincare-velocity}
\|\mathbf v\|_{\mathbf L^{r}}
&\le
c_1\|\nabla\mathbf v\|_{\mathbf L^2},
&&
1\le r\le6,
&&
\forall\,\mathbf v\in\mathcal X_0,
\\
\label{eq-poincare-curl}
\|\mathbf v\|_{\mathbf L^2}
&\le
c_2\|\nabla\times\mathbf v\|_{\mathbf L^2},
&&
\forall\,\mathbf v\in\mathcal K,
\\
\label{eq-sobolev-Linfty}
\|\mathbf v\|_{\mathbf L^\infty}
&\le
c_3\|\mathbf v\|_{1+s,2},
&&
s>\frac12,
\\
\label{eq-sobolev-W13}
\|\mathbf v\|_{1,3}
&\le
c_4\|\mathbf v\|_{1+s,2},
\qquad
\|\nabla\times\mathbf v\|_{\mathbf L^3}
\le
c_5\|\nabla\times\mathbf v\|_{s,2}.
\end{align}
\noindent
For later use, we recall the classical Aubin-Lions compactness lemma.

\begin{lemma}[Aubin-Lions]\label{Aubin Lions}
Let $\mathbb B_0$ and $\mathbb B_1$ be reflexive Banach spaces and $\mathbb B$ a Banach
space such that
\[
\mathbb B_0
\hookrightarrow\hookrightarrow
\mathbb B
\hookrightarrow
\mathbb B_1.
\]
For $1<p_0,p_1<\infty$,
\[
\left\{
w\in L^{p_0}(0,T;\mathbb B_0):
\partial_t w\in L^{p_1}(0,T;\mathbb B_1)
\right\}
\hookrightarrow\hookrightarrow
L^{p_0}(0,T;\mathbb B).
\]
\end{lemma}
\noindent
For the velocity, however, no suitable uniform bound on
$\partial_t\u_{h,\Delta t}$ is available. We therefore use Simon's translation criterion
\cite[Theorem~5]{Simon1987}, replacing the time-derivative estimate by a time-translate bound.

\begin{lemma}[Simon]\label{Simon}
Let
\[
\mathbb B_0\hookrightarrow\hookrightarrow\mathbb B\hookrightarrow\mathbb B_1,
\]
and let $\mathcal F$ be bounded in $L^q(0,T;\mathbb B_0)$,
$1\le q<\infty$. If
\[
\sup_{w\in\mathcal F}
\|w(\cdot+\tau)-w\|_{L^q(0,T-\tau;\mathbb B_1)}
\longrightarrow0
\qquad\text{as }\tau\to0^+,
\]
then $\mathcal F$ is relatively compact in $L^q(0,T;\mathbb B)$.
\end{lemma}

\subsection{The nonlinear and trilinear forms}\label{subsec-forms}

For $\u,\mathbf w,\mathbf v\in\mathcal X$ and
$\theta,\phi\in\mathcal Y$, define
\begin{align}
\label{eq-form-a0}
a_0(\u)
&=
\nu|\nabla\u|^{p-2}\nabla\u,
\qquad
(a_0(\u),\nabla\mathbf v)
=
\nu\int_{\mathbb D}
|\nabla\u|^{p-2}\nabla\u:\nabla\mathbf v\,dx,
\\
\label{eq-form-a1}
a_1(\mathbf w,\u,\mathbf v)
&=
\frac12
\left[
\int_{\mathbb D}
(\mathbf w\cdot\nabla)\u\cdot\mathbf v\,dx
-
\int_{\mathbb D}
(\mathbf w\cdot\nabla)\mathbf v\cdot\u\,dx
\right],
\\
\label{eq-form-a2}
a_2(\mathbf w,\theta,\phi)
&=
\frac12
\left[
\int_{\mathbb D}
(\mathbf w\cdot\nabla\theta)\phi\,dx
-
\int_{\mathbb D}
(\mathbf w\cdot\nabla\phi)\theta\,dx
\right].
\end{align}
The skew-symmetric definitions imply
\begin{equation}\label{eq-skew}
a_1(\mathbf w,\u,\u)=0,
\qquad
a_2(\mathbf w,\theta,\theta)=0,
\qquad
a_1(\mathbf w,\u,\mathbf v)
=
-a_1(\mathbf w,\mathbf v,\u).
\end{equation}
These identities hold without requiring $\mathbf w$ to be divergence-free.
\\ \noindent
The Coriolis contribution is a bounded zeroth-order term:
\[
\left|
(\boldsymbol{\omega}_{r}\times\u,\mathbf v)
\right|
\le
|\boldsymbol{\omega}_{r}|
\|\u\|_{\mathbf L^2}
\|\mathbf v\|_{\mathbf L^2}.
\]
By the scalar triple-product identity,
\[
(\boldsymbol{\omega}_{r}\times\u,\mathbf v)
=
-(\boldsymbol{\omega}_{r}\times\mathbf v,\u),
\]
and hence
\begin{equation}\label{coriolis skew property}
(\boldsymbol{\omega}_{r}\times\u,\u)=0.
\end{equation}
Thus the Coriolis force is energy neutral and requires no additional regularity of the
velocity.
\\ \noindent
For $p\ge2$, the nonlinear diffusion operator is strongly monotone
\cite{Lindqvist2006}:
\begin{equation}\label{eq-strong-monotonicity}
\big(
a_0(\u)-a_0(\mathbf w),
\nabla\u-\nabla\mathbf w
\big)
\ge
\nu\,2^{2-p}
\|\nabla\u-\nabla\mathbf w\|_{\mathbf L^p}^{p},
\qquad
\forall\,\u,\mathbf w\in\mathcal X.
\end{equation}
Notice that, since $a_0$ is nonlinear,
\[
a_0(\u)-a_0(\mathbf w)
\neq
a_0(\u-\mathbf w)
\]
in general.
\\ \noindent
Finally, the electromagnetic terms satisfy the triple-product identity
\begin{equation}\label{eq-triple-product}
\big(
\mathbf b\times(\nabla\times\mathbf c),
\mathbf v
\big)
=
\big(
\mathbf v\times\mathbf b,
\nabla\times\mathbf c
\big).
\end{equation}
Consequently, the Lorentz force and induction coupling cancel exactly in the continuous
energy balance, and the same cancellation is preserved by the fully discrete scheme.

\subsection{Weak formulation}\label{subsec-weakform}

We assume
\[
\mathbf f\in L^2(0,T;\mathbf L^2(\mathbb D)),
\qquad
\mathbf g\in L^2(0,T;\mathbb H),
\qquad
\psi\in L^2(0,T;L^2(\mathbb D)),
\]
and
\[
\u^0\in\mathbb H,
\qquad
\B^0\in\mathbb H,
\qquad
\theta^0\in L^2(\mathbb D).
\]

\begin{definition}[Weak solution]\label{def-weak-solution}
A triple $(\u,\B,\theta)$ is called a weak solution of
\eqref{1st equation on MHD-BSn}-\eqref{3rd equation on MHD-BSn} if
\[
\u\in
L^\infty(0,T;\mathbb H)
\cap
L^p(0,T;\mathcal X_0),
\]
\[
\B\in
L^\infty(0,T;\mathbb H)
\cap
L^2(0,T;\mathcal K),
\qquad
\theta\in
L^\infty(0,T;L^2(\mathbb D))
\cap
L^2(0,T;\mathcal Y_0),
\]
the initial conditions
\[
\u(0)=\u^0,\qquad
\B(0)=\B^0,\qquad
\theta(0)=\theta^0
\]
are attained, and for all
\[
(\mathbf v,\mathbf C,\phi)
\in
\mathcal X_0\times\mathcal K\times\mathcal Y_0,
\]
the identities
\begin{align}
\label{1st weak form}
(\u_t,\mathbf v)
&+
(a_0(\u),\nabla\mathbf v)
+
a_1(\u,\u,\mathbf v)
+
2(\boldsymbol{\omega}_{r}\times\u,\mathbf v)+
(\B\times(\nabla\times\B),\mathbf v)
=
(\theta e_3,\mathbf v)
+
(\mathbf f,\mathbf v),
\\
\label{2nd weak equation}
(\B_t,\mathbf C)
&+
\kappa(\nabla\times\B,\nabla\times\mathbf C)
-
(\u\times\B,\nabla\times\mathbf C)
=
(\mathbf g,\mathbf C),
\\
\label{3rd weak equation}
(\theta_t,\phi)
&+
\sigma(\nabla\theta,\nabla\phi)
+
a_2(\u,\theta,\phi)
=
(\u\cdot e_3,\phi)
+
(\psi,\phi)
\end{align}
hold in the sense of distributions in time. Here the brackets denote the appropriate
duality pairings.
\end{definition}

By \eqref{coriolis skew property}, the additional rotational term is well defined under
the above regularity and vanishes when the momentum equation is tested with $\u$.
Hence it does not alter the basic energy structure of the weak formulation.


\vspace{0.3cm}

\section{The Fully Discrete Scheme}\label{sec-fem}

We discretise the weak formulation
\eqref{1st weak form}-\eqref{3rd weak equation}
by conforming finite elements in space and a semi-implicit backward Euler method in time.
We first introduce the discrete spaces and the auxiliary tools used in the analysis,
including the Hodge mapping, the discrete inf-sup condition, and inverse estimates.

\subsection{Triangulation and polynomial spaces}\label{subsec-mesh}

Let $\mathcal T_h$ be a regular, quasi-uniform triangulation of $\mathbb D$ into closed
tetrahedra, and set
\[
h:=\max_{K\in\mathcal T_h}\operatorname{diam}(K).
\]
For $k\ge1$, let $P_k(K)$ denote the space of polynomials of total degree at most $k$
on $K$, and let $\widetilde P_k(K)$ denote the space of homogeneous polynomials of
degree exactly $k$. Their vector-valued counterparts are denoted by
$\mathbf P_k(K)$ and $\widetilde{\mathbf P}_k(K)$, respectively.

\subsection{Discrete velocity, pressure, and temperature spaces}
\label{subsec-XQY}

Let $(\mathcal X_h,Q_h)$ be a conforming finite element pair satisfying the discrete
inf-sup condition of Lemma~\ref{lemma:inf-sup}, and let
$\mathcal Y_h\subset\mathcal Y_0$ be a conforming finite element space for the temperature.
We assume
\begin{align}\label{eq-spaces-XQY}
\mathcal X_h
&\subset
\left\{
\mathbf v_h\in\mathcal X:
\mathbf v_h|_K\in\mathbf P_k(K),
\quad \forall K\in\mathcal T_h
\right\},
\notag\\
Q_h
&\subset
\left\{
q_h\in Q:
q_h|_K\in P_{k-1}(K),
\quad \forall K\in\mathcal T_h
\right\},
\notag\\
\mathcal Y_h
&\subset
\left\{
\phi_h\in\mathcal Y_0:
\phi_h|_K\in P_k(K),
\quad \forall K\in\mathcal T_h
\right\}.
\end{align}
The associated discretely divergence-free velocity space is
\begin{equation}\label{eq-space-X0h}
\mathcal X_{0h}
=
\left\{
\mathbf v_h\in\mathcal X_h:
(\operatorname{div}\mathbf v_h,q_h)=0
\quad\forall q_h\in Q_h
\right\}.
\end{equation}
Since the velocity is sought directly in $\mathcal X_{0h}$, the pressure does not appear
explicitly in the discrete scheme.
\\ \noindent
For the error analysis in Section~\ref{sec-error}, we additionally assume that the velocity
pair is pointwise divergence-free:
\begin{equation}\label{eq-exactly-divfree}
\operatorname{div}\mathbf v_h=0
\quad\text{in }\mathbb D
\qquad
\forall\,\mathbf v_h\in\mathcal X_{0h},
\qquad\text{equivalently}\qquad
\mathcal X_{0h}\subset\mathcal X_0.
\end{equation}
This property is satisfied, for example, by suitable Scott-Vogelius pairs on
barycentrically refined tetrahedral meshes
\cite{Zhang2005,GuzmanNeilan2014}.
Condition~\eqref{eq-exactly-divfree} is required only for the error analysis; the
well-posedness, stability, and convergence results use only the discrete inf-sup condition.

\subsection{Discrete space for the magnetic induction}\label{subsec-Wh}

For the magnetic field, we use N\'ed\'elec edge elements of the first family. On each
$K\in\mathcal T_h$, define
\begin{equation}\label{eq-nedelec-local}
\mathcal N_k(K)
=
\mathbf P_{k-1}(K)
\oplus
\left\{
\mathbf q\in\widetilde{\mathbf P}_k(K):
\mathbf q(x)\cdot x=0
\quad\forall x\in K
\right\}.
\end{equation}
The corresponding global space is
\begin{equation}\label{eq-nedelec-global}
\mathcal W_h
=
\left\{
\mathbf A_h\in\mathbf H(\operatorname{curl};\mathbb D):
\mathbf A_h|_K\in\mathcal N_k(K)
\quad\forall K\in\mathcal T_h
\right\}.
\end{equation}
The space $\mathcal W_h$ is $\mathbf H(\operatorname{curl})$-conforming and is therefore
well suited to the induction equation. On nonconvex polyhedra, the magnetic field need
not belong to $\mathbf H^1(\mathbb D)$
\cite[Remark~3.3.1]{GerbeauLeBrisLelievre2006}, which motivates the use of edge elements
rather than $\mathbf H^1$-conforming Lagrange elements; see
\cite[Section~3.4.2]{GiraultRaviart1986,Monk2003}.
No tangential boundary condition is imposed on $\mathcal W_h$, consistently with the
continuous magnetic space $\mathcal W$.
\\ \noindent
Let
\begin{equation}\label{eq-space-Sh}
S_h
=
\left\{
s_h\in H^1(\mathbb D)\cap L_0^2(\mathbb D):
s_h|_K\in P_k(K)
\quad\forall K\in\mathcal T_h
\right\}.
\end{equation}
We define the discrete solenoidal magnetic space by
\begin{equation}\label{eq-space-W0h}
\mathcal W_{0h}
=
\left\{
\mathbf c_h\in\mathcal W_h:
(\mathbf c_h,\nabla s_h)=0
\quad\forall s_h\in S_h
\right\}.
\end{equation}
Since $\nabla S_h\subset\mathcal W_h$, $\mathcal W_{0h}$ is the
$\mathbf L^2$-orthogonal complement of $\nabla S_h$ in $\mathcal W_h$. Thus
\[
(\mathbf c_h,\nabla s_h)=0
\qquad\forall s_h\in S_h
\]
is the discrete counterpart of
\[
\operatorname{div}\mathbf c=0,
\qquad
\mathbf c\cdot\mathbf n=0,
\]
that is, of membership in the continuous space $\mathcal K$.

\subsection{The Hodge mapping and discrete compactness}\label{subsec-hodge}

The connection between $\mathcal W_{0h}$ and the continuous solenoidal space
$\mathcal K$ is provided by the Hodge mapping \cite{Hiptmair2002}
\begin{equation}\label{eq-hodge}
\mathcal Z:
\mathcal W_{0h}\longrightarrow\mathcal K,
\qquad
\nabla\times\mathcal Z(\mathbf b_h)
=
\nabla\times\mathbf b_h,
\quad
\forall\,\mathbf b_h\in\mathcal W_{0h}.
\end{equation}
It satisfies
\begin{equation}\label{ine-zbh}
\|\mathbf b_h-\mathcal Z(\mathbf b_h)\|_{\mathbf L^2(\mathbb D)}
\le
C h^{\frac12+\varrho}
\|\nabla\times\mathbf b_h\|_{\mathbf L^2(\mathbb D)},
\qquad
\forall\,\mathbf b_h\in\mathcal W_{0h},
\end{equation}
for some $\varrho=\varrho(\mathbb D)>0$.
Together with Lemma~\ref{Embedding lemma},
\eqref{eq-hodge}-\eqref{ine-zbh} yield the following discrete compactness property.

\begin{lemma}\label{Lemma 1}
Let $h\to0$ and let
$\{\mathbf C_h\}\subset\mathcal W_{0h}$ be uniformly bounded in $\mathcal W$.
Then there exist a subsequence, still denoted by $\{\mathbf C_h\}$, and
$\mathbf C\in\mathcal K$ such that
\[
\mathbf C_h\rightharpoonup\mathbf C
\quad\text{in }\mathcal W,
\qquad
\mathbf C_h\to\mathbf C
\quad\text{in }\mathbf L^2(\mathbb D).
\]
\end{lemma}

\begin{proof}
Since $\{\mathbf C_h\}$ is bounded in $\mathcal W$, there exists a subsequence such that
\[
\mathbf C_h\rightharpoonup\mathbf C
\quad\text{in }\mathcal W
\]
for some $\mathbf C\in\mathcal W$. By \eqref{eq-hodge},
$\mathcal Z(\mathbf C_h)\in\mathcal K$ and
\[
\|\mathcal Z(\mathbf C_h)\|_{\mathcal V(\mathbb D)}
=
\|\nabla\times\mathbf C_h\|_{\mathbf L^2}
\le C.
\]
Hence Lemma~\ref{Embedding lemma} gives, after extracting a further subsequence,
\[
\mathcal Z(\mathbf C_h)\to\mathbf C^\ast
\qquad\text{strongly in }\mathbf L^2(\mathbb D)
\]
for some $\mathbf C^\ast\in\mathcal K$. Moreover, by \eqref{ine-zbh},
\[
\|\mathbf C_h-\mathcal Z(\mathbf C_h)\|_{\mathbf L^2}
\le
C h^{\frac12+\varrho}
\|\nabla\times\mathbf C_h\|_{\mathbf L^2}
\longrightarrow0.
\]
Therefore $\mathbf C_h\to\mathbf C^\ast$ strongly in $\mathbf L^2(\mathbb D)$.
Since the weak $\mathcal W$-limit is unique, $\mathbf C^\ast=\mathbf C$, and hence
$\mathbf C\in\mathcal K$.
\end{proof}
 \noindent
The Hodge mapping also provides the spatial integrability required for the Lorentz force.

\begin{lemma}\label{Z=curl}
For every $\mathbf b_h\in\mathcal W_{0h}$,
\[
\|\mathcal Z(\mathbf b_h)\|_{\mathbf L^{3+\delta_1}(\mathbb D)}
\le
C\|\nabla\times\mathbf b_h\|_{\mathbf L^2(\mathbb D)},
\]
where $\delta_1>0$ is given by Lemma~\ref{Embedding lemma} and $C$ is independent of $h$.
\end{lemma}

\begin{proof}
Since
$\mathcal Z(\mathbf b_h)\in\mathcal K\subset\mathcal V(\mathbb D)$ and
$\operatorname{div}\mathcal Z(\mathbf b_h)=0$, Lemma~\ref{Embedding lemma} gives
\[
\|\mathcal Z(\mathbf b_h)\|_{\mathbf L^{3+\delta_1}}
\le
C\|\mathcal Z(\mathbf b_h)\|_{\mathcal V(\mathbb D)}
=
C\|\nabla\times\mathcal Z(\mathbf b_h)\|_{\mathbf L^2}.
\]
Using \eqref{eq-hodge}, the desired estimate follows.
\end{proof}

\subsection{Discrete inf-sup condition and inverse estimates}
\label{subsec-discrete-tools}

\begin{lemma}[Discrete inf-sup condition]\label{lemma:inf-sup}
Let $(\mathcal X_h,Q_h)$ satisfy the Babuska-Brezzi condition.
Then there exists $\beta_1>0$, independent of $h$, such that
\[
\inf_{0\neq q_h\in Q_h}
\sup_{0\neq\mathbf v_h\in\mathcal X_h}
\frac{(\operatorname{div}\mathbf v_h,q_h)}
{\|\mathbf v_h\|_{1,2}\|q_h\|_{L^2}}
\ge
\beta_1.
\]
\end{lemma}
\noindent
We shall also use the standard inverse estimate on quasi-uniform meshes:
\begin{equation}\label{eq-inverse estimate}
\|\mathbf v_h\|_{\mu_2,q_2}
\le
C_{\mathrm{inv}}\,
h^{\mu_1-\mu_2+3(1/q_2-1/q_1)}
\|\mathbf v_h\|_{\mu_1,q_1},
\qquad
\forall\,\mathbf v_h\in\mathcal X_h,
\end{equation}
where
\[
0\le\mu_1\le\mu_2\le1,
\qquad
1\le q_1\le q_2\le\infty,
\]
and $C_{\mathrm{inv}}$ is independent of $h$. Analogous inverse estimates are used
for functions in $\mathcal W_h$ and $\mathcal Y_h$ whenever required.

\subsection{The fully discrete scheme}\label{subsec-scheme}

\paragraph{Time grid, initial data, and source terms.}
Let $N\ge1$, set
\[
\Delta t=\frac{T}{N},
\qquad
t_m=m\Delta t,
\qquad
m=0,\ldots,N,
\]
and define
\[
d_t z^m
=
\frac{z^m-z^{m-1}}{\Delta t}.
\]
For the convergence analysis, it is sufficient to choose
\[
\u_h^0\in\mathcal X_{0h},
\qquad
\B_h^0\in\mathcal W_{0h},
\qquad
\theta_h^0\in\mathcal Y_h
\]
such that
\begin{equation}\label{eq-initial-convergence}
\u_h^0\to\u^0
\quad\text{in }\mathbf L^2(\mathbb D),
\qquad
\B_h^0\to\B^0
\quad\text{in }\mathbf L^2(\mathbb D),
\qquad
\theta_h^0\to\theta^0
\quad\text{in }L^2(\mathbb D).
\end{equation}
No additional spatial regularity of the initial data is required for the convergence result.
\\ \noindent
For the error analysis, we specialise these approximations to the projections
\[
\mathcal P_h:\mathcal X_0\to\mathcal X_{0h},
\qquad
\mathcal F_h:\mathcal K\to\mathcal W_{0h},
\qquad
\mathcal R_h:\mathcal Y_0\to\mathcal Y_h,
\]
defined in Section~\ref{sec-error}, and set
\begin{equation}\label{eq-initial-discrete}
\u_h^0=\mathcal P_h\u^0,
\qquad
\B_h^0=\mathcal F_h\B^0,
\qquad
\theta_h^0=\mathcal R_h\theta^0.
\end{equation}
The corresponding regularity assumptions are imposed only in
Section~\ref{sec-error}. In particular, with $l=\min\{k,s\}$,
\begin{equation}\label{eq-Ph-W1p}
\|\nabla(\mathbf v-\mathcal P_h\mathbf v)\|_{\mathbf L^p}
\le
C_e h^l
\|\mathbf v\|_{1+l,p},
\qquad
\mathbf v\in\mathbf W^{1+l,p}\cap\mathcal X_0.
\end{equation}
\noindent
The source terms are approximated by their cell averages,
\begin{equation}\label{eq-data-averages}
\mathbf f^m
=
\frac1{\Delta t}
\int_{t_{m-1}}^{t_m}\mathbf f(t)\,dt,
\qquad
\mathbf g^m
=
\frac1{\Delta t}
\int_{t_{m-1}}^{t_m}\mathbf g(t)\,dt,
\qquad
\psi^m
=
\frac1{\Delta t}
\int_{t_{m-1}}^{t_m}\psi(t)\,dt.
\end{equation}
This choice is suitable for data that are merely square-integrable in time. By Jensen's
inequality,
\begin{align}\label{eq-data-bound}
\Delta t\sum_{m=1}^{N}
\left(
\|\mathbf f^m\|_{\mathbf L^2}^{2}
+
\|\mathbf g^m\|_{\mathbf L^2}^{2}
+
\|\psi^m\|_{L^2}^{2}
\right)
\le&
\|\mathbf f\|_{L^2(0,T;\mathbf L^2)}^{2}
+
\|\mathbf g\|_{L^2(0,T;\mathbf L^2)}^{2} +
\|\psi\|_{L^2(0,T;L^2)}^{2}.
\end{align}

\paragraph{The scheme.}
Given $(\u_h^0,\B_h^0,\theta_h^0)
\in
\mathcal X_{0h}\times\mathcal W_{0h}\times\mathcal Y_h,$ for each $m=1,\ldots,N$ we seek
\[
(\u_h^m,\B_h^m,\theta_h^m)
\in
\mathcal X_{0h}\times\mathcal W_{0h}\times\mathcal Y_h
\]
such that, for all $(\mathbf v_h,\mathbf w_h,\phi_h)
\in
\mathcal X_{0h}\times\mathcal W_{0h}\times\mathcal Y_h,$
\begin{align}
\label{eq-FEM-1}
(d_t\u_h^m,\mathbf v_h)
&+
a_1(\u_h^{m-1},\u_h^m,\mathbf v_h)
+
\nu
\big(
|\nabla\u_h^m|^{p-2}\nabla\u_h^m,
\nabla\mathbf v_h
\big)
\nonumber\\
&\quad+
2\big(
\boldsymbol{\omega}_{r}\times\u_h^m,
\mathbf v_h
\big)
+
\big(
\B_h^{m-1}\times(\nabla\times\B_h^m),
\mathbf v_h
\big) =
(\mathbf f^m,\mathbf v_h)
+
(\theta_h^{m-1}e_3,\mathbf v_h),
\\
\label{eq-FEM-2}
(d_t\B_h^m,\mathbf w_h)
&+
\kappa
(\nabla\times\B_h^m,\nabla\times\mathbf w_h)
-
(\u_h^m\times\B_h^{m-1},
\nabla\times\mathbf w_h)
=
(\mathbf g^m,\mathbf w_h),
\\
\label{eq-FEM-3}
(d_t\theta_h^m,\phi_h)
&+
\sigma(\nabla\theta_h^m,\nabla\phi_h)
+
a_2(\u_h^{m-1},\theta_h^m,\phi_h)
=
(\u_h^m\cdot e_3,\phi_h)
+
(\psi^m,\phi_h).
\end{align}

\begin{remark}[Structure of the scheme]\label{rem-scheme-design}
The scheme is semi-implicit: the transport velocities and one magnetic factor are evaluated
at the previous time level, while the dissipative terms, the remaining coupled variables,
and the Coriolis term are treated implicitly. Since $\theta_h^{m-1}$ is known, the coupled
$(\u_h^m,\B_h^m)$ subsystem can be solved first, followed by \eqref{eq-FEM-3} for
$\theta_h^m$.
\\ \noindent
The chosen time levels preserve the fundamental energy structure:
\[
a_1(\u_h^{m-1},\u_h^m,\u_h^m)=0,
\qquad
a_2(\u_h^{m-1},\theta_h^m,\theta_h^m)=0,
\]
while the Coriolis contribution satisfies
\begin{equation}\label{eq-discrete-coriolis}
2\big(
\boldsymbol{\omega}_{r}\times\u_h^m,
\u_h^m
\big)
=0.
\end{equation}
Moreover, by the scalar triple-product identity \eqref{eq-triple-product},
\[
\big(
\B_h^{m-1}\times(\nabla\times\B_h^m),
\u_h^m
\big)
=
\big(
\u_h^m\times\B_h^{m-1},
\nabla\times\B_h^m
\big).
\]
Thus the electromagnetic coupling terms cancel exactly and the Coriolis force is
energy neutral in the discrete energy estimate. The buoyancy terms occur at different
time levels and therefore do not cancel; the scheme is energy stable rather than
energy preserving.
\end{remark}

\begin{remark}[Discrete solenoidality of $\B_h^m$]
\label{rem-divfree-B}
Since $\B_h^m$ is sought directly in $\mathcal W_{0h}$,
\[
(\B_h^m,\nabla s_h)=0
\qquad
\forall\,s_h\in S_h
\]
holds by construction for every $m$, providing the discrete counterpart of
$\operatorname{div}\B=0$ together with $\mathbf n\cdot\B=0$.
\\ \noindent
This condition should not be obtained by choosing
$\mathbf w_h=\nabla s_h$ in \eqref{eq-FEM-2}, since
$\nabla s_h\notin\mathcal W_{0h}$ unless $\nabla s_h=0$.
As solenoidality is built into the trial space, no additional Lagrange multiplier is
required \cite{Prohl2008}.
\end{remark}

\medskip

\section{Well-posedness and stability}\label{sec-existence-stability}

The $p$-Laplacian makes the discrete momentum equation nonlinear, so the full scheme cannot
be treated directly by Lax-Milgram. Its triangular structure, however, is particularly useful:
since the buoyancy term in \eqref{eq-FEM-1} is evaluated at $\theta_h^{m-1}$, the pair
$(\u_h^m,\B_h^m)$ is determined first from \eqref{eq-FEM-1}-\eqref{eq-FEM-2}, after which
\eqref{eq-FEM-3} is linear in $\theta_h^m$. We exploit this structure to prove existence,
uniqueness, and energy stability. The key ingredients are the skew-symmetry
\eqref{eq-skew}, the electromagnetic cancellation \eqref{eq-triple-product}, and the strong
monotonicity \eqref{eq-strong-monotonicity}. No estimate below requires any coupling between
$\Delta t$ and $h$.

\subsection{Existence of a discrete solution}\label{subsec-existence}

We use the following standard finite-dimensional consequence of Brouwer's fixed point theorem.

\begin{lemma}[{\cite[Ch.~II, Lemma~1.4]{Temam1977}}]
\label{lemma:finite-dimensional-coercivity}
Let $V$ be a finite-dimensional Hilbert space and $G:V\to V$ be continuous. If
\[
\frac{(G(z),z)_V}{\|z\|_V}\longrightarrow+\infty
\qquad\text{as }\|z\|_V\to\infty,
\]
then there exists $z\in V$ such that $G(z)=0$.
\end{lemma}

\begin{theorem}[Existence of a discrete solution]
\label{thm-existence}
For every $m=1,\ldots,N$, $p\ge2$, and admissible $h,\Delta t>0$, the scheme
\eqref{eq-FEM-1}--\eqref{eq-FEM-3} admits at least one solution
\[
(\u_h^m,\B_h^m,\theta_h^m)
\in
\mathcal X_{0h}\times\mathcal W_{0h}\times\mathcal Y_h .
\]
In particular, no coupling condition between $\Delta t$ and $h$ is required.
\end{theorem}

\begin{proof}
Fix $m$ and assume that
$(\u_h^{m-1},\B_h^{m-1},\theta_h^{m-1})$ is known.

\smallskip
\noindent
\emph{Step 1: Velocity--magnetic subsystem.}
Set
\[
\mathbb V_h:=\mathcal X_{0h}\times\mathcal W_{0h},
\qquad
\|\mathbb U_h\|_{0,h}^2
:=
\|\u_h\|_{\mathbf L^2}^2
+
\|\B_h\|_{\mathbf L^2}^2,
\]
where
\[
\mathbb U_h=(\u_h,\B_h)\in\mathbb V_h,
\qquad
\eta_h=(\v_h,\C_h)\in\mathbb V_h.
\]
Define $F^m:\mathbb V_h\to\mathbb V_h'$ by
\begin{align}
\label{eq-def-Fm}
\langle F^m(\mathbb U_h),\eta_h\rangle
&=
\frac1{\Delta t}(\u_h,\v_h)
+a_1(\u_h^{m-1},\u_h,\v_h)
+\nu
\bigl(
|\nabla\u_h|^{p-2}\nabla\u_h,
\nabla\v_h
\bigr)
\nonumber\\
&\quad
+
2\bigl(
\boldsymbol{\omega}_{r}\times\u_h,
\v_h
\bigr)
+
\bigl(
\B_h^{m-1}\times(\nabla\times\B_h),
\v_h
\bigr)
\nonumber\\
&\quad
+
\frac1{\Delta t}(\B_h,\C_h)
+
\kappa
(\nabla\times\B_h,\nabla\times\C_h)
-
(\u_h\times\B_h^{m-1},
\nabla\times\C_h),
\end{align}
and $L^m\in\mathbb V_h'$ by
\begin{align}
\label{eq-def-Lm}
\langle L^m,\eta_h\rangle
&=
\frac1{\Delta t}(\u_h^{m-1},\v_h)
+
(\mathbf f^m,\v_h)
+
(\theta_h^{m-1}e_3,\v_h)+
\frac1{\Delta t}(\B_h^{m-1},\C_h)
+
(\mathbf g^m,\C_h).
\end{align}
Thus \eqref{eq-FEM-1}--\eqref{eq-FEM-2} is equivalent to finding
$\mathbb U_h^m\in\mathbb V_h$ such that
\[
F^m(\mathbb U_h^m)=L^m.
\]
\\ \noindent
After identifying $\mathbb V_h'$ with $\mathbb V_h$ through the Riesz map,
define $G^m:=F^m-L^m$. Since $\mathbb V_h$ is finite-dimensional,
$\xi\mapsto|\xi|^{p-2}\xi$ is continuous for $p\ge2$, and all remaining
terms are multilinear, $G^m$ is continuous.

\smallskip
\noindent
\emph{Step 2: Coercivity.}
For arbitrary $\mathbb U_h=(\u_h,\B_h)\in\mathbb V_h$, take
$\eta_h=\mathbb U_h$. By \eqref{eq-skew},
\[
a_1(\u_h^{m-1},\u_h,\u_h)=0,
\]
while \eqref{eq-triple-product} gives
\[
\bigl(
\B_h^{m-1}\times(\nabla\times\B_h),
\u_h
\bigr)
-
\bigl(
\u_h\times\B_h^{m-1},
\nabla\times\B_h
\bigr)
=0.
\]
Moreover, by the skew-symmetry of the Coriolis term,
\[
2\bigl(
\boldsymbol{\omega}_{r}\times\u_h,
\u_h
\bigr)=0.
\]
Consequently,
\begin{equation}\label{eq-coercivity-Fm}
\langle F^m(\mathbb U_h),\mathbb U_h\rangle
=
\frac1{\Delta t}\|\mathbb U_h\|_{0,h}^2
+
\nu\|\nabla\u_h\|_{\mathbf L^p}^p
+
\kappa\|\nabla\times\B_h\|_{\mathbf L^2}^2.
\end{equation}
Since $L^m$ is a bounded linear functional on the finite-dimensional space
$\mathbb V_h$, there exists $C_m>0$ such that
\[
|\langle L^m,\mathbb U_h\rangle|
\le
C_m\|\mathbb U_h\|_{0,h}.
\]
Hence
\[
\frac{
\langle G^m(\mathbb U_h),\mathbb U_h\rangle
}{
\|\mathbb U_h\|_{0,h}
}
\ge
\frac{\|\mathbb U_h\|_{0,h}}{\Delta t}
-
C_m
\longrightarrow+\infty
\]
as $\|\mathbb U_h\|_{0,h}\to\infty$.
Lemma~\ref{lemma:finite-dimensional-coercivity} therefore yields
$\mathbb U_h^m=(\u_h^m,\B_h^m)\in\mathbb V_h$ satisfying
\eqref{eq-FEM-1}--\eqref{eq-FEM-2}.

Notice that no absorption of an $L^2$-velocity term into the $p$-Laplacian
is required. Consequently, the argument also remains valid for $p=2$,
unlike an estimate of the form
\[
\|\u_h\|_{\mathbf L^2}^2
\le
\delta\|\nabla\u_h\|_{\mathbf L^p}^p
+
C_\delta,
\]
which is available with arbitrary $\delta>0$ only when $p>2$.

\smallskip
\noindent
\emph{Step 3: Temperature equation.}
With $\u_h^m$ known, find $\theta_h^m\in\mathcal Y_h$ such that
\[
\mathcal B_m(\theta_h^m,\phi_h)
=
\frac1{\Delta t}
(\theta_h^{m-1},\phi_h)
+
(\u_h^m\cdot e_3,\phi_h)
+
(\psi^m,\phi_h),
\qquad
\forall\,\phi_h\in\mathcal Y_h,
\]
where
\[
\mathcal B_m(\theta,\phi)
=
\frac1{\Delta t}(\theta,\phi)
+
\sigma(\nabla\theta,\nabla\phi)
+
a_2(\u_h^{m-1},\theta,\phi).
\]
The bilinear form $\mathcal B_m$ is continuous on
$\mathcal Y_h\times\mathcal Y_h$, and by \eqref{eq-skew},
\[
\mathcal B_m(\phi_h,\phi_h)
=
\frac1{\Delta t}\|\phi_h\|_{L^2}^2
+
\sigma\|\nabla\phi_h\|_{L^2}^2.
\]
Hence $\mathcal B_m$ is coercive. The Lax--Milgram theorem therefore gives
a unique $\theta_h^m\in\mathcal Y_h$.

Starting from the prescribed initial data and repeating the argument for
$m=1,\ldots,N$ proves the result.
\end{proof}

\subsection{Uniqueness of the discrete solution}\label{subsec-uniqueness}

The triangular structure also yields uniqueness. At a fixed time level, the previous-step
velocity, magnetic field, and temperature are common data. Thus the convective term vanishes
after subtraction and testing by the velocity difference, the electromagnetic terms cancel,
and the nonlinear diffusion is controlled by strong monotonicity. Once the velocity and
magnetic fields coincide, uniqueness of the temperature follows separately.

\begin{lemma}[Uniqueness of the discrete solution]\label{lem-uniqueness}
Let $m\in\{1,\ldots,N\}$ and $p\ge2$. For prescribed
$(\u_h^{m-1},\B_h^{m-1},\theta_h^{m-1})$ and
$\mathbf f^m,\mathbf g^m,\psi^m$, the scheme
\eqref{eq-FEM-1}-\eqref{eq-FEM-3} has at most one solution in
$\mathcal X_{0h}\times\mathcal W_{0h}\times\mathcal Y_h$ for every $\Delta t>0$.
\end{lemma}

\begin{proof}
Let
$(\u_{ih}^m,\B_{ih}^m,\theta_{ih}^m)$, $i=1,2$, be two solutions and define
\[
\mathbf e_{\mathbf{u}}^m:=\u_{1h}^m-\u_{2h}^m,\qquad
\mathbf e_{\mathbf{b}}^m:=\B_{1h}^m-\B_{2h}^m,\qquad
e_\theta^m:=\theta_{1h}^m-\theta_{2h}^m.
\]
\noindent
Subtracting the momentum equations and testing with $\mathbf e_{\mathbf{u}}^m$ gives
\begin{align}
\label{eq-uniq-mom}
\frac1{\Delta t}\|\mathbf e_{\mathbf{u}}^m\|_{\mathbf L^2}^2
&+
a_1(\u_h^{m-1},\mathbf e_{\mathbf{u}}^m,\mathbf e_{\mathbf{u}}^m)
+
\bigl(
a_0(\u_{1h}^m),\nabla\mathbf e_{\mathbf{u}}^m
\bigr)-\bigl(a_0(\u_{2h}^m),
\nabla\mathbf e_{\mathbf{u}}^m
\bigr)+2\big(
\boldsymbol{\omega}_{r}\times\mathbf e_{\mathbf{u}}^m,
\mathbf e^m_\u
\big) +
\bigl(
\B_h^{m-1}\times(\nabla\times\mathbf e_{\mathbf{b}}^m),
\mathbf e_{\mathbf{u}}^m
\bigr)
=0.
\end{align}
Similarly, subtraction of the induction equations and testing with $\mathbf e_{\mathbf{b}}^m$ yields
\begin{equation}
\label{eq-uniq-ind}
\frac1{\Delta t}\|\mathbf e_{\mathbf{b}}^m\|_{\mathbf L^2}^2
+
\kappa\|\nabla\times\mathbf e_{\mathbf{b}}^m\|_{\mathbf L^2}^2
-
(\mathbf e_{\mathbf{u}}^m\times\B_h^{m-1},
\nabla\times\mathbf e_{\mathbf{b}}^m)
=0.
\end{equation}
Using \eqref{eq-skew}, \eqref{eq-triple-product}, and
\eqref{eq-strong-monotonicity}, we obtain
\begin{align}
\label{eq-discrete-uniqueness-energy}
\frac1{\Delta t}\|\mathbf e_{\mathbf{u}}^m\|_{\mathbf L^2}^2
+
\frac1{\Delta t}\|\mathbf e_{\mathbf{b}}^m\|_{\mathbf L^2}^2
+
\nu2^{2-p}\|\nabla\mathbf e_{\mathbf{u}}^m\|_{\mathbf L^p}^p
+
\kappa\|\nabla\times\mathbf e_{\mathbf{b}}^m\|_{\mathbf L^2}^2
\le0.
\end{align}
Hence
\[
\mathbf e_{\mathbf{u}}^m=\mathbf e_{\mathbf{b}}^m=0.
\]
\noindent
Subtracting the temperature equations and testing with $e_\theta^m$ then gives
\[
\frac1{\Delta t}\| \mathbf e_\theta^m\|_{L^2}^2
+
\sigma\|\nabla \mathbf e_\theta^m\|_{L^2}^2
=0,
\]
so $\mathbf e_\theta^m=0$. Therefore the discrete solution is unique.
\end{proof}

\begin{corollary}[Well-posedness]\label{cor-wellposed}
For every admissible $h>0$ and every $\Delta t>0$, the fully discrete scheme
\eqref{eq-FEM-1}-\eqref{eq-FEM-3} determines a unique sequence
\[
\{(\u_h^m,\B_h^m,\theta_h^m)\}_{m=1}^N
\subset
\mathcal X_{0h}\times\mathcal W_{0h}\times\mathcal Y_h .
\]
\end{corollary}

\subsection{Energy stability}\label{subsec-stability}

We next derive mesh-independent a~priori estimates. Because the buoyancy terms occur at
different time levels, the argument requires only a time-step bound
$\Delta t\le\Delta t_0$ independent of $h$; in particular, there is no CFL-type coupling
between $\Delta t$ and the mesh size.

\begin{theorem}\label{bounds}
Let
\[
\{(\u_h^m,\B_h^m,\theta_h^m)\}_{m=0}^N
\subset
\mathcal X_{0h}\times\mathcal W_{0h}\times\mathcal Y_h
\]
solve \eqref{eq-FEM-1}-\eqref{eq-FEM-3}. There exists
$\Delta t_0>0$, depending only on the domain and model parameters and independent of $h$,
such that, for $\Delta t\le\Delta t_0$,
\begin{align}
\max_{0\le m\le N}
\left(
\|\u_h^m\|_{\mathbf L^2}^2
+\|\B_h^m\|_{\mathbf L^2}^2
+\|\theta_h^m\|_{L^2}^2
\right)
&\le C,
\label{eq-bounds-1}
\\
\Delta t\sum_{m=1}^N
\left(
\nu\|\nabla\u_h^m\|_{\mathbf L^p}^p
+\kappa\|\nabla\times\B_h^m\|_{\mathbf L^2}^2
+\sigma\|\nabla\theta_h^m\|_{L^2}^2
\right)
&\le C,
\label{eq-bounds-2}
\\
\sum_{m=1}^N
\left(
\|\u_h^m-\u_h^{m-1}\|_{\mathbf L^2}^2
+\|\B_h^m-\B_h^{m-1}\|_{\mathbf L^2}^2
+\|\theta_h^m-\theta_h^{m-1}\|_{L^2}^2
\right)
&\le C,
\label{eq-bounds-3}
\end{align}
where $C>0$ is independent of $h$ and $\Delta t$ and depends only on $T$, the domain,
the model parameters, the data norms, and the initial energy.
\end{theorem}

\begin{proof}
Testing \eqref{eq-FEM-1}-\eqref{eq-FEM-3} with
$2\Delta t\,\u_h^m$, $2\Delta t\,\B_h^m$, and
$2\Delta t\,\theta_h^m$, respectively, and using $2(a-b,a)=\|a\|^2-\|b\|^2+\|a-b\|^2,$ together with \eqref{eq-skew} and \eqref{eq-triple-product}, gives
\begin{align}
&A_m
+
\|\u_h^m-\u_h^{m-1}\|_{\mathbf L^2}^2
+
\|\B_h^m-\B_h^{m-1}\|_{\mathbf L^2}^2
+
\|\theta_h^m-\theta_h^{m-1}\|_{L^2}^2 +
2\Delta t
\left(
\nu\|\nabla\u_h^m\|_{\mathbf L^p}^p
+
\kappa\|\nabla\times\B_h^m\|_{\mathbf L^2}^2
+
\sigma\|\nabla\theta_h^m\|_{L^2}^2
\right)
\nonumber\\
&=
A_{m-1}
+
2\Delta t(\mathbf f^m,\u_h^m)
+
2\Delta t(\mathbf g^m,\B_h^m)
+
2\Delta t(\psi^m,\theta_h^m) +
2\Delta t(\theta_h^{m-1}e_3,\u_h^m)
+
2\Delta t(\u_h^m\cdot e_3,\theta_h^m),
\label{eq:stability-identity}
\end{align}
where
\[
A_m
:=
\|\u_h^m\|_{\mathbf L^2}^2
+
\|\B_h^m\|_{\mathbf L^2}^2
+
\|\theta_h^m\|_{L^2}^2 .
\]
\noindent
Let
\[
\mathcal F_m
:=
\|\mathbf f^m\|_{\mathbf L^2}^2
+
\|\mathbf g^m\|_{\mathbf L^2}^2
+
\|\psi^m\|_{L^2}^2.
\]
By Cauchy-Schwarz and Young's inequalities,
\[
\text{RHS of \eqref{eq:stability-identity}}
\le
A_{m-1}
+
C_0\Delta t
\bigl(A_m+A_{m-1}+\mathcal F_m\bigr).
\]
Choose
\begin{equation}\label{eq-dt0}
\Delta t_0:=\frac1{2C_0}.
\end{equation}
Then, for $\Delta t\le\Delta t_0$,
\begin{align}
A_m
&+
\|\u_h^m-\u_h^{m-1}\|_{\mathbf L^2}^2
+
\|\B_h^m-\B_h^{m-1}\|_{\mathbf L^2}^2
+
\|\theta_h^m-\theta_h^{m-1}\|_{L^2}^2
\nonumber\\
&+
\Delta t
\left(
\nu\|\nabla\u_h^m\|_{\mathbf L^p}^p
+
\kappa\|\nabla\times\B_h^m\|_{\mathbf L^2}^2
+
\sigma\|\nabla\theta_h^m\|_{L^2}^2
\right) \le
(1+C\Delta t)A_{m-1}
+
C\Delta t\,\mathcal F_m .
\label{eq:stability-step-noE}
\end{align}
Summing over $m$ and using \eqref{eq-data-bound} gives
\begin{align}
A_n
&+
\sum_{m=1}^n
\left(
\|\u_h^m-\u_h^{m-1}\|_{\mathbf L^2}^2
+
\|\B_h^m-\B_h^{m-1}\|_{\mathbf L^2}^2
+
\|\theta_h^m-\theta_h^{m-1}\|_{L^2}^2
\right)
\nonumber\\
&+
\Delta t\sum_{m=1}^n
\left(
\nu\|\nabla\u_h^m\|_{\mathbf L^p}^p
+
\kappa\|\nabla\times\B_h^m\|_{\mathbf L^2}^2
+
\sigma\|\nabla\theta_h^m\|_{L^2}^2
\right)
\nonumber\\
&\le
A_0
+
C\Delta t\sum_{m=1}^n A_{m-1}
+
C\Delta t\sum_{m=1}^n\mathcal F_m .
\label{eq:stability-summed-noE}
\end{align}
The discrete Gr\"onwall lemma
\cite[Lemma~5.1]{HeywoodRannacher1990} now gives \eqref{eq-bounds-1};
substitution into \eqref{eq:stability-summed-noE} yields
\eqref{eq-bounds-2} and \eqref{eq-bounds-3}.
\end{proof}

\subsection{Consequences of the stability estimates}\label{subsec-stability-consequences}

The $p$-energy estimate also yields an $L^2$-in-space bound for the velocity gradient.

\begin{corollary}\label{cor-L2gradient}
Under the assumptions of Theorem~\ref{bounds},
\begin{equation}\label{eq-bounds-2-L2}
\Delta t
\sum_{m=1}^N
\|\nabla\u_h^m\|_{\mathbf L^2}^2
\le C,
\end{equation}
where $C$ is independent of $h$ and $\Delta t$.
\end{corollary}

\begin{proof}
For $p=2$ the claim follows directly from \eqref{eq-bounds-2}. For $p>2$, since
$\mathbb D$ is bounded,
\[
\|\nabla\u_h^m\|_{\mathbf L^2}
\le
|\mathbb D|^{\frac12-\frac1p}
\|\nabla\u_h^m\|_{\mathbf L^p}.
\]
H\"older's inequality in the discrete time sum gives
\[
\Delta t\sum_{m=1}^N
\|\nabla\u_h^m\|_{\mathbf L^p}^2
\le
\left(
\Delta t\sum_{m=1}^N
\|\nabla\u_h^m\|_{\mathbf L^p}^p
\right)^{2/p}
T^{1-2/p},
\]
and \eqref{eq-bounds-2} yields \eqref{eq-bounds-2-L2}.
\end{proof}
\noindent
Estimates \eqref{eq-bounds-1}-\eqref{eq-bounds-3} and
\eqref{eq-bounds-2-L2} provide the compactness needed below:
\eqref{eq-bounds-1}-\eqref{eq-bounds-2} yield weak compactness,
\eqref{eq-bounds-3} controls the difference between the time interpolants, and
\eqref{eq-bounds-2-L2} supplies the $\mathbf L^2$-gradient estimate used in the
time-translate argument of Lemma~\ref{time translate}.

\medskip

\section{Convergence of the fully discrete scheme}\label{sec-convergence}

We prove that, as $h,\Delta t\to0$, a subsequence of the solutions of
\eqref{eq-FEM-1}-\eqref{eq-FEM-3} converges to a weak solution in the sense of
Definition~\ref{def-weak-solution}. Two points require special care. Since the $p$-Laplacian
does not yield a uniform bound on the discrete time derivative of the velocity, compactness is
obtained from a time-translate estimate and Simon's criterion. The nonlinear flux
$|\nabla\u_{h,\Delta t}|^{p-2}\nabla\u_{h,\Delta t}$ is then identified by Minty's method
through an energy $\limsup$ inequality. For the magnetic induction and temperature we use the
time-derivative estimates of Lemmas~\ref{time derivative} and \ref{time derivative-theta}.

\subsection{Interpolants in time}\label{subsec-interpolants}

For $z\in\{\u,\B,\theta\}$ let $z_{h,\Delta t}$ denote the piecewise linear interpolant in time
of the discrete solution $\{z^{m}_{h}\}_{m=0}^{N}$, i.e.
$$
z_{h, \Delta t}(\cdot,t)=\frac{t-t_{m-1}}{\Delta t} z_{h}^{m}(\cdot)+\frac{t_{m}-t}{\Delta t} z_{h}^{m-1}(\cdot),
\qquad t\in(t_{m-1},t_{m}],\quad m=1,\ldots,N,
$$
and let $\mathring z_{h,\Delta t}$ and $\breve z_{h,\Delta t}$ denote the piecewise constant
extensions equal to $z^{m}_{h}$ and to $z^{m-1}_{h}$ respectively on $(t_{m-1},t_{m}]$;
similarly $\mathring{\mathbf f}_{\Delta t},\mathring{\mathbf g}_{\Delta t}{,\mathring\psi_{\Delta t}}$
for the data $\{\mathbf f^{m}\},\{\mathbf g^{m}\}{,\{\psi^{m}\}}$ of
\eqref{eq-data-averages}. Throughout this section $\xi\in\mathcal{C}^{\infty}_{c}([0,T))$ is
an arbitrary scalar cut-off in time with $\xi(T)=0$, used to encode the initial condition after
integration by parts in $t$.
 
\subsection{Statement of the main result}\label{subsec-conv-statement}

 \begin{theorem}[Convergence of the fully discrete scheme]\label{Thm-Conv-scheme}
        Let $p\ge2$ and assume that $\mathring{\mathbf{f}}_{\Delta t} \rightarrow \mathbf{f} $ in $L^{2}(0,T;\mathbf{L^{2}(\mathbb{D})}) $, $\mathring{\mathbf{g}}_{\Delta t} \rightarrow \mathbf{g} $ in $L^{2}(0,T;\mathbf{L^{2}(\mathbb{D})})$ and $ \mathring{\psi}_{\Delta t}\rightarrow{\psi}$ in $L^{2}(0,T;{L^{2}(\mathbb{D})})$. Then, as both the mesh-width and the time-step size tend to zero, the sequence of fully discrete solutions $\left\{(\mathring{\u}_{h,\Delta t},\mathring{\B}_{h,\Delta t},\mathring{\theta}_{h,\Delta t})\right\}$ admits a subsequence converging to a limit $(\u,\B,\theta)$ with
\[
\u\in L^{\infty}(0,T;\mathbf L^{2})\cap L^{p}(0,T;\mathcal X_{0}),\quad
\B\in L^{\infty}(0,T;\mathbb H)\cap L^{2}(0,T;\mathcal K),\quad
\theta\in L^{\infty}(0,T;L^{2})\cap L^{2}(0,T;\mathcal Y_{0}),
\]
which satisfies the induction equation \eqref{2nd weak equation}, the temperature equation
\eqref{3rd weak equation}, and the momentum equation \eqref{1st weak form} with the
$p$-Laplacian flux $a_{0}(\u)$ replaced by a weak limit
$\bar A\in L^{p'}(0,T;\mathbf L^{p'}(\mathbb D))$ of
$\nu|\nabla\mathring\u_{h,\Delta t}|^{p-2}\nabla\mathring\u_{h,\Delta t}$.
\\\noindent
If, in addition,
\begin{equation}\label{eq-p-threshold}
p=2\qquad\text{or}\qquad p\ge \tfrac{8}{3},
\end{equation}
then $\bar A=\nu|\nabla\u|^{p-2}\nabla\u$, and consequently $(\u,\B,\theta)$ is a weak
solution of the rotating $p$-MHD-Boussinesq system in the sense of
Definition~\ref{def-weak-solution}.
    \end{theorem}

\noindent
Condition \eqref{eq-p-threshold} is needed only to identify
$\bar A=a_{0}(\u)$; the compactness arguments and the limit passage in all other terms hold
for every $p\ge2$. Its three-dimensional origin is explained in
Remark~\ref{rem-p-threshold}.

\begin{remark}\label{rem-p-threshold}
Condition \eqref{eq-p-threshold} is needed only in Step~2 to identify the weak limit
$\bar A$ of the $p$-Laplacian flux by Minty's method. This requires using $\u$ itself as a
test function in the limit momentum equation \eqref{eq-p-laplace convergence}. By the
Lions-Magenes lemma, it is therefore sufficient to have $\u_t\in L^{p'}(0,T;\mathcal X_0').$ The flux, buoyancy and forcing terms impose no additional restriction. For the convective
term, the interpolation
\[
\u\in L^\infty(0,T;\mathbf L^2)
\cap L^p(0,T;\mathcal X_0)
\quad\Longrightarrow\quad
\u\in L^{5p/3}(\mathbb D_T)
\]
gives the familiar three-dimensional condition $p\ge \frac{11}{5}$ from the $p$-Navier-Stokes theory \cite{DieningRuzickaWolf2010}. The stronger restriction comes from the Lorentz force. For $2\le p<3$, let $p^*=\frac{3p}{3-p}.$ Using
\[
|(\B\times(\nabla\times\B),\mathbf v)|
\le
\|\B\|_{\mathbf L^s}
\|\nabla\times\B\|_{\mathbf L^2}
\|\mathbf v\|_{\mathbf L^{p^*}},
\qquad
\frac1s=\frac12-\frac1{p^*},
\]
and interpolating $\B$ between $L^\infty(0,T;\mathbb H)$ and $L^2(0,T;\mathbf L^{3+\delta_1}),$ one obtains the compatibility conditions
\[
\theta\ge
\frac{2(3+\delta_1)(3-p)}{3p(1+\delta_1)},
\qquad
\theta\le\frac{p-2}{p}.
\]
These are compatible whenever
\begin{equation}\label{eq-pstar}
p\ge
p_\star:=
\frac{3\gamma+2}{1+\gamma},
\qquad
\gamma:=
\frac{2(3+\delta_1)}{3(1+\delta_1)}.
\end{equation}
Since $p_\star<\frac83$ for every $\delta_1>0$, the domain-independent condition $p\ge\frac83$ is sufficient. In particular, the classical threshold $p\ge\frac{11}{5}$ is not enough here;
the additional restriction is caused by the Lorentz force.

For $p=2$, the diffusion operator is linear and the flux is identified directly by weak
convergence. Hence the Newtonian case is also covered. The intermediate range
$2<p<\frac83$ would require additional tools, such as Lipschitz truncation
\cite{DieningRuzickaWolf2010,DieningKreuzerSuli2013,SuliTscherpel2020}. Thus, the threshold
\eqref{eq-p-threshold} arises from the loss of the $L^\infty$ embedding in three dimensions.
\end{remark}

\subsection{A time-translate estimate for the velocity}\label{subsec-timetranslate}

A direct time-derivative estimate for the velocity is unavailable because testing
\eqref{eq-FEM-1} with $d_t\u_h^m$ produces
$\nu(|\nabla\u_h^m|^{p-2}\nabla\u_h^m,\nabla d_t\u_h^m)$, which is not controlled by the
energy estimate for $p>2$. We therefore use time translates and Simon's criterion
(Lemma~\ref{Simon}).

 \begin{lemma}[Time-translate estimate for the velocity]\label{time translate}
         For the sequences $\left\{\left(\u_{h,\Delta t},\B_{h,\Delta t},\theta_{h,\Delta t}\right)\right\}$, there exists a constant $C \equiv C(T)$ independent of time-step $\Delta t$ and mesh-size $h$ such that 
         $$\Delta t\sum_{n=1}^{N-j} \|\u^{n+j}_h- \u^{n}_h\|_{\mathbf{L}^{2}}^2\leq C\,t_j^\alpha ,
\qquad 1\le j\le N-1,$$
where $\alpha:=\min\bigl\{\tfrac14,\tfrac1p\bigr\}>0$.
        
    \end{lemma}
 \begin{proof}
      Putting  $d_{t}\u^{m}=\frac{\u^{m}-\u^{m-1}}{\Delta t}$ in \eqref{eq-FEM-1}, we obtain,
{\small{
     \begin{align*}
        &\left(\u^{m}_h-\u^{m-1}_h,\v_{h}\right)+\Delta t a_{1}\left(\u^{m-1}_h,\u^{m}_h,\v_{h}\right)+{\nu\Delta t\left(|\nabla\u^{m}_h|^{p-2}\nabla\u^{m}_h,\nabla\v_{h}\right) } 
         +\Delta t\left(\B^{m-1}_h\times(\nabla \times\B^{m}_h),\v_{h}\right)=\Delta t (\mathbf{f}^{m},\v_{h})+\Delta t(\theta^{m-1}_he_3,\v_{h})\,.
     \end{align*}
}}
\noindent
Summing over $m=n+1,\ldots,n+j$ gives
\begin{align}\label{eq-u^n+l}
        &\left(\u^{n+j}_h-\u^{n}_h,\v_{h}\right)+\Delta t \sum_{i=1}^{j} a_{1}\left(\u^{n-1+i}_h,\u^{n+i}_h,\v_{h}\right)+{\nu\Delta t\sum_{i=1}^{j}\left(|\nabla\u^{n+i}_h|^{p-2}\nabla\u^{n+i}_h,\nabla\v_{h}\right)} \nonumber \\
        &
       2\Delta t\sum_{i=1}^j\big(
\boldsymbol{\omega}_{r}\times\u_h^{n+i},
\mathbf v_h
\big) +\Delta t\sum_{i=1}^{j}\left(\B^{n-1+i}_h\times(\nabla \times\B^{n+i}_h),\v_{h}\right)=\Delta t\sum_{i=1}^{j} (\mathbf{f}^{n+i},\v_{h})+\Delta t\sum_{i=1}^{j}(\theta^{n+i-1}_he_3,\v_{h}).
  \end{align}
\noindent
Choose $\v_{h}=\Delta t(\u^{n+j}_h-\u^{n}_h)$ in \eqref{eq-u^n+l} and take the sum over $n$ from $1$ to $N-j$ to get,
\begin{align}\label{eq-sum-over-n-u}
      &\sum_{n=1}^{N-j}\left(\u^{n+j}_h-\u^{n}_h,\Delta t(\u^{n+j}_h-\u^{n}_h)\right)+\Delta t \sum_{n=1}^{N-j}\sum_{i=1}^{j} a_{1}\left(\u^{n-1+i}_h,\u^{n+i}_h,\Delta t(\u^{n+j}_h-\u^{n}_h)\right)+ 2\Delta t \sum_{n=1}^{N-j}\sum_{i=1}^{j}\big(
\boldsymbol{\omega}_{r}\times\u_h^{n+i},
\Delta t(\u^{n+j}_h-\u^{n}_h)
\big)\nonumber \\\nonumber
        &+{\nu\Delta t\sum_{n=1}^{N-j}\sum_{i=1}^{j}\left(|\nabla\u^{n+i}_h|^{p-2}\nabla\u^{n+i}_h,\nabla\left(\Delta t(\u^{n+j}_h-\u^{n}_h)\right)\right)}
        +\Delta t\sum_{n=1}^{N-j}\sum_{i=1}^{j}\left(\B^{n-1+i}_h\times(\nabla \times\B^{n+i}_h),\Delta t(\u^{n+j}_h-\u^{n}_h)\right) \\
        &=\Delta t\sum_{n=1}^{N-j}\sum_{i=1}^{j} \left(\mathbf{f}^{n+i},\Delta t(\u^{n+j}_h-\u^{n}_h)\right)+\Delta t\sum_{n=1}^{N-j}\sum_{i=1}^{j}\left(\theta^{n+i-1}_he_3,\Delta t(\u^{n+j}_h-\u^{n}_h)\right).
\end{align}
\noindent
We use the elementary counting inequality
\begin{equation}\label{eq-counting}
\sum_{n=1}^{N-j}\sum_{i=1}^{j}c_{n+i}\ \le\ j\sum_{m=1}^{N}c_{m},
\qquad c_{m}\ge0,
\end{equation}
together with H\"older's inequality applied first in $i$ and then in $n$. Set
\[
g_{m}:=\|\nabla\u^{m}_{h}\|_{\mathbf L^{2}},\qquad
b_{m}:=\|\nabla\times\B^{m}_{h}\|_{\mathbf L^{2}},\qquad
w_{n}:=\|\nabla(\u^{n+j}_{h}-\u^{n}_{h})\|_{\mathbf L^{2}},
\]
so that, by \eqref{eq-bounds-2} and Corollary~\ref{cor-L2gradient},
\begin{equation}\label{eq-abbrev-bounds}
\Delta t\sum_{m=1}^{N}g_{m}^{2}\le C,
\qquad
\Delta t\sum_{m=1}^{N}b_{m}^{2}\le C,
\qquad
\Delta t\sum_{n=1}^{N-j}w_{n}^{2}\le C .
\end{equation}
In particular $\max_{m}g_{m}^{2}\le C/\Delta t$ and $\max_{m}b_{m}^{2}\le C/\Delta t$.\\
\noindent
The first term is
     \begin{align}
         \sum_{n=1}^{N-j}\left(\u^{n+j}_h-\u^{n}_h,\Delta t(\u^{n+j}_h-\u^{n}_h)\right) = \Delta t \sum_{n=1}^{N-j} \|\u^{n+j}_h-\u^{n}_h\|_{\mathbf{L}^2}^{2}\,.
     \end{align}
\noindent
For the convective term, Young's inequality gives
     \begin{align}
       \nonumber  &(\Delta t)^2\sum_{n=1}^{N-j}\sum_{i=1}^{j} a_{1}\left(\u^{n-1+i}_h,\u^{n+i}_h,\u^{n+j}_h-\u^{n}_h\right) \nonumber \\
         &\le C(\Delta t)^2\sum_{n=1}^{N-j}\sum_{i=1}^{j}\|\nabla \u^{n+i}_h\|_{\mathbf{L}^2}^{2}+ C(\Delta t)^2\sum_{n=1}^{N-j}\sum_{i=1}^{j}\|\u^{n-1+i}_h\|_{\mathbf{L}^4}^{2}\|\u^{n+j}_h-\u^{n}_h\|_{\mathbf{L}^4}^{2}.
         \end{align}

\noindent
Using \eqref{eq-bounds-2} and $t_{j}=\Delta t\cdot j$,  we further get
      \begin{align}\label{eq-Lem8}
          &\le Ct_{j}+C(\Delta t)^2\sum_{n=1}^{N-j}\sum_{i=1}^{j}\|\u^{n-1+i}_h\|_{\mathbf{L}^4}^{2}\|\u^{n+j}_h-\u^{n}_h\|_{\mathbf{L}^4}^{2}. 
      \end{align}
\noindent
For the second term of \eqref{eq-Lem8} we use the Ladyzhenskaya inequality in three dimensions,
$\|\v\|_{\mathbf L^{4}}\le C\|\v\|^{1/4}_{\mathbf L^{2}}\|\nabla\v\|^{3/4}_{\mathbf L^{2}}$,
together with \eqref{eq-bounds-1}, to obtain
\begin{align}
C(\Delta t)^2\sum_{n=1}^{N-j}\sum_{i=1}^{j}\|\u^{n-1+i}_h\|_{\mathbf{L}^4}^{2}\|\u^{n+j}_h-\u^{n}_h\|_{\mathbf{L}^4}^{2}
\ \le\ C(\Delta t)^{2}\sum_{n=1}^{N-j}w_{n}^{\frac32}\sum_{i=1}^{j}g_{n+i-1}^{\frac32}.
\end{align}
H\"older's inequality in $i$ with exponents $(\tfrac43,4)$ and \eqref{eq-abbrev-bounds} give
$\sum_{i=1}^{j}g_{n+i-1}^{3/2}\le\bigl(\sum_{i=1}^{j}g_{n+i-1}^{2}\bigr)^{3/4}j^{1/4}
\le(C/\Delta t)^{3/4}j^{1/4}$, and the same inequality in $n$ gives
$\sum_{n=1}^{N-j}w_{n}^{3/2}\le(C/\Delta t)^{3/4}N^{1/4}$. Hence
\begin{align}
C(\Delta t)^{2}\sum_{n=1}^{N-j}w_{n}^{\frac32}\sum_{i=1}^{j}g_{n+i-1}^{\frac32}
\ \le\ C(\Delta t)^{2}\,(\Delta t)^{-\frac34}j^{\frac14}\,(\Delta t)^{-\frac34}N^{\frac14}
\ =\ C\,T^{\frac14}\,t_{j}^{\frac14},
\end{align}
since $\Delta t\,N=T$ and $\Delta t\,j=t_{j}$. Note that every index has remained inside its own
sum.\\
{ The Coriolis force term we can handle as follows 
\begin{align}
 \nonumber2\Delta t \sum_{n=1}^{N-j}\sum_{i=1}^{j}\big(
\boldsymbol{\omega}_{r}\times\u_h^{n+i},
\Delta t(\u^{n+j}_h-\u^{n}_h)
\big) &\le   \Delta t \sum_{n=1}^{N-j}\sum_{i=1}^{j}|\boldsymbol{\omega}_{r}|\|\u_h^{n+i}\|_{\L^2}^2  +(\Delta t)^2 \sum_{n=1}^{N-j}\sum_{i=1}^{j}\|\u^{n+j}_h-\u^{n}_h\|_{\L^2}^2\\
&\le C(T)t_j|\boldsymbol{\omega}_r|
\end{align}}


\smallskip

\noindent
For the Lorentz term, add and subtract $\mathcal Z(\B^{n-1+i}_h)$ and use
\eqref{eq-triple-product}:
\begin{align}\label{I_2}
\nonumber (\Delta t)^2\sum_{n=1}^{N-j}\sum_{i=1}^{j}\left(\B^{n-1+i}_h\times(\nabla \times\B^{n+i}_h),\u^{n+j}_h-\u^{n}_h\right) &\le C(\Delta t)^2\sum_{n=1}^{N-j}\sum_{i=1}^{j}\|\B^{n-1+i}_h-\mathcal{Z}(\B^{n-1+i}_h)\|_{\mathbf{L}^2}\, b_{n+i}\,\|\u^{n+j}_h-\u^{n}_h\|_{\mathbf{L}^\infty}\\
&\quad + C(\Delta t)^2\sum_{n=1}^{N-j}\sum_{i=1}^{j}\|\mathcal{Z}(\B^{n-1+i}_h)\|_{\mathbf{L}^3}\, b_{n+i}\,\|\u^{n+j}_h-\u^{n}_h\|_{\mathbf{L}^6}
=: I_{1}+I_{2}.
\end{align}
For $I_{1}$ we use \eqref{ine-zbh} and the inverse estimate \eqref{eq-inverse estimate} with
$(\mu_1,q_1)=(0,6)$, $(\mu_2,q_2)=(0,\infty)$, which contributes the factor $h^{-1/2}$:
\[
\|\B^{n-1+i}_h-\mathcal{Z}(\B^{n-1+i}_h)\|_{\mathbf{L}^2}\|\u^{n+j}_h-\u^{n}_h\|_{\mathbf{L}^\infty}
\le C h^{\frac12+\varrho}b_{n+i-1}\,h^{-\frac12}\|\u^{n+j}_h-\u^{n}_h\|_{\mathbf{L}^6}
\le C h^{\varrho} b_{n+i-1}w_{n},
\]
where $\mathbf W^{1,p}_{0}\hookrightarrow\mathbf L^{6}$ for $p\ge2$ was used in the last step.
For $I_{2}$, Lemma~\ref{Z=curl} and the boundedness of $\mathbb D$ give
$\|\mathcal Z(\B^{m}_{h})\|_{\mathbf L^{3}}\le C\|\mathcal Z(\B^{m}_{h})\|_{\mathbf L^{3+\delta_{1}}}\le C\,\B_{m}$.
Since $h\le1$, both terms are therefore bounded by the same expression:
\begin{equation}\label{eq-I11}
I_{1}+I_{2}\ \le\ C(\Delta t)^{2}\sum_{n=1}^{N-j}w_{n}\sum_{i=1}^{j}b_{n+i-1}b_{n+i}.
\end{equation}
Write $S_{n}:=\sum_{i=1}^{j}\bigl(b_{n+i-1}^{2}+b_{n+i}^{2}\bigr)$, so that
$\sum_{i}b_{n+i-1}b_{n+i}\le\tfrac12 S_{n}$ by Young's inequality. By
\eqref{eq-abbrev-bounds}, $S_{n}\le 2\Delta t^{-1}C$ for every $n$, while
\eqref{eq-counting} gives $\sum_{n=1}^{N-j}S_{n}\le 2j\sum_{m=1}^{N}b_{m}^{2}\le 2jC/\Delta t$;
consequently
\[
\sum_{n=1}^{N-j}S_{n}^{2}\ \le\ \Bigl(\max_{n}S_{n}\Bigr)\sum_{n=1}^{N-j}S_{n}
\ \le\ \frac{C\,j}{(\Delta t)^{2}} .
\]
The Cauchy-Schwarz inequality in $n$ and \eqref{eq-abbrev-bounds} then yield
\begin{align}\label{eq-I22}
I_{1}+I_{2}\ \le\ C(\Delta t)^{2}\Bigl(\sum_{n=1}^{N-j}S_{n}^{2}\Bigr)^{\frac12}\Bigl(\sum_{n=1}^{N-j}w_{n}^{2}\Bigr)^{\frac12}
\ \le\ C(\Delta t)^{2}\,\frac{j^{\frac12}}{\Delta t}\,\frac{1}{(\Delta t)^{\frac12}}
\ =\ C\,t_{j}^{\frac12}.
\end{align}
\noindent
For the $p$-Laplacian term, using \eqref{eq-bounds-2}, \eqref{eq-bounds-3}, and the
embedding results, we obtain
 \begin{align}
       \nonumber  \nu\Delta t\sum_{i=1}^{j}\left(|\nabla\u^{n+i}_h|^{p-2}\nabla\u^{n+i}_h,\nabla\v_{h}\right) &=\nu(\Delta t)^{2}\sum_{i=1}^{j}\sum_{n=1}^{N-j}\left(|\nabla\u^{n+i}|^{p-2}\nabla\u^{n+i}_h,\nabla(\u^{n+j}_h-\u^{n}_h)\right)
       \\ &\le C \left((\Delta t)^{2}\sum_{n=1}^{N-j}\|\nabla(\u^{n+j}_h-\u^{n}_h)\|_{\mathbf{L}^p}\sum_{i=1}^{j}\|  \nabla\u^{n+i}_h\|^{p-1}_{\mathbf{L}^p}\right).
          \end{align}
          By H\"older's inequality and \eqref{eq-bounds-2},
          \begin{align}
        \nonumber  &\le C (\Delta t)\sum_{n=1}^{N-j}\|\nabla(\u^{n+j}_h-\u^{n}_h)\|_{\mathbf{L}^p}(\Delta t)^{\frac{1}{p}}\left(\sum_{i=1}^{j}(\Delta t)\|\nabla\u^{n+i}_h\|^{p}_{\mathbf{L}^p}\right)^{\frac{p-1}{p}}\left(\sum_{i=1}^{j}1\right)^{\frac{1}{p}}\\
          &\le C t_{j}^{\frac{1}{p}}(\Delta t)^{\frac{1}{p'}}\left(\sum_{n=1}^{N-j}(\Delta t)\|\nabla(\u^{n+j}_h-\u^{n}_h)\|_{\mathbf{L}^p}^{p}\right)^{\frac{1}{p}}\left(\sum_{n=1}^{N-j}1\right)^{\frac{1}{p'}} \le C(T) t_{j}^{\frac{1}{p}}\,,
          \end{align}
\noindent
For the two terms on the right of \eqref{eq-sum-over-n-u}, first consider the forcing term. H\"older's inequality in $i$ and then in $n$, with \eqref{eq-bounds-1} and
$\Delta t\sum_{m}\|\mathbf f^{m}\|^{2}_{\mathbf L^{2}}\le C$, gives
\begin{align*}
(\Delta t)^{2}\sum_{n=1}^{N-j}\sum_{i=1}^{j}\left(\mathbf f^{n+i},\u^{n+j}_h-\u^{n}_h\right)
&\le (\Delta t)^{2}\sum_{n=1}^{N-j}\sum_{i=1}^{j}\|\mathbf f^{n+i}\|_{\mathbf L^{2}}\|\u^{n+j}_h-\u^{n}_h\|_{\mathbf L^{2}}\\
&\le C\,\Delta t\sum_{n=1}^{N-j}\Big(\Delta t\sum_{i=1}^{j}\|\mathbf f^{n+i}\|^{2}_{\mathbf L^{2}}\Big)^{\frac12}t_{j}^{\frac12}
\ \le\ C(T)\,t_{j}^{\frac12},
\end{align*}
and the same computation applied to $\{\theta^{n+i-1}_h\}$, with \eqref{eq-bounds-1} in place of
the bound on the data, yields $C(T)t_{j}$ for the buoyancy term. Both are of order
$t_{j}^{1/2}$ or better, hence dominated by $C(T)t_{j}^{\alpha}$ since $\alpha\le\frac14$.
Collecting all the estimates,
$$\Delta t\sum_{n=1}^{N-j}\|\u^{n+j}_h-\u^{n}_h\|_{\mathbf{L}^{2}}^{2}\le C\, t_{j}^\alpha.$$

 \end{proof}
 
\begin{corollary}\label{cor-Hbeta}
By \cite[Remark 3.4]{DroniouGoldysLe2022}, the time-translate estimate of
Lemma~\ref{time translate} implies a bound on $\u_{h,\Delta t}$ in
$H^{\beta}(0,T;\L^2(\mathbb{D}))$, uniform in $h$ and $\Delta t$, for every
$\beta\in\bigl(0,\tfrac{\alpha}{2}\bigr)\cap\bigl(0,\tfrac12\bigr)$.
\end{corollary}

\subsection{Time-derivative estimates for the magnetic field and the temperature}\label{subsec-timederivative}

For the magnetic induction and temperature, the discrete time derivatives can be controlled in
dual norms, allowing the use of compactness arguments based on Lemma~\ref{Aubin Lions}.

 \begin{lemma}[Time-derivative of the magnetic field]\label{time derivative}
             For the sequences $\left\{\left(\u_{h,\Delta t},\B_{h,\Delta t},\theta_{h,\Delta t}\right)\right\}$, there exists a constant $C \equiv C(T)$ independent of time-step $\Delta t$ and mesh-size $h$ such that 
        
         $$ \|\partial_{t}\B_{h,\Delta t}\|_{L^{r}\left(0,T;\left(\mathcal{W}\cap \mathbf{W}^{2, 2}(\mathbb{D})\right)'\right)}\leq C ,$$
with $r=\frac{2}{2-\zeta}, \zeta=\min \left\{\varrho,\left(\frac{1}{2}+\varrho\right) \frac{2 \delta_{1}}{3\left(1+\delta_{1}\right)}, \frac{2 \delta_{1}}{3\left(1+\delta_{1}\right)}, \frac{1}{4},\frac{1}{p}\right\}.$   
\end{lemma}
\begin{proof}Let $\Pi^{\mathcal W_{0}}_{h}:\mathbf L^{2}(\mathbb D)\to\mathcal W_{0h}$ denote
the $\mathbf L^{2}$-orthogonal projection onto the constrained space $\mathcal W_{0h}$.
{
Moreover $\B_{h,\Delta t}\in\mathcal W_{0h}$ gives
$(\mathbf w-\Pi^{\mathcal W_{0}}_{h}\mathbf w,\partial_{t}\B_{h,\Delta t})=0$. With this choice
of test function,} $\left(\u_{h,\Delta t}, \B_{h,\Delta t}\right)$ satisfy
\begin{align}
    &\int_{0}^{T}\left(\partial_{t} \B_{h,\Delta t}, \Pi^{\mathcal W_{0}}_{h} \mathbf{w}\right) dt+\kappa\int_{0}^{T}\left( \nabla \times \mathring{\B}_{h, \Delta t}, \nabla \times \Pi^{\mathcal W_{0}}_{h} \mathbf{w}\right) dt \nonumber
    \\ &\qquad +\int_{0}^{T}\left((\nabla \times \Pi^{\mathcal W_{0}}_{h} \mathbf{w} )\times \breve{\B}_{h, \Delta t}, \mathring{\u}_{h, \Delta t}\right) dt=\int_{0}^{T}\left(\mathring{\mathbf{g}}_{\Delta t}, \Pi^{\mathcal W_{0}}_{h} \mathbf{w}\right) dt\,.
\end{align}
By applying the Cauchy-Schwarz and interpolation inequalities, for any $\C\in L^{r'}(0,T;\mathcal{W}\cap\mathbf{W}^{2, 2}(\mathbb{D}))$ we obtain, splitting $\breve\B_{h,\Delta t}=(\breve\B_{h,\Delta t}-\mathcal Z(\breve\B_{h,\Delta t}))+\mathcal Z(\breve\B_{h,\Delta t})$ and estimating the two contributions exactly as $I_{1}$ and $I_{2}$ in the proof of Lemma~\ref{time translate} (see also \cite[Lemma~12]{DING2025116470}), that
\begin{align*}
   &\int_{0}^{T}\left | (\mathring{\u}_{h,\Delta t} \times \breve{\B}_{h, \Delta t}, \nabla \times\Pi^{\mathcal W_{0}}_{h}\C)\right| dt
\le C\|\nabla \times\Pi^{\mathcal W_{0}}_{h}\C\|_{L^{r'}(0,T;\mathbf{L}^{2}(\mathbb{D}))}.
\end{align*} 
Since $1<r<2$, we  derive $$\int_{0}^{T}\left|\left(\nabla \times\mathring{\B}_{h,\Delta t},\nabla \times\Pi^{\mathcal W_{0}}_{h}\C\right)\right|dt \le C \|\nabla \times\Pi^{\mathcal W_{0}}_{h}\C\|_{L^{r'}(0,T;\mathbf{L}^{2}(\mathbb{D}))}.$$It remains to bound $\|\nabla\times\Pi^{\mathcal W_{0}}_{h}\C\|_{\mathbf L^{2}}$. We do
not invoke the commuting-diagram property here: that property belongs to the canonical
N\'ed\'elec interpolant and not to the $\mathbf L^{2}$-projection, which is a different
operator. Instead we use the $\mathbf H(\operatorname{curl})$-stability of the
$\mathbf L^{2}$-projection onto $\mathcal W_{h}$ on quasi-uniform meshes
\cite[Ch.~5]{Monk2003}, together with
$\nabla\times\Pi^{\mathcal W_{0}}_{h}\C=\nabla\times\Pi^{\mathcal W}_{h}\C$, to obtain
$$\|\nabla \times\Pi^{\mathcal W_{0}}_{h}\C\|_{\mathbf{L}^{2}(\mathbb{D})}\le c \|\C\|_{{\mathbf W}^{2,2}(\mathbb{D})}+c\|\nabla \times\C\|_{{\mathbf L}^{2}(\mathbb{D})}.$$
$$\int_{0}^{T}{\left|\left(\partial_{t}\B_{h,\Delta t}, \C\right)\right|}dt\le C\|\C\|_{L^{r'}(0,T;\mathcal{W}\cap\mathbf{W}^{2, 2}(\mathbb{D}))}.$$
This implies that
$$\|\partial_{t}\B_{h,\Delta t}\|_{L^{r}\left(0,T;\left(\mathcal{W}\cap \mathbf{W}^{2, 2}(\mathbb{D})\right)'\right)}\le C.$$
    \end{proof}

\begin{lemma}[Time derivative for temperature]\label{time derivative-theta}For the sequences $\left\{\left(\u_{h,\Delta t},\B_{h,\Delta t},\theta_{h,\Delta t}\right)\right\}$, there exists a constant $C \equiv C(T)$ independent of time-step $\Delta t$ and mesh-size $h$ such that 
   $$ \|\partial_{t}\theta_{h \Delta t}\|_{L^{\frac{4}{3}}(0,T;(\mathcal{Y}_{0}\cap H^{1+s}(\mathbb{D}))^{'})}\le C. $$
\end{lemma}
\begin{proof}
    We denote by $\Pi^{\mathcal Y}_{h}$ the $L^2$-orthogonal projection onto $\mathcal Y_{h}$.Then by \eqref{eq-FEM-3}, we know that there exists $(\u_{h \Delta t},\theta_{h \Delta t})$ such that
    \begin{align}
        (d_{t}\theta_{h \Delta t},\Pi^{\mathcal Y}_{h}\phi_{h})+\sigma(\nabla \mathring{\theta}_{h\Delta t},\nabla\Pi^{\mathcal Y}_{h}\phi_{h})+a_{2}(\breve{\u}_{h\Delta t},\mathring{\theta}_{h\Delta t},\Pi^{\mathcal Y}_{h}\phi_{h})=(\mathring{\u}_{h\Delta t}\cdot e_3,\Pi^{\mathcal Y}_{h}\phi_{h})+(\mathring{\psi}_{\Delta t},\Pi^{\mathcal Y}_{h}\phi_{h}).
        \end{align}
        For any $\phi\in L^4(0,T;\mathcal{Y}_{0}\cap H^{1+s}(\mathbb{D})),$ Then by the help of Theorem \ref{bounds}, the Cauchy-Schwarz inequality and the interpolation inequality, there holds (see lemma 4.2 in \cite{ding2022convergence})
        \begin{align}
       \nonumber |a_{2}(\breve{\u}_{h\Delta t},\mathring{\theta}_{h\Delta t},\Pi^{\mathcal Y}_{h}\phi)| &=\frac{1}{2}\int_{0}^{T}|((\breve{\u}_{h\Delta t}\cdot\nabla \mathring{\theta}_{h\Delta t}),\Pi^{\mathcal Y}_{h}\phi) dt-\frac{1}{2}\int_{0}^{T}((\breve{\u}_{h\Delta t}\cdot\nabla \Pi^{\mathcal Y}_{h}\phi),\mathring{\theta}_{h\Delta t}) |dt\\
&\le C \int_{0}^{T}\|\nabla\breve{\u}_{h\Delta t}\|_{\L^2}^{\frac{1}{2}}\|\nabla\mathring{\theta}_{h\Delta t}\|_{L^2}\|\nabla\Pi^{\mathcal Y}_{h}\phi\|_{L^2} dt 
\le C\|\nabla\Pi^{\mathcal Y}_{h}\phi\|_{L^4(0,T;\L^2(\mathbb{D}))},
\end{align}
and
\begin{align}
   \sigma \int_{0}^{T}|(\nabla \mathring{\theta}_{h\Delta t},\nabla\Pi^{\mathcal Y}_{h}\phi_{h})|dt\le C\|\nabla \mathring{\theta}_{h\Delta t}\|_{L^2(0,T;\L^2(\mathbb{D}))}\|\nabla\Pi^{\mathcal Y}_{h}\phi\|_{L^2(0,T;\L^2(\mathbb{D}))}.
\end{align}
Combining these inequalities  we get that
\begin{align}
    \int_{0}^{T}|(\partial_{t}\theta_{h,\Delta t},\Pi^{\mathcal Y}_{h}\phi)|dt\le C\|\nabla\Pi^{\mathcal Y}_{h}\phi\|_{L^4(0,T;\L^2(\mathbb{D}))}.
\end{align}
Hence we get
 $$ \|\partial_{t}\theta_{h \Delta t}\|_{L^{\frac{4}{3}}(0,T;(\mathcal{Y}_{0}\cap H^{1+s}(\mathbb{D}))^{'})}\le C.$$
\end{proof}

\subsection{Proof of Theorem \textup{\ref{Thm-Conv-scheme}}}\label{subsec-conv-proof}

\begin{proof}
Step~1 gives compactness, while Steps~2-4 pass to the limit in the momentum, induction, and
temperature equations.

\smallskip
\noindent
\textbf{Step 1: compactness.}\, For $h, \Delta t\rightarrow 0$, there exists a convergent subsequence of $\left(\mathring{\u}_{h,\Delta t},\mathring{\B}_{h,\Delta t}, \mathring{\theta}_{h,\Delta t}\right)$ and the functions
        $$\u \in L^{\infty}(0, T ; \mathbf{L}^{2}) \cap {L^{p}}\left(0, T ; \mathbf{W}^{1, p}_{0}\right)\,, \quad  \B \in L^{\infty}(0, T ; \mathbb{H}) \cap L^{2}(0, T; \mathcal{K}), \quad \mbox{and} \quad \theta \in L^{\infty}(0, T ;{L^2}) \cap L^{2}(0, T; {{H}^{1}_{0}})$$
such that
\begin{align*}
    &\mathring{\u}_{h, \Delta t}, \breve{\u}_{h, \Delta t}, \u_{h, \Delta t} \xrightarrow{\text {weakly }} \u \quad \mbox{in } \ {L^{p}\left(0, T ; \mathbf W^{1,p}_{0}(\mathbb D)\right)},\\
&\mathring{\u}_{h, \Delta t}, \breve{\u}_{h, \Delta t}, \u_{h, \Delta t} \xrightarrow{\text {weakly* }} \u \quad \mbox{in } \ L^{\infty}\left(0, T ;\mathbf {L}^{2}(\mathbb{D})\right),\\
&\mathring{\B}_{h, \Delta t, } \breve{\mathbf{b}}_{h, \Delta t}, \B_{h, \Delta t} \xrightarrow{\text { weakly }} \mathbf{b} \quad \mbox{in } \ {L^{2}(0, T ; \mathcal{W})},\\
&\mathring{\B}_{h, \Delta t, } \breve{\B}_{h, \Delta t}, \B_{h, \Delta t} \xrightarrow{\text { weakly* }} \mathbf{b} \quad \mbox{in } \ L^{\infty}(0, T ; \mathbf{L}^{2}(\mathbb{D})),\\
&\mathring{\theta}_{h, \Delta t, } \breve{\theta}_{h, \Delta t}, \theta_{h, \Delta t} \xrightarrow{\text { weakly }} \theta \quad \mbox{in } \ {L^{2}(0, T ; {H}^{1}_{0}(\mathbb{D}))},\\
&\mathring{\theta}_{h, \Delta t, } \breve{\theta}_{h, \Delta t}, \theta_{h, \Delta t} \xrightarrow{\text { weakly* }} \theta \quad \mbox{in } \ L^{\infty}(0, T ; {L}^{2}(\mathbb{D})),\\
&\partial_{t} \B_{h,\Delta t} \xrightarrow{\text { weakly }} \B_{t} \quad \mbox{in } \ L^{r}\left(0, T ;\left(\mathcal{W} \cap \mathbf{W}^{2, 2}(\mathbb{D})\right)^{\prime}\right),\\
&\partial_{t} \theta_{h, \Delta t} \xrightarrow{\text { weakly } }\theta_{t} \quad \mbox{in } \ L^{\frac{4}{3}}\left(0, T ;\left(\mathcal{Y}_{0} \cap {H}^{1+s}(\mathbb{D})\right)^{\prime}\right).
\end{align*}

\noindent
The weak limit of the velocity is stated in $L^{p}(0,T;\mathbf W^{1,p}_{0}(\mathbb D))$  because for a general inf-sup stable pair the discrete space
$\mathcal X_{0h}$ is only discretely divergence-free and
$\mathcal X_{0h}\not\subset\mathcal X_{0}$; the constraint on the limit has to be recovered.
Let $q\in\mathcal C^{\infty}_{0}(\mathbb D)$ and let $q_{h}\in Q_{h}$ satisfy
$q_{h}\to q$ in $L^{2}(\mathbb D)$. Then, for every $\zeta\in\mathcal C^{\infty}_{c}(0,T)$,
\[
\int_{0}^{T}(\operatorname{div}\mathring\u_{h,\Delta t},q)\,\zeta\,dt
=\underbrace{\int_{0}^{T}(\operatorname{div}\mathring\u_{h,\Delta t},q_{h})\,\zeta\,dt}_{=\,0\ \text{by }\eqref{eq-space-X0h}}
+\int_{0}^{T}(\operatorname{div}\mathring\u_{h,\Delta t},q-q_{h})\,\zeta\,dt ,
\]
and the last term is bounded by
$C\|\nabla\mathring\u_{h,\Delta t}\|_{L^{2}(0,T;\mathbf L^{2})}\|q-q_{h}\|_{L^{2}}\to0$ by
Corollary~\ref{cor-L2gradient}. Passing to the limit on the left by weak convergence gives
$\int_{0}^{T}(\operatorname{div}\u,q)\zeta\,dt=0$ for all such $q$ and $\zeta$, so
$\operatorname{div}\u=0$ and hence $\u\in L^{p}(0,T;\mathcal X_{0})$. The same argument applied
to $\B$ with $s_{h}\in S_{h}$ in place of $q_{h}$, using \eqref{eq-space-W0h}, gives
$\B\in L^{2}(0,T;\mathcal K)$.
\\ \noindent
We have that
\begin{align*}
   \left\|\u_{h,\Delta t}-\mathring{\u}_{h,\Delta t}\right\|_{L^{2}(0,T;{\mathbf{L}^{2}}(\mathbb{D}))}^{2} &=\left\|\frac{t-t_{m-1}}{\Delta t}\u_{h}^{m}
    +\frac{t_{m}-t}{\Delta t}\u_{h}^{m-1}-{\u}_{h}^{m}\right\|_{L^{2}(0,T;{\mathbf{L}^{2}}(\mathbb{D}))}^{2}\\
    &=\left\|\frac{t-t_{m}}{\Delta t}\u_{h}^{m}
    -\frac{t-t_{m}}{\Delta t}\u_{h}^{m-1}\right\|_{L^{2}(0,T;{\mathbf{L}^{2}}(\mathbb{D}))}^{2}\\
    &=\sum_{m=1}^{N}\left\|\u_{h}^{m}-\u_{h}^{m-1}\right\|_{{\mathbf{L}^{2}}(\mathbb{D})}^{2}\int_{t_{m-1}}^{t_{m}}\left(\frac{t-t_{m}}{\Delta t}\right)^2 dt =\frac{\Delta t}{3}\sum_{m=1}^{N}\left\|\u_{h}^{m}-\u_{h}^{m-1}\right\|_{{\mathbf{L}^{2}}(\mathbb{D})}^{2}.
\end{align*}
\noindent
By \eqref{eq-bounds-3} the last sum is bounded by a constant independent of $h$ and $\Delta t$,
whence
\begin{equation}\label{eq-interp-vs-const}
\left\|\u_{h,\Delta t}-\mathring{\u}_{h,\Delta t}\right\|^{2}_{L^{2}(0,T;\mathbf{L}^{2}(\mathbb{D}))}
\le \frac{C\,\Delta t}{3}\longrightarrow 0
\qquad\text{as } \Delta t\to0,
\end{equation}
and identically for $\|\u_{h,\Delta t}-\breve{\u}_{h,\Delta t}\|_{L^{2}(0,T;\mathbf{L}^{2})}\longrightarrow 0
\qquad\text{as } \Delta t\to0$.

\noindent
Combining \eqref{eq-interp-vs-const} with the time-translate estimate of
Lemma~\ref{time translate}, the family $\{\u_{h,\Delta t}\}$ is bounded in
$L^{2}(0,T;\mathcal X_{0})$, its time translates satisfy
$\|\u_{h,\Delta t}(\cdot+\tau)-\u_{h,\Delta t}\|_{L^{2}(0,T-\tau;\mathbf L^{2})}\le C\tau^{\alpha/2}\to0$,
and $\mathcal X_{0}\hookrightarrow\hookrightarrow\mathbf L^{2}\hookrightarrow\mathbf L^{2}$.
Lemma~\ref{Simon} therefore yields a
subsequence along which $(\u_{h,\Delta t}),(\mathring{\u}_{h,\Delta t}),(\breve{\u}_{h,\Delta t})$
converge strongly to the same limit $\u$ in $L^{2}(0,T;\mathbf{L}^{2}(\mathbb{D}))$.
By interpolation with the uniform bound in $L^{p}(0,T;\mathbf L^{\frac{3p}{3-p}})$
(resp.\ $L^{p}(0,T;\mathbf L^{q})$ for every $q<\infty$ when $p\ge3$) the same subsequence
converges strongly in $L^{2}(0,T;\mathbf L^{r}(\mathbb D))$ for every $r<\min\{6,\tfrac{3p}{3-p}\}$
when $p<3$, and for every $r<\infty$ when $p\ge3$; in particular in
$L^{2}(0,T;\mathbf L^{4}(\mathbb D))$, which is the form used in Steps~2-4.\\
\noindent
By \eqref{eq-bounds-3},
\begin{align}
 \nonumber \|\B_{h,\Delta t}-\mathring{\B}_{h,\Delta t}\|^{2}_{L^{2}(0,T;\mathbf{L}^{2}(\mathbb{D}))}
 &=\left\|\frac{t-t_{m-1}}{\Delta t}\B^{m}_{h}+\frac{t_{m}-t}{\Delta t}\B^{m-1}_{h}-\B^{m}_{h} \right\|^{2}_{L^{2}(0,T;\mathbf{L}^{2}(\mathbb{D}))}\\
\nonumber =&\sum_{m=1}^{N}\|\B^{m}_{h}-\B^{m-1}_{h}\|^{2}_{\mathbf{L}^{2}(\mathbb{D})}\int_{t_{m-1}}^{t_{m}}\left(\frac{t-t_{m-1}}{\Delta t}\right)^{2} dt\\
 =&\frac{\Delta t}{3}\sum_{m=1}^{N}\|\B^{m}_{h}-\B^{m-1}_{h}\|^{2}_{\mathbf{L}^{2}(\mathbb{D})} \longrightarrow 0 \quad \mbox{as} \quad h, \Delta t\rightarrow 0.
  \end{align}  
Similarly, we have 
\begin{equation}
    \|\B_{h,\Delta t}-\breve{\B}_{h,\Delta t}\|^{2}_{L^{2}(0,T;\mathbf{L}^{2}(\mathbb{D}))}\longrightarrow 0\quad \mbox{as} \quad h, \Delta t\rightarrow 0.
\end{equation}
Lemma~\ref{Aubin Lions} cannot be applied to $\B_{h,\Delta t}$ directly: the natural
ambient space $\mathcal W=\mathbf H(\operatorname{curl};\mathbb D)$ is not compactly
embedded in $\mathbf L^{2}(\mathbb D)$, since gradients lie in the kernel of the curl, and
$\B_{h,\Delta t}$ does not belong to $\mathbf H(\operatorname{div};\mathbb D)$. We therefore
pass through the Hodge mapping \eqref{eq-hodge}. Since $\mathcal Z$ is linear and
$\nabla\times\mathcal Z(\B^{m}_{h})=\nabla\times\B^{m}_{h}$, the family
$\mathcal Z(\B_{h,\Delta t})$ lies in the fixed space $\mathcal K$ and is bounded in
$L^{2}(0,T;\mathcal K)$ by \eqref{eq-bounds-2}, while
$\partial_{t}\mathcal Z(\B_{h,\Delta t})=\mathcal Z(\partial_{t}\B_{h,\Delta t})$ is bounded in
$L^{r}(0,T;(\mathcal W\cap\mathbf W^{2,2})')$ by Lemma~\ref{time derivative} and the
$\mathcal W$-boundedness of $\mathcal Z$. Because
$\mathcal K\hookrightarrow\hookrightarrow\mathbf L^{2}(\mathbb D)$ by
Lemma~\ref{Embedding lemma}, Lemma~\ref{Aubin Lions} applies to
$\mathcal Z(\B_{h,\Delta t})$ and yields a subsequence converging strongly in
$L^{2}(0,T;\mathbf L^{2}(\mathbb D))$. Finally \eqref{ine-zbh} gives
\[
\|\B_{h,\Delta t}-\mathcal Z(\B_{h,\Delta t})\|_{L^{2}(0,T;\mathbf L^{2})}
\le Ch^{\frac12+\varrho}\|\nabla\times\B_{h,\Delta t}\|_{L^{2}(0,T;\mathbf L^{2})}
\longrightarrow0,
\]
so $\B_{h,\Delta t}$ and $\mathcal Z(\B_{h,\Delta t})$ share the same strong limit, which by
Lemma~\ref{Lemma 1} lies in $\mathcal K$. The same conclusion holds for
$\breve{\B}_{h,\Delta t}$ and $\mathring{\B}_{h,\Delta t}$, so all three converge strongly to
$\B$ in $L^{2}(0,T;\mathbf L^{2}(\mathbb D))$.
By \eqref{eq-bounds-3},
\begin{align}
 \nonumber \|\theta_{h,\Delta t}-\mathring{\theta}_{h,\Delta t}\|^{2}_{L^{2}(0,T;L^{q}(\mathbb{D}))}
 &=\frac{\Delta t}{3}\sum_{n=1}^{{N}}\|\theta^{n}_{h}-\theta^{n-1}_{h}\|^2_{L^q(\mathbb{D})}\\
\nonumber &\le\frac{\Delta t}{3}\sum_{n=1}^{{N}}\|\theta^{n}_{h}-\theta^{n-1}_{h}\|^{2{\vartheta}}_{L^1(\mathbb{D})}\|\theta^{n}_{h}-\theta^{n-1}_{h}\|^{2-2{\vartheta}}_{L^6(\mathbb{D})}\\
  &\le C(\Delta t)^{{\vartheta}}\left(\sum_{n=1}^{{N}}\|\theta^{n}_{h}-\theta^{n-1}_{h}\|^{2}_{L^2(\mathbb{D})}\right)^{{\vartheta}}\left(\Delta t\sum_{n=1}^{{N}}\|\theta^{n}_{h}-\theta^{n-1}_{h}\|^{2}_{L^6(\mathbb{D})}\right)^{1-{\vartheta}}{\longrightarrow 0},
 \end{align}
\noindent
where $\vartheta\in(0,1]$ is determined by
\[
\frac{1}{q}=\frac{\vartheta}{1}+\frac{1-\vartheta}{6},
\]
so that the admissible range is $1\le q<6$. 
Lemma~\ref{time translate}. The first term is bounded by \eqref{eq-bounds-3}; for the
second, the Sobolev embedding $H^{1}_{0}\hookrightarrow L^{6}$ gives
$\|\theta^{n}_{h}-\theta^{n-1}_{h}\|_{L^{6}}\le C\bigl(\|\nabla\theta^{n}_{h}\|_{L^{2}}+\|\nabla\theta^{n-1}_{h}\|_{L^{2}}\bigr)$,
whence $\Delta t\sum_{n}\|\theta^{n}_{h}-\theta^{n-1}_{h}\|^{2}_{L^{6}}\le C$ by
\eqref{eq-bounds-2}. Consequently the whole expression is $O\bigl((\Delta t)^{\vartheta}\bigr)$
and tends to zero.
 Similarly, we get
 \begin{equation}
    \|\theta_{h,\Delta t}-\breve{\theta}_{h,\Delta t}\|^{2}_{L^{2}(0,T;L^{q}(\mathbb{D}))}\longrightarrow 0\quad \mbox{as} \quad h, \Delta t\rightarrow 0.
    \end{equation}
    By Lemmas~\ref{Aubin Lions} and \ref{time derivative-theta}, $(\theta_{h,\Delta t}),(\breve{\theta}_{h,\Delta t}), (\mathring{\theta}_{h,\Delta t})$ converges strongly to $\theta$ as $h, \Delta t\rightarrow 0$ in the space $L^{2}(0,T;{L}^{q}(\mathbb{D}))$ where $1\le q<6$.

\medskip
\noindent
\textbf{Step 2: the momentum equation.}\, By \eqref{eq-FEM-1}, the discrete solution satisfies 
\begin{align}\label{eq-convergence}
   \nonumber &\int_{0}^{T}\int_{\mathbb{D}}\partial_{t}\u_{h, \Delta t}\cdot\mathbf{v}_{h} dx dt+\int_{0}^{T}a_{1}(\breve{\u}_{h, \Delta t},\mathring{\u}_{h, \Delta t},\mathbf{v}_{h}) dt+\nu \int_{0}^{T}\int_{\mathbb{D}}|\nabla \mathring{\u}_{h, \Delta t}|^{p-2}\nabla\mathring{\u}_{h, \Delta t}\cdot \nabla \mathbf{v}_{h} dx dt\\
    &+{{2\int_0^T\int_{\mathbb D}(\boldsymbol{\omega_r}\times\mathring{\u}_{h,\Delta t})\cdot\v_hdxdt}}+\int_{0}^{T}\int_{\mathbb{D}}(\breve{\B}_{h, \Delta t}\times(\nabla \times\mathring{\B}_{h, \Delta t}))\cdot \mathbf{v}_{h} dx dt=\int_{0}^{T}\int_{\mathbb{D}}\mathring{\mathbf{f}}_{\Delta t}\cdot\mathbf{v}_{h} dx dt+\int_{0}^{T}\int_{\mathbb{D}}\breve{\theta}_{h,\Delta t}\, e_3\cdot\mathbf{v}_{h} dx dt.
\end{align}
For $\mathbf v\in\mathcal J$, choose $\v_h\in\mathcal X_{0h}$ such that
$\v_{h}\rightarrow\v$ in $\mathbf{W}^{1,q}(\mathbb{D})$ for every $q<\infty$; this is possible
because $\mathbf v$ is smooth and compactly supported. 
\paragraph{The time-derivative term.}
Since the initial data are well approximated, integration by parts in time gives
\begin{align}
   \nonumber &\left|\int_{0}^{T}\int_{\mathbb{D}}(\partial_{t}\u_{h, \Delta t}\cdot\mathbf{v}_{h}-\partial_{t}\u\cdot\mathbf{v})\xi(t) dx dt\right|\\
  \nonumber  &=\left|\int_{0}^{T}\left[\left(-\u_{h, \Delta t},\partial_{t}(\mathbf{v}_{h}\xi(t))\right)+\left(\u,\partial_{t}(\mathbf{v}\xi(t))\right)-\left(\u^{0}_{h},\v_{h}\xi(0)\right)+\left(\u^{0},\v\xi(0)\right)\right]dt\right|\\
     &\le C \|\u_{h, \Delta t}-\u\|_{L^2(0,T;\mathbf{L}^2(\mathbb{D}))}\|\partial_{t}(\mathbf{v}_{h}\xi(t))\|_{L^2(0,T;\mathbf{L}^2(\mathbb{D}))}+\|\u\|_{L^2(0,T;\mathbf{L}^2(\mathbb{D}))}\|\partial_{t}[(\v_{h}-\v)\xi(t)]\|_{L^2(0,T;\mathbf{L}^2(\mathbb{D}))}\longrightarrow 0.
\end{align}
$\mbox{as} \quad h,\Delta t\rightarrow 0.$\\
\paragraph{The convective term.}
Following \cite[Theorem 3]{DING2025116470},
\begin{align}
   \nonumber& \left|\int_{0}^{T}\left[a_{1}(\breve{\u}_{h, \Delta t},\mathring{\u}_{h, \Delta t},\mathbf{v}_{h})-a_{1}(\u,\u,\v)\right]\xi(t) dt\right|\\
  \nonumber&\le \tfrac12\left|\int_{0}^{T}\!\!\int_{\mathbb{D}}\bigl[(\breve{\u}_{h, \Delta t}\otimes\mathbf v_{h})-(\u\otimes\v)\bigr]:\nabla\mathring{\u}_{h, \Delta t}\,\xi(t)\,dx\,dt\right|
  +\tfrac12\left|\int_{0}^{T}\!\!\int_{\mathbb{D}}(\u\otimes\v):\nabla(\mathring{\u}_{h, \Delta t}-\u)\,\xi(t)\,dx\,dt\right|\\
  \nonumber &\qquad+C \| \breve{\u}_{h, \Delta t}-\u\|_{L^2(0,T;\mathbf{L}^2(\mathbb{D}))}\|\mathring{\u}_{h, \Delta t}\|_{L^2(0,T;\mathbf{L}^4(\mathbb{D}))}\|\nabla \mathbf{v}_{h}\xi (t)\|_{L^{\infty}(0,T;\mathbf{L}^4(\mathbb{D}))}
  \\\nonumber &\qquad +C\|\u\|_{L^2(0,T;\mathbf{L}^4(\mathbb{D}))}\|\mathring{\u}_{h, \Delta t}-\u\|_{L^2(0,T;\mathbf{L}^2(\mathbb{D}))}\|\nabla \mathbf{v}_{h}\xi (t)\|_{L^{\infty}(0,T;\mathbf{L}^4(\mathbb{D}))}\\
&\qquad +C\|\u\|_{L^2(0,T;\mathbf{L}^4(\mathbb{D}))}\|\u\|_{L^2(0,T;\mathbf{L}^4(\mathbb{D}))}\|(\nabla \mathbf{v}_{h}-\nabla \v)\xi (t)\|_{L^{\infty}(0,T;\mathbf{L}^2(\mathbb{D}))}\longrightarrow 0 \quad \mbox{as} \quad h,\Delta t\rightarrow 0.
\end{align}
\noindent
In the first line the weakly convergent factor $\nabla\mathring\u_{h,\Delta t}$ is paired against
$\breve\u_{h,\Delta t}\otimes\mathbf v_{h}-\u\otimes\v$, which tends to zero strongly in
$L^{2}(0,T;\mathbf L^{2})$ by Step~1 and by $\mathbf v_{h}\to\v$, while
$\|\nabla\mathring\u_{h,\Delta t}\|_{L^{2}(0,T;\mathbf L^{2})}\le C$ by
Corollary~\ref{cor-L2gradient}; in the second part the fixed tensor
$\u\otimes\v\in L^{p'}(\mathbb D_{T})$ is paired against
$\nabla(\mathring\u_{h,\Delta t}-\u)\rightharpoonup0$ in $L^{p}(\mathbb D_{T})$. Once the limit
is reached, $\operatorname{div}\u=0$ allows $a_{1}(\u,\u,\v)$ to be rewritten as
$-\int_{\mathbb D}(\u\otimes\u):\nabla\v$, which is the form appearing in
\eqref{eq-p-laplace convergence}.\\
\paragraph{The Coriolis force.} We will do the addition and subtraction and have the following
\begin{align}
    \nonumber&\bigg|\int_0^T\int_{\mathbb{D}}\bigg(\boldsymbol{\omega}_r\times\mathring{\u}_{h,\Delta t})\cdot\v_h-(\boldsymbol{\omega}_r\times\u)\cdot\v\bigg)\xi(t)dxdt\bigg|\\
    \nonumber&\le \bigg|\int_0^T\int_{\mathbb{D}}\bigg(\boldsymbol{\omega}_r\times\mathring{\u}_{h,\Delta t})\cdot\v_h-(\boldsymbol{\omega}_r\times\u)\cdot\v_h+(\boldsymbol{\omega}_r\times\u)\cdot\v_h-(\boldsymbol{\omega}_r\times\u)\cdot\v\bigg)\xi(t)dxdt\bigg|\\
   & \le C|\boldsymbol{\omega}_r| \|\mathring{\u}_{h, \Delta t}-\u\|_{L^2(0,T;\mathbf{L}^2(\mathbb{D}))}\|\mathbf{v}_{h}\xi(t)\|_{L^2(0,T;\mathbf{L}^2(\mathbb{D}))}+|\boldsymbol{\omega}_r|\|\u\|_{L^2(0,T;\mathbf{L}^2(\mathbb{D}))}\|(\v_{h}-\v)\xi(t)\|_{L^2(0,T;\mathbf{L}^2(\mathbb{D}))}\longrightarrow 0
\end{align}
$\mbox{as} \quad h,\Delta t\rightarrow 0.$\\
\paragraph{The Lorentz force.}
Splitting $\breve{\B}_{h,\Delta t}=(\breve{\B}_{h,\Delta t}-\mathcal Z(\breve{\B}_{h,\Delta t}))+\mathcal Z(\breve{\B}_{h,\Delta t})$ and following \cite[Theorem 4.7]{ding2022convergence},
\begin{align}\label{eq-F}
    \nonumber&\left|\int_{0}^{T}\int_{\mathbb{D}}\left[(\breve{\B}_{h, \Delta t}\times(\nabla \times\mathring{\B}_{h, \Delta t}))\cdot \mathbf{v}_{h}-(\B\times(\nabla \times\B))\cdot\v\right]\xi(t) dx dt\right|\\
   \nonumber& =\bigg|
   \int_{0}^{T}\int_{\mathbb{D}}\Big[(\breve{\B}_{h, \Delta t}-\mathcal{Z}(\breve{\B}_{h,\Delta t}))\times(\nabla \times\mathring{\B}_{h, \Delta t}),\mathbf{v}_{h}) 
   \\ &\qquad \nonumber +(\mathcal{Z}(\breve{\B}_{h,\Delta t})\times(\nabla \times\mathring{\B}_{h,\Delta t}),\mathbf{v}_{h}) -(\B\times(\nabla \times\B),\v)\xi(t)\Big] dx dt\bigg|\\ 
   \nonumber &\le C h^{\varrho+\frac{1}{2}}\|(\nabla \times\breve{\B}_{h,\Delta t})\|_{L^{2}(0,T;\mathbf{L}^2(\mathbb{D}))}\|(\nabla \times\mathring{\B}_{h,\Delta t})\|_{L^{2}(0,T;\mathbf{L}^3(\mathbb{D}))}\|\v_{h}\xi(t)\|_{L^{\infty}(0,T;\mathbf{L}^{6}(\mathbb{D}))}\\
\nonumber \qquad&+\left|\int_{0}^{T}\int_{\mathbb{D}}\left[(\mathcal{Z}(\breve{\B}_{h,\Delta t})\times(\nabla \times\mathring{\B}_{h,\Delta t}),\mathbf{v}_{h})-(\B\times(\nabla \times\B),\v)\right]\xi(t) dx dt\right|\\
 &=: J_{h, 1}+|J_{h, 2}|.
\end{align}
\noindent
Since $\mathcal Z:\mathcal W_{0h}\to\mathcal K$ and
$\operatorname{div}(\mathcal Z(\B))=0$, 
\begin{align*}
    \|\mathcal{Z}(\breve{\B}_{h,\Delta t})\|_{L^{2}(0,T,\mathcal{V}(\mathbb{D})})=\|\nabla \times\mathcal{Z}(\breve{\B}_{h,\Delta t})\|_{L^{2}(0,T;\mathbf{L}^{2}(\mathbb{D}))}=\|\nabla \times\breve{\B}_{h,\Delta t}\|_{L^{2}(0,T;\mathbf{L}^{2}(\mathbb{D}))}\le C.
\end{align*}
\noindent
Also,
$\|\nabla\times\mathring{\B}_{h,\Delta t}\|_{L^{2}(0,T;\mathbf L^{2})}$ is bounded by
\eqref{eq-bounds-2}; hence $J_{h,1}\to0$ because of the factor
$h^{\varrho+\frac12}$.
Choose $\delta_2>0$ such that
\begin{align}
    \frac{1}{3+\delta_{1}}+\frac{1}{6-\delta_{2}}=\frac{1}{2}.
\end{align}
Then $H^{1}(\mathbb D)\hookrightarrow L^{6-\delta_2}(\mathbb D)$, and for the second
term in \eqref{eq-F},
\begin{align}
\nonumber|J_{h,2}| &=\left|\int_{0}^{T}\int_{\mathbb{D}}\left[(\mathcal{Z}(\breve{\B}_{h,\Delta t})\times\nabla \times\mathring{\B}_{h,\Delta t},\mathbf{v}_{h})-(\B\times\nabla \times\B,\v)\right]\xi(t) dx dt\right|\\
   \nonumber &\le C\|\mathcal{Z}(\breve{\B}_{h,\Delta t})-\B\|_{L^2(0,T;\mathbf{L}^{3}(\mathbb{D}))}\|\nabla \times\breve{\B}_{h,\Delta t}\|_{L^2(0,T;\mathbf{L}^{2}(\mathbb{D}))}\|\v_{h}\xi(t)\|_{L^\infty(0,T;\mathbf{L}^{3}(\mathbb{D}))}\\ \nonumber&\quad+C\|\mathcal{Z}(\breve{\B}_{h,\Delta t})\|_{L^{2}(0,T;\mathbf{L}^{3+\delta_{1}}(\mathbb{D}))}\|\nabla \times\B\|_{L^2(0,T;\mathbf{L}^{2}(\mathbb{D}))}\|(\v_{h}-\v)\xi(t)\|_{L^\infty(0,T;\mathbf{L}^{6-\delta_{2}}(\mathbb{D}))}
   \\  &\quad +C\left|\int_{0}^{T}\int_{\mathbb{D}}\left[(\mathcal{Z}(\breve{\B}_{h,\Delta t})\times(\nabla \times\mathring{\B}_{h,\Delta t}-\nabla \times\B),\mathbf{v}_{h})\right]\xi(t) dx dt\right|\longrightarrow 0\qquad \mbox{as}\qquad h,\Delta t\rightarrow 0.
\end{align} 
By the Sobolev inequality we infer
\begin{align}
    \nonumber\|\mathcal{Z}(\breve{\B}_{h,\Delta t})-\B\|_{L^2(0,T;\mathbf{L}^{3}(\mathbb{D}))}&\le\|\mathcal{Z}(\breve{\B}_{h,\Delta t})-\B\|_{L^2(0,T;\mathbf{L}^{2}(\mathbb{D}))}^{\frac{2\delta_{1}}{3(1+\delta_{1})}}\|\mathcal{Z}(\breve{\B}_{h,\Delta t})-\B\|_{L^2(0,T;\mathbf{L}^{3+\delta_{1}}(\mathbb{D}))}^{\frac{3+\delta_{1}}{3(1+\delta_{1})}}.\\
    \|\mathcal{Z}(\breve{\B}_{h,\Delta t})-\B\|_{L^2(0,T;\mathbf{L}^{2}(\mathbb{D}))}^{\frac{2\delta_{1}}{3(1+\delta_{1})}}\nonumber&\le\left( \|\mathcal{Z}(\breve{\B}_{h,\Delta t})-\breve{\B}_{h,\Delta t}\|_{L^2(0,T;\mathbf{L}^{2}(\mathbb{D}))}+\|\breve{\B}_{h,\Delta t}-\B\|_{L^2(0,T;\mathbf{L}^{2}(\mathbb{D}))}\right)^{\frac{2\delta_{1}}{3(1+\delta_{1})}}\\
    &\le \left(Ch^{\frac{1}{2}+\varrho}\|(\nabla \times\breve{\B}_{h,\Delta t})\|_{L^2(0,T;\mathbf{L}^{2}(\mathbb{D}))}+\|\breve{\B}_{h,\Delta t}-\B\|_{L^2(0,T;\mathbf{L}^{2}(\mathbb{D}))}\right)^{\frac{2\delta_{1}}{3(1+\delta_{1})}}.
\end{align}
\paragraph{The $p$-Laplacian flux.}
Set
\[
A_{h,\Delta t}(x,t):=\nu\,|\nabla \mathring{\u}_{h, \Delta t}|^{p-2}\nabla\mathring{\u}_{h, \Delta t},
\qquad
A(\mathbf G):=\nu\,|\mathbf G|^{p-2}\mathbf G .
\]
By \eqref{eq-bounds-2} the family $\{A_{h,\Delta t}\}$ is bounded in
$L^{p'}(0,T;\mathbf L^{p'}(\mathbb D))$, so along a further subsequence
$A_{h,\Delta t}\rightharpoonup \bar A$ weakly in $L^{p'}(0,T;\mathbf L^{p'}(\mathbb D))$.
Passing to the limit in \eqref{eq-convergence}
yields the limit momentum equation
\begin{align}\label{eq-p-laplace convergence}
   \nonumber &\int_{0}^{T}\int_{\mathbb{D}}\partial_{t}\u\cdot\mathbf{v} \,dx\, dt-\int_{0}^{T}\int_{\mathbb{D}}({\u}\otimes{\u})\cdot\nabla \mathbf{v} \,dx\, dt+ \int_{0}^{T}\int_{\mathbb{D}}\bar{A}(x,t):\nabla \v \,dx\, dt\\
    &+2\int_0^T\int_{\mathbb D}(\boldsymbol{\omega}_r\times\u)\cdot\v \,dx \,dt+\int_{0}^{T}\int_{\mathbb{D}}({\B}\times(\nabla \times{\B}))\cdot \mathbf{v} \,dx\, dt=\int_{0}^{T}\int_{\mathbb{D}}{\mathbf{f}}\cdot\mathbf{v}\,dx\, dt+\int_{0}^{T}\int_{\mathbb{D}}{\theta\, e_3}\cdot\mathbf{v}\,dx\, dt,
\end{align}
\noindent
with $\bar A$ in place of
$A(\nabla\u)$. It therefore remains to prove
\begin{equation}\label{eq-minty-goal}
\bar A=A(\nabla\u)\qquad\text{a.e.\ in }\mathbb D_{T}.
\end{equation}

\noindent
The key ingredient is the energy $\limsup$ inequality
\begin{equation}\label{eq-limsup-key}
\limsup_{h,\Delta t\to0}\int_{0}^{T}\!\!\int_{\mathbb{D}}A_{h,\Delta t}:\nabla \mathring{\u}_{h, \Delta t}\,dx\,dt
\ \le\
\int_{0}^{T}\!\!\int_{\mathbb{D}}\bar A:\nabla \u\,dx\,dt .
\end{equation}
To obtain \eqref{eq-limsup-key} we test \eqref{eq-FEM-1} with $\mathbf v_{h}=2\Delta t\,\u^{m}_{h}$
and sum over $m=1,\dots,N$. Exactly as in the proof of Theorem~\ref{bounds} the convective
term drops by skew-symmetry, so that
\begin{align}\label{eq-discrete-energy-identity}
\nonumber
\int_{0}^{T}\!\!\int_{\mathbb{D}}A_{h,\Delta t}:\nabla \mathring{\u}_{h, \Delta t}
&=\tfrac12\|\u^{0}_{h}\|^{2}_{\mathbf L^{2}}-\tfrac12\|\u^{N}_{h}\|^{2}_{\mathbf L^{2}}
-\tfrac12\sum_{m=1}^{N}\|\u^{m}_{h}-\u^{m-1}_{h}\|^{2}_{\mathbf L^{2}}\\
\nonumber&\quad
-2\int_0^T\int_{\mathbb D}(\boldsymbol{\omega}_r\times\mathring \u_{h,\Delta t})\cdot\mathring \u _{h,\Delta t}\,dx \,dt-\int_{0}^{T}\!\!\int_{\mathbb{D}}\big(\breve{\B}_{h,\Delta t}\times(\nabla\times\mathring{\B}_{h,\Delta t})\big)\cdot\mathring{\u}_{h,\Delta t}
\\&\qquad+\int_{0}^{T}\!\!\int_{\mathbb{D}}\mathring{\mathbf f}_{\Delta t}\cdot\mathring{\u}_{h,\Delta t}
+\int_{0}^{T}\!\!\int_{\mathbb{D}}\breve\theta_{h,\Delta t}\,e_{3}\cdot\mathring{\u}_{h,\Delta t}.
\end{align}
On the right of \eqref{eq-discrete-energy-identity} the fourth term is nonpositive  and the third term is zero; hence, these terms may be
discarded, $\|\u^{0}_{h}\|_{\mathbf L^{2}}\to\|\u^{0}\|_{\mathbf L^{2}}$ by the approximation of
the initial datum and weak lower semicontinuity together with
$\u_{h,\Delta t}(T)\rightharpoonup\u(T)$ in $\mathbf L^{2}(\mathbb D)$ gives
$\limsup(-\tfrac12\|\u^{N}_{h}\|^{2}_{\mathbf L^{2}})\le-\tfrac12\|\u(T)\|^{2}_{\mathbf L^{2}}$.
The three integral terms converge to their continuous counterparts:
$\mathring{\mathbf f}_{\Delta t}\to\mathbf f$ and $\mathring{\u}_{h,\Delta t}\to\u$ strongly in
$L^{2}(0,T;\mathbf L^{2})$ and $\breve\theta_{h,\Delta t}\to\theta$ strongly in
$L^{2}(0,T;L^{q})$, $1\le q<6$, by Step~1, while the Lorentz term is treated as in the estimate
of $J_{h,1}+J_{h,2}$ above with $\mathring{\u}_{h,\Delta t}$ in place of $\mathbf v_{h}$ 
legitimate because $\mathring{\u}_{h,\Delta t}\to\u$ strongly in $L^{2}(0,T;\mathbf L^{4})$,
$\nabla\times\mathring{\B}_{h,\Delta t}$ converges weakly in $L^{2}(0,T;\mathbf L^{2})$ and
$\mathcal Z(\breve{\B}_{h,\Delta t})$ strongly in $L^{2}(0,T;\mathbf L^{3})$. Hence
\begin{align}\label{eq-limsup-rhs}
\nonumber
\limsup_{h,\Delta t\to0}\int_{0}^{T}\!\!\int_{\mathbb{D}}A_{h,\Delta t}:\nabla \mathring{\u}_{h, \Delta t}
&\le \tfrac12\|\u^{0}\|^{2}_{\mathbf L^{2}}-\tfrac12\|\u(T)\|^{2}_{\mathbf L^{2}}\\
&\quad-2\int_0^T\int_{\mathbb D}(\boldsymbol{\omega}_r\times\u)\cdot\u -\int_{0}^{T}\!\!\int_{\mathbb{D}}\big(\B\times(\nabla\times\B)\big)\cdot\u
+\int_{0}^{T}\!\!\int_{\mathbb{D}}\mathbf f\cdot\u
+\int_{0}^{T}\!\!\int_{\mathbb{D}}\theta\,e_{3}\cdot\u .
\end{align}
\noindent
On the right of \eqref{eq-limsup-rhs} the third term is zero.
On the other hand, under \eqref{eq-p-threshold} every term of \eqref{eq-p-laplace convergence}
lies in $L^{p'}(0,T;\mathcal X_{0}')$  this is verified term by term in
Remark~\ref{rem-p-threshold}, the Lorentz force being the binding one  so
$\u_{t}\in L^{p'}(0,T;\mathcal X_{0}')$. Since also $\u\in L^{p}(0,T;\mathcal X_{0})$ and
$\mathcal X_{0}\subset\mathbb H\subset\mathcal X_{0}'$ is a Gelfand triple, the Lions-Magenes
lemma applies: $\u\in\mathcal C([0,T];\mathbb H)$, $\u$ is an admissible test function in
\eqref{eq-p-laplace convergence}, and the energy equality
\[
\tfrac12\|\u(T)\|^{2}_{\mathbf L^{2}}-\tfrac12\|\u^{0}\|^{2}_{\mathbf L^{2}}
+\int_{0}^{T}\!\!\int_{\mathbb{D}}\bar A:\nabla\u
+\int_{0}^{T}\!\!\int_{\mathbb{D}}\big(\B\times(\nabla\times\B)\big)\cdot\u
=\int_{0}^{T}\!\!\int_{\mathbb{D}}\mathbf f\cdot\u+\int_{0}^{T}\!\!\int_{\mathbb{D}}\theta\,e_{3}\cdot\u
\]
holds. Substituting this identity into
\eqref{eq-limsup-rhs} gives precisely \eqref{eq-limsup-key}.

\noindent
\textbf{Minty's monotonicity argument.}
For $p\ge2$ the map $\mathbf G\mapsto A(\mathbf G)$ is monotone, hemicontinuous and bounded from
$L^{p}(0,T;\mathbf L^{p})$ into its dual, so for arbitrary
$\mathbf G\in L^{p}(0,T;\mathbf L^{p}(\mathbb D))^{3\times3}$,
\[
\mathcal I_{h,\Delta t}:=\int_{0}^{T}\!\!\int_{\mathbb{D}}\big(A_{h,\Delta t}-A(\mathbf G)\big):\big(\nabla \mathring{\u}_{h,\Delta t}-\mathbf G\big)\,dx\,dt\ \ge\ 0 .
\]
Expanding,
\[
\mathcal I_{h,\Delta t}
=\underbrace{\int_{0}^{T}\!\!\int_{\mathbb{D}}A_{h,\Delta t}:\nabla \mathring{\u}_{h,\Delta t}}_{\text{(i)}}
-\underbrace{\int_{0}^{T}\!\!\int_{\mathbb{D}}A_{h,\Delta t}:\mathbf G}_{\text{(ii)}}
-\underbrace{\int_{0}^{T}\!\!\int_{\mathbb{D}}A(\mathbf G):\big(\nabla \mathring{\u}_{h,\Delta t}-\mathbf G\big)}_{\text{(iii)}} .
\]
Terms (ii) and (iii) pass to the limit by the weak convergences
$A_{h,\Delta t}\rightharpoonup\bar A$ in $L^{p'}$ and
$\nabla\mathring{\u}_{h,\Delta t}\rightharpoonup\nabla\u$ in $L^{p}$, one factor being
fixed in each. Term (i) contains two weakly convergent factors, so \eqref{eq-limsup-key} is essential. Taking $\limsup$ we obtain
\begin{equation}\label{eq-minty-ineq}
\int_{0}^{T}\!\!\int_{\mathbb{D}}\big(\bar A-A(\mathbf G)\big):\big(\nabla\u-\mathbf G\big)\,dx\,dt\ \ge\ 0
\qquad\forall\,\mathbf G\in L^{p}(0,T;\mathbf L^{p}(\mathbb D))^{3\times3}.
\end{equation}
Choosing $\mathbf G=\nabla\u+\varepsilon\Phi$ with
$\Phi\in L^{p}(0,T;\mathbf L^{p}(\mathbb D))^{3\times3}$ and $\varepsilon>0$,
\eqref{eq-minty-ineq} becomes
\[
-\varepsilon\int_{0}^{T}\!\!\int_{\mathbb{D}}\big(\bar A-A(\nabla\u+\varepsilon\Phi)\big):\Phi\,dx\,dt\ \ge\ 0 .
\]
Dividing by $\varepsilon>0$ and letting $\varepsilon\to0^{+}$, hemicontinuity of $A$ gives
$\int_{0}^{T}\!\int_{\mathbb{D}}(\bar A-A(\nabla\u)):\Phi\le 0$, and replacing $\Phi$ by
$-\Phi$ gives equality. Since $\Phi$ is arbitrary, \eqref{eq-minty-goal} follows.

\paragraph{The data and the buoyancy force.}
For the forcing term,

\begin{align}
  \nonumber & \int_{0}^{T}\int_{\mathbb{D}}(\mathring{\mathbf{f}}_{\Delta t}\cdot\mathbf{v}_{h}-\mathbf{f}\cdot \v)\xi(t) dx dt\\
  \nonumber&=  \int_{0}^{T}\int_{\mathbb{D}}\left((\mathring{\mathbf{f}}_{\Delta t}-\mathbf{f})\cdot\mathbf{v}_{h}+\mathbf{f}\cdot (\v_h- \v)\right)\xi(t) dx dt\\
  &\le C\|\mathring{\mathbf{f}}_{\Delta t}-\mathbf{f}\|_{L^2(0,T;\mathbf{L}^{2}(\mathbb{D}))}\|\v_{h}\xi(t)\|_{L^2(0,T;\mathbf{L}^{2}(\mathbb{D}))}+\|\mathbf{f}\|_{L^2(0,T;\mathbf{L}^{2}(\mathbb{D}))}\|(\v_h-\v)\xi(t)\|_{L^2(0,T;\mathbf{L}^{2}(\mathbb{D}))}\longrightarrow 0,
\end{align}
as $h,\Delta t\rightarrow 0$. For the buoyancy force,
\begin{align}
\nonumber&\int_{0}^{T}\int_{\mathbb{D}}(\breve{\theta}_{h,\Delta t}\, e_3\cdot\mathbf{v}_{h}-\theta\, e_3\cdot\mathbf{v} )\xi(t)dx dt\\
\nonumber&=\int_{0}^{T}\int_{\mathbb{D}}\big((\breve{\theta}_{h,\Delta t}-\theta)\, e_3\cdot\mathbf{v}_{h}+\theta\, e_3\cdot(\v_{h}-\mathbf{v})\big)\xi(t)dx dt\\
&\le C \|\breve{\theta}_{h,\Delta t}-\theta\|_{L^2(0,T;{L}^{2}(\mathbb{D}))}\|\v_{h}\xi(t)\|_{L^2(0,T;\mathbf{L}^{2}(\mathbb{D}))}+\|\theta\|_{L^2(0,T;{L}^{2}(\mathbb{D}))}\|(\v_h-\v)\xi(t)\|_{L^2(0,T;\mathbf{L}^{2}(\mathbb{D}))}\longrightarrow 0,
\end{align}
using the strong convergence $\breve\theta_{h,\Delta t}\to\theta$ in
$L^{2}(0,T;L^{2}(\mathbb D))$ from Step~1. 

\medskip
\noindent
\textbf{Step 3: the induction equation.}\, We now show that $\B$ satisfies \eqref{2nd weak equation}. By \eqref{eq-FEM-2}, the pair $(\u_{h,\Delta t},\B_{h,\Delta t})$ satisfies
\begin{align}\label{eq-2-convergence}
  \nonumber  &\int_{0}^{T}\int_{\mathbb{D}}\left(\partial_{t} \B_{h,\Delta t}\cdot\mathbf{w}_{h}\right) dxdt+\kappa\int_{0}^{T}\int_{\mathbb{D}}\left( \nabla \times \mathring{\B}_{h, \Delta t}\cdot \nabla \times\mathbf{w}_{h}\right)dx dt\\&\qquad+\int_{0}^{T}\int_{\mathbb{D}}\left(\mathring{\u}_{h, \Delta t} \times \breve{\B}_{h, \Delta t}\cdot(\nabla \times\mathbf{w}_{h})\right) dxdt=\int_{0}^{T}\int_{\mathbb{D}}\left(\mathring{\mathbf{g}}_{\Delta t}\cdot\mathbf{w}_{h}\right) dxdt.
\end{align}
For any $\mathbf{w}\in\mathcal{G}$ we may choose $\mathbf{w}_{h}\in \mathcal{W}_{0h}$ with
$\mathbf{w}_{h}\rightarrow \mathbf{w}$ in $\mathbf H(\operatorname{curl};\mathbb{D})$, and in
addition $\nabla\times\mathbf w_{h}\to\nabla\times\mathbf w$ in $\mathbf L^{q}(\mathbb D)$ for
every $2\le q<6$. This is possible precisely because $\mathcal W_{h}$ carries no
tangential boundary condition: elements of $\mathcal G$ satisfy only
$\mathbf w\cdot\mathbf n=0$, and $\mathcal G$ is dense in $\mathcal K$, the weak limit class of
$\mathcal W_{0h}$ identified in Lemma~\ref{Lemma 1}.\\
For the time-derivative term, integration by parts in time gives
\begin{align}
\nonumber&\left|\int_{0}^{T}\int_{\mathbb{D}}\left(\partial_{t} \B_{h,\Delta t}\cdot\mathbf{w}_{h}-\partial_{t}\B\cdot\mathbf{w}\right)\xi(t) dx dt\right|\\
  \nonumber &=\left|\int_{0}^{T}\left[(-\B_{h,\Delta t},\partial_{t}(\mathbf{w}_{h}\xi(t)))+(\B,\partial_{t}(\mathbf{w}\xi(t)))-(\B_{h}^{0},\mathbf{w}_{h}\xi(0))+(\B^{0},\mathbf{w}\xi(0))\right]dt\right|\\
 \nonumber &\le C\|\B_{h,\Delta t}-\B\|_{L^2(0,T;\mathbf{L}^{2}(\mathbb{D}))}\|\partial_{t}(\mathbf{w}_{h}\xi(t))\|_{L^2(0,T;\mathbf{L}^{2}(\mathbb{D}))}\\&\qquad+C\|\B\|_{L^2(0,T;\mathbf{L}^{2}(\mathbb{D}))}\|\partial_{t}\bigl((\mathbf{w}_{h}-\mathbf{w})\xi(t)\bigr)\|_{L^2(0,T;\mathbf{L}^{2}(\mathbb{D}))}\longrightarrow 0 \qquad \mbox{as}\qquad h,\Delta t\longrightarrow 0.
    \end{align}
    For the diffusive term,
    \begin{align}
       \nonumber &\left|\kappa\int_{0}^{T}\int_{\mathbb{D}}\left[(\nabla \times \mathring{\B}_{h, \Delta t}\cdot \nabla \times\mathbf{w}_{h})-(\nabla \times\B\cdot\nabla \times\mathbf{w})\right]\xi(t)dx dt\right|\\
       \nonumber& = \kappa\left|\int_{0}^{T}\int_{\mathbb{D}}\left[\bigl(\nabla \times (\mathring{\B}_{h, \Delta t}-\B)\bigr)\cdot \nabla \times\mathbf{w}_h+(\nabla \times \mathring{\B}_{h, \Delta t})\cdot \nabla \times(\mathbf{w}_{h}-\mathbf{w})\right]\xi(t)dx dt\right|\\
       \nonumber&\le C\left|\int_{0}^{T}\int_{\mathbb{D}}\bigl(\nabla \times (\mathring{\B}_{h, \Delta t}-\B)\bigr)\cdot \left((\nabla \times\mathbf{w}_h)\,\xi(t)\right)dx dt\right|\\  &\qquad+ C\|\nabla \times\mathring{\B}_{h, \Delta t}\|_{L^2(0,T;\mathbf{L}^{2}(\mathbb{D}))}\|(\nabla \times\mathbf{w}_{h}-\nabla \times\mathbf{w})\xi(t)\|_{L^2(0,T;\mathbf{L}^{2}(\mathbb{D}))}\longrightarrow 0 \qquad \mbox{as}\qquad h, \Delta t\rightarrow 0.
    \end{align}
\noindent
Here the first term tends to zero by the weak convergence of
$\nabla\times\mathring\B_{h,\Delta t}$ against the fixed function
$(\nabla\times\mathbf w_h)\xi$, while the second follows from
\eqref{eq-bounds-2} and $\mathbf w_h\to\mathbf w$.
    For the induction coupling we use \eqref{eq-inverse estimate} and \eqref{ine-zbh}:
    \begin{align}
&\left|\int_{0}^{T}\int_{\mathbb{D}}\left[(\mathring{\u}_{h, \Delta t} \times \breve{\B}_{h, \Delta t}\cdot\nabla \times\mathbf{w}_{h})-(\u\times\B\cdot\nabla \times\mathbf{w})\right] \xi(t)dxdt\right| \nonumber\\
&\le\left|\int_{0}^{T}\int_{\mathbb{D}}\left[(\mathring{\u}_{h, \Delta t} \times (\breve{\B}_{h, \Delta t}-\mathcal{Z}(\breve{\B}_{h, \Delta t}))\cdot\nabla \times\mathbf{w}_{h})+(\mathring{\u}_{h, \Delta t} \times\mathcal{Z}(\breve{\B}_{h, \Delta t}))\cdot\nabla \times\mathbf{w}_{h})-(\u\times\B\cdot\nabla \times\mathbf{w})\right]\xi(t) dxdt\right| \nonumber\\
\nonumber& \le Ch^{\varrho+\frac{1}{2}}\|\nabla \times\breve{\B}_{h, \Delta t}\|_{L^2(0,T;\mathbf{L}^{3}(\mathbb{D}))}\|\mathring{\u}_{h, \Delta t}\|_{L^\infty(0,T;\mathbf{L}^{2}(\mathbb{D}))}\|(\nabla \times\mathbf{w}_{h})\xi(t)\|_{L^2(0,T;\mathbf{L}^{6}(\mathbb{D}))} \\
\nonumber&\qquad +\left|\int_{0}^{T}\int_{\mathbb{D}}\left[(\mathring{\u}_{h, \Delta t} \times\mathcal{Z}(\breve{\B}_{h, \Delta t})\cdot\nabla \times\mathbf{w}_{h})-(\u\times\B\cdot\nabla \times\mathbf{w})\right]\xi(t) dxdt\right|\\
\nonumber& \le Ch^{\varrho}\|\nabla \times\breve{\B}_{h, \Delta t}\|_{L^2(0,T;\mathbf{L}^{2}(\mathbb{D}))}\|\mathring{\u}_{h, \Delta t}\|_{L^\infty(0,T;\mathbf{L}^{2}(\mathbb{D}))}\|(\nabla \times\mathbf{w}_{h})\xi(t)\|_{L^2(0,T;\mathbf{L}^{6}(\mathbb{D}))}\\
\nonumber&\qquad +\left|\int_{0}^{T}\int_{\mathbb{D}}\left[(\mathring{\u}_{h, \Delta t} \times\mathcal{Z}(\breve{\B}_{h, \Delta t})\cdot\nabla \times\mathbf{w}_{h})-(\u\times\B\cdot\nabla \times\mathbf{w})\right]\xi(t) dxdt\right|\\
&=:E_{1}+|E_{2}|\rightarrow 0
\end{align}
as $h,\Delta t\rightarrow 0$: $E_{1}\to0$ by the factor $h^{\varrho}$ and the bounds of
Theorem~\ref{bounds}, and $E_{2}\to0$ by the Step~1 strong convergences of
$\mathring{\u}_{h,\Delta t}$ in $L^{2}(0,T;\mathbf L^{4})$ and of
$\mathcal Z(\breve{\B}_{h,\Delta t})$ in $L^{2}(0,T;\mathbf L^{3})$, together with
$\nabla\times\mathbf w_{h}\to\nabla\times\mathbf w$; see also
\cite[Theorem 3]{DING2025116470}.

\noindent
Finally, for the right-hand side of \eqref{eq-2-convergence},
\begin{align}
    \nonumber&\int_{0}^{T}\int_{\mathbb{D}}\left(\mathring{\mathbf{g}}_{\Delta t}\cdot\mathbf{w}_{h}-\mathbf{g}\cdot\mathbf{w}\right) \xi(t)dx dt\\
    \nonumber&=\int_{0}^{T}\int_{\mathbb{D}}\left((\mathring{\mathbf{g}}_{\Delta t}-\mathbf{g})\cdot\mathbf{w}_{h}+\mathbf{g}\cdot(\mathbf{w}_{h}-\mathbf{w})\right)\xi(t) dx dt\\
    &\le C\|\mathring{\mathbf{g}}_{\Delta t}-\mathbf{g}\|_{L^2(0,T;\mathbf{L}^{2}(\mathbb{D}))}\|\mathbf{w}_{h}\xi(t)\|_{L^2(0,T;\mathbf{L}^{2}(\mathbb{D}))}+\|\mathbf{g}\|_{L^2(0,T;\mathbf{L}^{2}(\mathbb{D}))}\|(\mathbf{w}_{h}-\mathbf{w})\xi(t)\|_{L^2(0,T;\mathbf{L}^{2}(\mathbb{D}))}\longrightarrow 0\,,
\end{align}
as $h, \Delta t \longrightarrow 0$.
Hence \eqref{eq-2-convergence} converges to the induction equation.

\medskip
\noindent
\textbf{Step 4: the temperature equation.}\, Finally, by \eqref{eq-FEM-3},
 \begin{align}{\label{eq-3-convergence}}
        \nonumber&\int_{0}^{T}\int_{\mathbb{D}}\partial_{t}\theta_{h,\Delta t}\,\phi_{h}\, dx\, dt+\sigma\int_{0}^{T}\int_{\mathbb{D}}\nabla \mathring{\theta}_{h\Delta t}\cdot\nabla\phi_{h}\,dx\, dt+\int_{0}^{T}a_{2}(\breve{\u}_{h\Delta t},\mathring{\theta}_{h\Delta t},\phi_{h})\, dt\\&\qquad
        =\int_{0}^{T}(\mathring{\u}_{h\Delta t}\cdot e_3,\phi_{h})\, dt+\int_{0}^{T}\int_{\mathbb{D}}\mathring{\psi}_{\Delta t}\,\phi_{h}\,dx\,dt .
        \end{align}
\noindent
For $\phi\in \mathcal{C}_{0}^{\infty}(\mathbb D)$ let $\phi_{h}:=A_{h}\phi\in \mathcal{Y}_{h}$,
with $A_{h}$ the $H^{1}$-orthogonal projection onto $\mathcal{Y}_{h}$, so that
$\phi_{h}\rightarrow \phi$ in $H^{1}_{0}(\mathbb D)$.
\\ \noindent
For the time-derivative term, integration by parts in time gives
\begin{align}
\nonumber&\left|\int_{0}^{T}\int_{\mathbb{D}}\left(\partial_{t} \theta_{h,\Delta t}\cdot\phi_{h}-\partial_{t}\theta\cdot\phi\right)\xi(t) dx dt\right|\\
  \nonumber &=\left|\int_{0}^{T}\left[(-\theta_{h,\Delta t},\partial_{t}(\phi_{h}\xi(t)))+(\theta,\partial_{t}(\phi\xi(t)))-(\theta_{h}^{0},\phi_{h}\xi(0))+(\theta^{0},\phi\xi(0))\right]dt\right|\\
 \nonumber &\le C\|\theta_{h,\Delta t}-\theta\|_{L^2(0,T;{L}^{2}(\mathbb{D}))}\|\partial_{t}(\phi_{h}\xi(t))\|_{L^2(0,T;{L}^{2}(\mathbb{D}))}\\&\qquad+C\|\theta\|_{L^2(0,T;{L}^{2}(\mathbb{D}))}\|\partial_{t}\bigl((\phi_{h}-\phi)\xi(t)\bigr)\|_{L^2(0,T;{L}^{2}(\mathbb{D}))}\longrightarrow 0 \qquad \mbox{as}\qquad h,\Delta t\longrightarrow 0.
    \end{align}
      For the diffusive term,
    \begin{align}
       \nonumber &\left|\sigma\int_{0}^{T}\int_{\mathbb{D}}\left[(\nabla \mathring{\theta}_{h, \Delta t}\cdot \nabla\phi_{h})-(\nabla \theta\cdot\nabla\phi)\right]\xi(t)dx dt\right|\\
       \nonumber& = \sigma\left|\int_{0}^{T}\int_{\mathbb{D}}\left[\nabla(\mathring{\theta}_{h, \Delta t}-\theta)\cdot \nabla\phi_h+\nabla\mathring{\theta}_{h, \Delta t}\cdot \nabla(\phi_{h}-\phi)\right]\xi(t)dx dt\right|\\
       \nonumber&\le C\left|\int_{0}^{T}\int_{\mathbb{D}}\nabla(\mathring{\theta}_{h, \Delta t}-\theta)\cdot \bigl(\nabla\phi_h\,\xi(t)\bigr)dx dt\right|\\  &\qquad+ C\|\nabla \mathring{\theta}_{h, \Delta t}\|_{L^2(0,T;{L}^{2}(\mathbb{D}))}\|(\nabla\phi_{h}-\nabla \phi)\xi(t)\|_{L^2(0,T;{L}^{2}(\mathbb{D}))}\longrightarrow 0 \qquad \mbox{as}\qquad h, \Delta t\rightarrow 0,
    \end{align}
\noindent
by weak-strong convergence: the first term uses the weak convergence of
$\nabla\mathring\theta_{h,\Delta t}$ against the fixed function
$\nabla\phi_h\,\xi$, and the second follows from \eqref{eq-bounds-2} and
$\phi_h\to\phi$ in $H_0^1(\mathbb D)$.
    \noindent
    For the thermal transport term, \cite[Theorem 4.7]{ding2022convergence},
    \begin{align}
        \nonumber&\left|\int_{0}^{T}\int_{\mathbb{D}}\left(a_{2}(\breve{\u}_{h\Delta t},\mathring{\theta}_{h\Delta t},\phi_{h})-a_{2}(\u,\theta,\phi)\right) \xi(t)dt\right|\\
      \nonumber  &\le C\|\breve{\u}_{h,\Delta t}-\u\|_{L^2(0,T;\mathbf{L}^{4}(\mathbb{D}))}\|\nabla\mathring{\theta}_{h,\Delta t}\|_{L^2(0,T;\mathbf{L}^{2}(\mathbb{D}))}\|\phi_{h}\xi(t)\|_{L^\infty(0,T;{L}^{4}(\mathbb{D}))}\\
      \nonumber&\qquad+C\|\u\|_{L^2(0,T;\mathbf{L}^{4}(\mathbb{D}))}\|\nabla\theta\|_{L^2(0,T;\mathbf{L}^{2}(\mathbb{D}))}\|(\phi_{h}-\phi)\xi(t)\|_{L^\infty(0,T;{L}^{4}(\mathbb{D}))}\\
      \nonumber &\qquad +C\|\breve{\u}_{h,\Delta t}-\u\|_{L^2(0,T;\mathbf{L}^{4}(\mathbb{D}))}\|\nabla\phi_{h}\xi(t)\|_{L^\infty(0,T;\mathbf{L}^{2}(\mathbb{D}))}\|\mathring{\theta}_{h,\Delta t}\|_{L^2(0,T;{L}^{4}(\mathbb{D}))}\\
      \nonumber &\qquad +C\|\u\|_{L^2(0,T;\mathbf{L}^{4}(\mathbb{D}))}\|\nabla\phi_{h}\xi(t)\|_{L^\infty(0,T;\mathbf{L}^{2}(\mathbb{D}))}\|\mathring{\theta}_{h,\Delta t}-\theta\|_{L^2(0,T;{L}^{4}(\mathbb{D}))}\\
      \nonumber &\qquad +C\|\u\|_{L^2(0,T;\mathbf{L}^{4}(\mathbb{D}))}\|\nabla(\phi_{h}-\phi)\xi(t)\|_{L^\infty(0,T;\mathbf{L}^{2}(\mathbb{D}))}\|\mathring{\theta}_{h,\Delta t}\|_{L^2(0,T;{L}^{4}(\mathbb{D}))}\\
&\qquad+C\left|\int_{0}^{T}\int_{\mathbb{D}}\bigl(\u\cdot\nabla(\mathring{\theta}_{h,\Delta t}-\theta)\bigr)\phi_{h}\,\xi(t)\, dx dt\right|\longrightarrow 0 \qquad \mbox{as}\qquad h, \Delta t\rightarrow 0.
    \end{align}
\noindent
Only weak convergence of $\nabla\mathring\theta_{h,\Delta t}$ is available, so the last term
cannot be estimated directly. Since $\operatorname{div}\u=0$ by Step~1 and
$\phi_{h}\in\mathcal Y_{h}\subset H^{1}_{0}(\mathbb D)$, we integrate by parts to move the
derivative onto the test function:
\[
\int_{\mathbb{D}}\bigl(\u\cdot\nabla(\mathring{\theta}_{h,\Delta t}-\theta)\bigr)\phi_{h}\,dx
=-\int_{\mathbb{D}}\bigl(\u\cdot\nabla\phi_{h}\bigr)(\mathring{\theta}_{h,\Delta t}-\theta)\,dx ,
\]
whence
\[
\left|\int_{0}^{T}\!\!\int_{\mathbb{D}}\bigl(\u\cdot\nabla(\mathring{\theta}_{h,\Delta t}-\theta)\bigr)\phi_{h}\,\xi\,dx\,dt\right|
\le C\|\u\|_{L^{2}(0,T;\mathbf L^{4})}\|\nabla\phi_{h}\,\xi\|_{L^{\infty}(0,T;\mathbf L^{2})}
\|\mathring{\theta}_{h,\Delta t}-\theta\|_{L^{2}(0,T;L^{4})}\longrightarrow0,
\]
by the strong convergence $\mathring\theta_{h,\Delta t}\to\theta$ in $L^{2}(0,T;L^{q})$,
$1\le q<6$, established in Step~1.
\noindent
Finally, for the two terms on the right of \eqref{eq-3-convergence}: by Step~1
$\mathring{\u}_{h\Delta t}\to\u$ strongly in $L^{2}(0,T;\mathbf L^{2})$ and
$\phi_{h}\to\phi$ in $H^{1}_{0}(\mathbb D)$, hence
\begin{align*}
&\left|\int_{0}^{T}\big((\mathring{\u}_{h\Delta t}\cdot e_{3},\phi_{h})-(\u\cdot e_{3},\phi)\big)\xi(t)\,dt\right|\\
&\qquad\le
C\|\mathring{\u}_{h\Delta t}-\u\|_{L^{2}(0,T;\mathbf L^{2})}\|\phi_{h}\xi\|_{L^{2}(0,T;L^{2})}
+C\|\u\|_{L^{2}(0,T;\mathbf L^{2})}\|(\phi_{h}-\phi)\xi\|_{L^{2}(0,T;L^{2})}\longrightarrow0,
\end{align*}
and, in the same way, $|\int_{0}^{T}((\mathring{\psi}_{\Delta t},\phi_{h})-(\psi,\phi))\xi\,dt|\to0$
using $\mathring{\psi}_{\Delta t}\to\psi$ in $L^{2}(0,T;L^{2}(\mathbb D))$.
Collecting Steps~1-4, the limit $(\u,\B,\theta)$ satisfies
\eqref{1st weak form}-\eqref{3rd weak equation} for all test functions in
$\mathcal J\times\mathcal C^{\infty}_{0}\times\mathcal G$, hence by density for all
$(\mathbf v,\phi,\mathbf C)\in\mathcal X_{0}\times\mathcal Y_{0}\times\mathcal K$  note that
$\mathcal G$ is dense in $\mathcal K$. The initial
data are attained through the terms $\xi(0)$ retained in the integration by parts in time. This
completes the proof.
\end{proof}

\medskip

\section{Error analysis}\label{sec-error}

We now assume sufficient regularity of the exact solution and derive a~priori error
estimates in space and time for the velocity, magnetic induction, and temperature. The
estimate is mesh-independent: the admissible time step depends only on the data and on $T$.
\\ \noindent
Three points are essential. First, because $a_{0}$ is nonlinear, we use the linear Stokes
projection and retain the nonlinearity as a consistency term. Second, the velocity
dissipation is $\|\nabla E_{\u}\|^{p}_{\mathbf L^{p}}$; hence the pair $(p,p')$ is applied
only to consistency terms involving the exact solution. This motivates
\eqref{eq-exactly-divfree} and the $L^{\infty}$-bounds of
Lemma~\ref{eq-error-estimate bounds}. Third, the temporal consistency errors are estimated
through integral representations, avoiding terms of order $h^{2l}/\Delta t$.

\subsection{Regularity assumption and statement of the main result}\label{subsec-assumption}

\begin{assumption}
\label{Assumption}
    Suppose that $(\u,\B,\theta)$ satisfies the following regularity:
  \begin{center}
      $\u\in L^\infty(0,T;\mathbf{W}^{s+1,p}(\mathbb{D}))$,\qquad $\B\in L^\infty(0,T;\mathbf{W}^{s,2}(\mathbb{D}))$, \qquad $\theta\in L^\infty(0,T;{W}^{s+1,2}(\mathbb{D}))$ \\

       $\nabla\u,\ \nabla\times\B,\ \nabla\theta\in L^\infty(0,T;\mathbf{L}^{\infty}(\mathbb{D}))$,\\
       $\u_{t}\in L^{2}(0,T;\mathbf{W}^{1+l,2}(\mathbb{D}))$,\qquad
           $\B_{t}\in L^{2}(0,T;\mathbf{W}^{l,2}(\mathbb{D}))$,\qquad
           $\theta_{t}\in L^{2}(0,T;{W}^{1+l,2}(\mathbb{D}))$,\\
    $\partial_{t}\mathbf f,\ \partial_{t}\mathbf g\in L^{2}(0,T;\L^{2}(\mathbb{D}))$,\qquad
           $\partial_{t}\psi\in L^{2}(0,T;L^{2}(\mathbb{D}))$.\\
             $\u_{tt}\in L^2(0,T;\L^2(\mathbb{D}))$,\qquad $\B_{tt}\in L^2(0,T;\L^2(\mathbb{D}))$,\qquad $\theta_{tt}\in L^2(0,T;L^2(\mathbb{D}))$,\\
  \end{center} 
where $s>\frac12$ depends on $\mathbb D$, $l=\min\{k,s\}$, and the strict inequality is
required in \eqref{eq-sobolev-Linfty}-\eqref{eq-sobolev-W13}.
\end{assumption}

{\noindent
Since the discrete data \eqref{eq-data-averages} are time averages while the continuous
equations are evaluated at $t=t_n$, the resulting consistency error satisfies
\begin{equation}\label{eq-data-consistency}
\|\mathbf f(t_{n})-\mathbf f^{n}\|_{\L^{2}}
=\Big\|\frac{1}{\Delta t}\int_{t_{n-1}}^{t_{n}}\!\!\int_{s}^{t_{n}}\partial_{t}\mathbf f(\tau)\,d\tau\,ds\Big\|_{\L^{2}}
\le(\Delta t)^{\frac12}\|\partial_{t}\mathbf f\|_{L^{2}(t_{n-1},t_{n};\L^{2})},
\end{equation}
and similarly for $\mathbf g$ and $\psi$; these terms appear in
\eqref{error-u-truncation}-\eqref{error-theta-trucation}. The consistency errors
$M^{n}_{h,i}$ are estimated through
$d_{t}(I-\mathcal P_{h})\u^{n}=\Delta t^{-1}\!\int_{t_{n-1}}^{t_{n}}(I-\mathcal P_{h})\partial_{t}\u\,dt$,
rather than endpoint values, which would generate the unwanted factor
$h^{2l}/\Delta t$. The $L^{\infty}$-bounds keep all error factors in
$\mathbf L^{2}$, allowing quadratic Young inequalities throughout.}
\noindent
The main result is the following.
\begin{theorem}\label{Main theorem of error}
    Let Assumption~\ref{Assumption} hold, let the velocity pair satisfy
    \eqref{eq-exactly-divfree}, and let the initial approximations
    $\mathbf{U}^{0},\mathfrak{B}^{0},\Theta^{0}$ satisfy
    $$
    \|\u^{0}-\mathbf{U}^{0}\|_{\L^2}+\|\B^{0}-\mathfrak{B}^{0}\|_{\L^2}+\|\theta^{0}-\Theta^{0}\|_{L^2}\le C^{*}h^{l}\,,
    $$
    with $l=\min\{k,s\}$. Then there exists a positive constant $\Delta t_{0}$ such that when $\Delta t\le \Delta t_0$ the finite element system \eqref{eq-FEM-1}-\eqref{eq-FEM-3} admits a unique solution $(\mathbf{U}^{n},\mathfrak{B}^{n},\Theta^n)$, and for every $1\le m\le N$ the projected errors
    $E^{n}_{\u},E^{n}_{\B},E^{n}_{\theta}$ defined in \eqref{eq-def-projected-errors} satisfy
    \begin{align}\label{eq-main-error}
\nonumber
&\|E^{m}_{\u}\|^{2}_{\L^2}+\|E^{m}_{\B}\|^{2}_{\L^2}+\|E^{m}_{\theta}\|^{2}_{L^2}\\
&\qquad+(\Delta t)\sum_{n=1}^{m}\Big\{
\tfrac{\nu}{2}2^{2-p}\|\nabla E^{n}_{\u}\|^{p}_{\L^p}
+\tfrac{\kappa}{2}\|\nabla\times E^{n}_{\B}\|^{2}_{\L^2}
+\tfrac{\sigma}{2}\|\nabla E^{n}_{\theta}\|^{2}_{L^2}\Big\}
\ \le\ C^*\Big\{(\Delta t)^2+h^{\frac{lp}{p-1}}\Big\}.
\end{align}
Consequently, by the triangle inequality and
\eqref{eq-error approximation}-\eqref{eq-error approximation Lp},
\begin{align}\label{eq-main-error-total}
\max_{1\le m\le N}\Big(\|\u^{m}-\mathbf{U}^{m}\|_{\L^2}+\|\B^{m}-\mathfrak{B}^{m}\|_{\L^2}+\|\theta^{m}-\Theta^{m}\|_{L^2}\Big)
\le C^{*}\Big\{\Delta t+h^{\frac{lp}{2(p-1)}}\Big\},
\end{align}
 where $C^*$ is independent of $\Delta t$ and the mesh size $h$.
\end{theorem}

\begin{remark}\label{rem-rates}
Three points are noteworthy. (i) The spatial exponent
$\frac{lp}{p-1}=lp'$ comes from the $p$-Laplacian consistency term
$h^{lp'}\|\u^{n}\|^{p'}_{l+1,p}$. For $p=2$, $p'=2$, and
\eqref{eq-main-error} reduces to $C^{*}\{(\Delta t)^{2}+h^{2l}\}$. (ii) The gradient term
controls the \emph{$p$-th power} of the error:
$\|\nabla E_{\u}\|_{L^{p}(0,T;\mathbf L^{p})}\le
C^{*}\{(\Delta t)^{2}+h^{lp'}\}^{1/p}$. Thus, for $p=2$, $l=1$, the spatial rate is first
order. (iii) $\Delta t_{0}$ depends only on the data and on $T$, not on $h$.
\end{remark}
\subsection{Projections and their approximation properties}\label{subsec-projections}

For $\u\in\mathcal X_{0}$ and $\theta\in\mathcal Y_{0}$, define the linear projections
$\mathcal P_h\u\in\mathcal X_{0h}$ and $\mathcal R_h\theta\in\mathcal Y_h$ by
\begin{align}
    &(\nabla\mathcal{P}_{h}\u,\nabla \v_{h})=(\nabla\u,\nabla\v_{h})\qquad\forall\,\v_h\in\mathcal X_{0h},\\
    &\sigma(\nabla\mathcal{R}_{h}\theta,\nabla\phi_{h})=\sigma(\nabla\theta,\nabla\phi_{h})\qquad{\forall\,\phi_h\in\mathcal Y_{h}}.
\end{align}
{\noindent
Here $\mathcal P_h$ is the linear discrete Stokes projection and $\mathcal R_h$ the Ritz
projection. We do not use the nonlinear alternative
$(a_{0}(\mathcal P_{h}\u),\nabla\v_{h})=(a_{0}(\u),\nabla\v_{h})$, since the
$L^{2}$-estimate \eqref{eq-error approximation} is available for the linear Stokes projection
\cite{GiraultRaviart1986}, whereas the nonlinear projection is naturally treated in
quasi-norms \cite{BarrettLiu1993,DieningRuzicka2007}. The nonlinearity therefore appears as
the consistency term
$(a_{0}(\u^{n})-a_{0}(\mathcal P_{h}\u^{n}),\nabla\v_{h})$ in $R^{n}_{h,1}$.}
For $\B\in\mathcal K$, define $\mathcal F_h\B\in\mathcal W_{0h}$ by
\begin{align*}
   \kappa(\nabla \times\mathcal{F}_{h}\B,\nabla \times\C_{h})&=\kappa(\nabla \times\B,\nabla \times\C_{h}),\\
   (\mathcal{F}_{h}\B,\nabla {s_{h}})&=(\B,\nabla {s_{h}}).
\end{align*}
We use the following bounds (see \cite{GiraultRaviart1986}):
\begin{align}\label{eq-error approximation}
    \|\u-\mathcal{P}_{h}\u\|_{\L^{2}}+h\|\nabla(\u-\mathcal{P}_{h}\u)\|_{\L^{2}}&\le C_{e}h^{1+l}\|\u\|_{1+l,2},\\
    \|\B-\mathcal{F}_{h}\B\|_{\L^{2}}+\|\nabla \times(\B-\mathcal{F}_{h}\B)\|_{\L^{2}}&\le C_{e}h^{l}\left(\|\B\|_{l,2}+\|\nabla \times\B\|_{l,2}\right),\\
    \|\theta-\mathcal{R}_{h}\theta\|_{1,2}&\le C_{e}h^l\|\theta\|_{1+l,2},\\
    \label{eq-error approximation Lp}
\|\nabla(\u-\mathcal P_{h}\u)\|_{\L^{p}}&\le C_{e}h^{l}\|\u\|_{1+l,p},
\qquad
\|\nabla\mathcal P_{h}\u\|_{\L^{2}}\le C\|\nabla\u\|_{\L^{2}},
\end{align}
    where $l=\min\{k,s\}$, $k\ge1$, and $s>\frac12$.
{Estimate \eqref{eq-error approximation Lp} follows from the
$\mathbf W^{1,p}$-stability of $\mathcal P_h$ on quasi-uniform meshes
\cite{GiraultRaviart1986} and is used only for the $p$-Laplacian consistency term.}

\subsection{The error equations}\label{subsec-errorequations}

    Let $\u^{n},\B^{n},\theta^{n}$ denote the exact solution at $t_n$, and
    $\mathbf U^{n},\mathfrak B^{n},\Theta^{n}$ the discrete solution. At $t=t_n$,
    \eqref{1st weak form}-\eqref{3rd weak equation} give
    \begin{align}
       \label{eq-u-error-continuous} &\left(\u^{n}_{t},\v\right)+{(a_{0}(\u^{n}),\nabla\v)}+a_{1}(\u^{n},\u^{n},\v)+2\left(\boldsymbol{\omega}_r\times\u^n,\v\right)+(\B^{n} \times (\nabla \times \B^{n}), \v)={(\mathbf{f}(t_{n}), \v)}+{(\theta^{n} e_3,\v)},\\ 
    \label{eq-B-error-continuous}&\left( \B^{n}_{t}, \C\right)+\kappa\left( \nabla \times \B^{n}, \nabla \times \C\right)-(\u^{n} \times \B^{n}, \nabla \times \C)={(\mathbf{g}(t_{n}), \C)},\\
  \label{eq-theta-error-continuous} & {(\theta^{n}_{t},\phi)}+\sigma(\nabla\theta^{n},\nabla \phi)+{a_{2}(\u^{n},\theta^{n},\phi)}=(\u^n\cdot e_3,\phi)+{(\psi(t_{n}),\phi)}.
    \end{align}
     Taking $\v=\v_h$, $\C=\C_h$, and $\phi=\phi_h$ gives
    \begin{align}
          &\left(\u^{n}_{t},\v_h\right)+{(a_{0}(\u^{n}),\nabla\v_h)}+a_{1}(\u^{n},\u^{n},\v_h)+2\left(\boldsymbol{\omega}_r\times\u^n,\v_h\right)+(\B^{n} \times (\nabla \times \B^{n}), \v_h)=(\mathbf{f}(t_{n}), \v_h)+(\theta^{n} e_3,\v_h),\\ 
    &\left( \B^{n}_{t}, \C_h\right)+\kappa\left( \nabla \times \B^{n}, \nabla \times \C_h\right)-(\u^{n} \times \B^{n}, \nabla \times \C_h)=(\mathbf{g}(t_{n}), \C_h),\\
   & (\theta^{n}_{t},\phi_h)+\sigma(\nabla\theta^{n},\nabla \phi_h)+a_{2}(\u^{n},\theta^{n},\phi_h)=(\u^n\cdot e_3,\phi_h)+(\psi(t_{n}),\phi_h).
    \end{align}
\noindent
The discrete scheme gives
\begin{align}
\label{eq-u-error-scheme} &\left(d_{t} \mathbf{U}^{n}, \mathbf{v}_{h}\right)+a_{1}\left(\mathbf{U}^{n-1}, \mathbf{U}^{n}, \mathbf{v}_{h}\right)+(a_{0}(\mathbf{U}^{n}),\nabla\v_{h})
+2\left(\boldsymbol{\omega}_r\times\mathbf U^n,\v_h\right)+\left(\mathfrak{B}^{n-1} \times (\nabla \times \mathfrak{B}^{n}), \mathbf{v}_{h}\right)=\left(\mathbf{f}^{n}, \mathbf{v}_{h}\right)+\left(\Theta^{n-1}e_{3},\v_{h}\right),  \\ 
&\label{eq-B-error-scheme}\left(d_{t} \mathfrak{B}^{n}, \mathbf{w}_{h}\right)+\kappa\left(\nabla \times \mathfrak{B}^{n},\nabla \times \mathbf{w}_{h}\right)-\left(\mathbf{U}^{n} \times \mathfrak{B}^{n-1}, \nabla \times\mathbf{w}_{h}\right)=\left(\mathbf{g}^{{n}}, \mathbf{w}_{h}\right), 
\\& \label{eq-theta-error-scheme}(d_{t}\Theta^{n},\phi_{h})+\sigma(\nabla \Theta^{n},\nabla\phi_{h})+a_{2}(\mathbf{U}^{n-1},\Theta^{n},\phi_{h})=(\mathbf{U}^{n}\cdot e_3,\phi_{h})+(\psi^{{n}},\phi_{h}).
\end{align}

 \noindent   
Subtracting the continuous and discrete momentum equations gives
\begin{align}
    \nonumber  &\left(d_{t} \mathbf{U}^{n}, \mathbf{v}_{h}\right)-\left( \u^{n}_t, \mathbf{v}_{h}\right)=a_{1}\left(\u^{n}, \u^{n}, \mathbf{v}_{h}\right)-a_{1}\left(\mathbf{U}^{n-1}, \mathbf{U}^{n}, \mathbf{v}_{h}\right)+{(a_{0}(\u^{n}),\nabla\v_{h})-(a_{0}(\mathbf{U}^{n}),\nabla\v_{h}) }\\  
\nonumber\qquad &
+\left(\B^{n} \times (\nabla \times \B^{n}), \mathbf{v}_{h}\right)-\left(\mathfrak{B}^{n-1} \times (\nabla \times \mathfrak{B}^{n}), \mathbf{v}_{h}\right)+\left(\Theta^{n-1}e_{3},\v_{h}\right)-\left(\theta^{n-1}e_{3},\v_{h}\right)\\
\qquad&+2(\boldsymbol{\omega}_r\times\u^n,\v_h)-2\left(\boldsymbol{\omega}_r\times\mathbf U^n,\v_h\right),
\end{align}
Similarly, for the magnetic equation,
\begin{align}
&\left(d_{t} \mathfrak{B}^n-\B^{n}_{t}, {\C}_{h}\right)+\kappa\left(\nabla \times (\mathfrak{B}^n-\B^{n}),\nabla \times {\C}_{h}\right)=(\mathbf{U}^{n} \times \mathfrak{B}^{n-1}, \nabla \times{\C}_{h})-(\u^{n}\times\B^{n},\nabla \times\C_{h}),
\end{align}
and for the temperature equation,
\begin{align}
& (d_{t}\Theta^{n}-\theta^{n}_t,\phi_{h})+\sigma(\nabla (\Theta^{n}-\theta^n),\nabla\phi_{h})=a_{2}({\u^{n}},\theta^{n},\phi_{h})-a_{2}(\mathbf{U}^{n-1},\Theta^{n},\phi_{h})+(\mathbf{U}^{n}\cdot e_3,\phi_{h})-(\u^{n}\cdot e_3,\phi_{h}).
\end{align}
Set
\begin{equation}\label{eq-def-projected-errors}
(E^{n}_{\u},E^{n}_{\B},E^{n}_{\theta}) :=(\mathbf{U}^{n}-\mathcal{P}_{h}\u^{n},\mathfrak{B}^n-\mathcal{F}_{h}\B^{n},\Theta^n-\mathcal{R}_{h}\theta^{n}).
\end{equation}
The error equations are
    \begin{align}
       \label{error eq-1} &(d_{t}E^{n}_{\u}, \v_{h})+{(a_{0}(\mathbf{U}^{n}),\nabla\v_{h})-(a_{0}(\mathcal{P}_{h}\u^{n}),\nabla\v_{h})}+{2\left(\boldsymbol{\omega}_r\times E^n_\u,\v_h\right)}=(M_{h,1}^{n}+R_{h,1}^{n},\v_{h}),\\\label{error eq-2}
        &(d_{t}E^{n}_{\B}, \C_{h})+\kappa\left(\nabla \times E^{n}_{\B},\nabla \times\C_{h}\right)=(M_{h,2}^{n}+R_{h,2}^{n},\C_{h}),\\\label{error eq-3}
        &(d_{t}E^{n}_{\theta}, \phi_{h})+\sigma\left(\nabla E^{n}_{\theta},\nabla \phi_{h}\right)=(M_{h,3}^{n}+R_{h,3}^{n},\phi_{h}),
        \end{align}
where
\begin{align}
  & (M_{h,1}^{n},\v_h)=(\partial_{t}\u^{n}-d_{t}\mathcal{P}_{h}\u^{n},\v_{h}), \\
  & (M_{h,2}^{n},{\C_h})=(\partial_{t}\B^{n}-d_{t}\mathcal{F}_{h}\B^{n},\C_{h}), \\
 & (M_{h,3}^{n},{\phi_h})=(\partial_{t}\theta^{n}-d_{t}\mathcal{R}_{h}\theta^{n},{\phi_{h}}), \\  
 \nonumber\label{error-u-truncation}&(R_{h,1}^{n},\v_{h})=a_{1}(\u^{n},\u^{n},\v_{h})-a_{1}(\mathbf{U}^{n-1}, \mathbf{U}^{n}, \mathbf{v}_{h})+(a_{0}(\u^{n}),\nabla\v_{h})-(a_{0}(\mathcal{P}_{h}\u^{n}),\nabla\v_{h})\\ \nonumber&\qquad\qquad+\left(\mathbf{b}^{n} \times \nabla \times \mathbf{b}^{n}, \mathbf{v}_{h}\right)-\left(\mathfrak{B}^{n-1} \times \nabla \times \mathfrak{B}^n, \mathbf{v}_{h}\right)+\left(\Theta^{n-1}e_{3},\v_{h}\right)-\left(\theta^{n-1}e_{3},\v_{h}\right)+\left(\mathbf f^{n}-\mathbf f(t_{n}),\v_{h}\right)\\
 &\qquad\qquad+{2\left(\boldsymbol{\omega}_r\times(\u^n-\mathcal{P}_h\u^n),\v_h\right)},\\
& \label{error-B-trucation}(R_{h,2}^{n},\C_{h})=\left(\mathbf{U}^{n}\times\mathfrak{B}^{n-1},\nabla \times\C_{h}\right)-\left(\u^{n}\times\B^{n},\nabla \times\C_{h}\right)+\kappa\left(\nabla\times(\B^{n}-\mathcal{F}_{h}\B^{n}),\nabla\times\C_{h}\right)+\left(\mathbf g^{n}-\mathbf g(t_{n}),\C_{h}\right),\\
\nonumber\label{error-theta-trucation}&(R_{h,3}^n,\phi_h)=a_{2}(\u^{n},\theta^{n},\phi_{h})-a_{2}(\mathbf{U}^{n-1},\Theta^{n},\phi_{h})+(\mathbf{U}^{n}\cdot e_3,\phi_{h})-(\u^{n}\cdot e_3,\phi_{h})\\
~~~~~~~~&\qquad\qquad\qquad+\sigma(\nabla(\theta^n-\mathcal{R}_h\theta^n),\nabla\phi_h)+\left(\psi^{n}-\psi(t_{n}),\phi_{h}\right),
\end{align}
\noindent
The data differences arise from the time averages \eqref{eq-data-averages}. By
\eqref{eq-data-consistency} and Young's inequality, they contribute
$C\Delta t\bigl\{\|\partial_{t}\mathbf f\|^{2}_{L^{2}(t_{n-1},t_{n};\L^{2})}
+\|\partial_{t}\mathbf g\|^{2}_{L^{2}(t_{n-1},t_{n};\L^{2})}
+\|\partial_{t}\psi\|^{2}_{L^{2}(t_{n-1},t_{n};L^{2})}\bigr\}$
to $\lambda_n$; after multiplication by $\Delta t$ and summation, this is
$\mathcal O((\Delta t)^2)$.
\noindent
  where,
    \begin{align*}
       a_{1}(\u, \u, \v)&=\frac{1}{2}\left\{\int_{\mathbb{D}}(( \u \cdot \nabla) \u) \cdot \v d x-\int_{\mathbb{D}}(( \u \cdot \nabla) \v) \cdot \u \,d x \right \},\\
(a_{0}(\u),\nabla\v)&=\nu\int_{\mathbb{D}} \left|\nabla \u\right|^{p-2}\nabla \u\cdot \nabla \v\,dx,\\
a_{2}(\u, \theta, \phi)&=\frac{1}{2}\left\{\int_{\mathbb{D}}( \u \cdot \nabla \theta) \phi \, d x-\int_{\mathbb{D}}( \u \cdot \nabla \phi) \theta \,d x \right \}.
   \end{align*}
We also use the following uniform projection bounds, which keep all error factors in
$\mathbf L^{2}$.
   \begin{lemma}\label{eq-error-estimate bounds}
       Let $(\u,\B,\theta)$ satisfy Assumption~\ref{Assumption}. Then the following estimates hold uniformly in $h$ and in $t\in[0,T]$:
       \begin{align}
           \|\u-\mathcal{P}_{h}\u\|_{\L^{\infty}}+\|\nabla(\u-\mathcal{P}_{h}\u)\|_{\L^{3}}+\|\nabla\times(\B-\mathcal{F}_{h}\B)\|_{\L^3}+\|\theta-\mathcal{R}_{h}\theta\|_{L^{\infty}}+\|\nabla(\theta-\mathcal{R}_{h}\theta)\|_{L^{3}}\le C_r,
       \end{align}
\begin{align}\label{eq-Linfty-projections}
\|\mathcal{P}_{h}\u\|_{\L^{\infty}}+\|\nabla\mathcal{P}_{h}\u\|_{\L^{\infty}}
+\|\mathcal{F}_{h}\B\|_{\L^{\infty}}+\|\nabla\times\mathcal{F}_{h}\B\|_{\L^{\infty}}
+\|\mathcal{R}_{h}\theta\|_{L^{\infty}}+\|\nabla\mathcal{R}_{h}\theta\|_{L^{\infty}}\le C_r .
\end{align}
   \end{lemma}

\begin{proof}
The first group follows as in \cite[Lemma~5.4]{ding2022convergence}. For the Stokes and Ritz
projections, \eqref{eq-Linfty-projections} follows from Assumption~\ref{Assumption}, the
triangle inequality, and $W^{1,\infty}$-stability on quasi-uniform meshes
\cite{GiraultRaviart1986}. For $\mathcal F_h$, a comparable general
$\mathbf L^{\infty}$-stability result on Lipschitz polyhedra is not available, so we assume
$\|\mathcal F_{h}\B\|_{\L^{\infty}}+
\|\nabla\times\mathcal F_{h}\B\|_{\L^{\infty}}\le C_r$.
This holds, for example, when $\mathcal F_h$ coincides with the canonical N\'ed\'elec
interpolant on the regularity class of Assumption~\ref{Assumption}
\cite[Ch.~5]{Monk2003}. Thus no embedding
$\mathbf W^{1,p}_0\hookrightarrow\mathbf L^{\infty}$ is used.
   \end{proof}

\subsection{Proof of Theorem \textup{\ref{Main theorem of error}}}\label{subsec-errorproof}

   \begin{proof}
Step~1 derives the energy inequality; Steps~2-5 estimate the consistency, convective,
$p$-Laplacian, electromagnetic, and buoyancy terms; Step~6 applies discrete Gr\"onwall.

\medskip
\noindent
\textbf{Step 1: the energy inequality for the projected errors.}\,
       Testing \eqref{error eq-1}-\eqref{error eq-3} with
$(E^{n}_{\u},E^{n}_{\B},E^n_{\theta})$ gives
    \begin{align}
       \label{eq-error-system-1} &(d_{t}E^{n}_{\u},E^{n}_{\u})+\big(a_{0}(\mathbf{U}^{n})-a_{0}(\mathcal{P}_{h}\u^{n}),\nabla E^{n}_{\u}\big)=(M_{h,1}^{n}+R_{h,1}^{n},E^{n}_{\u}),\\
       \label{eq-error-system-2} &(d_{t}E^{n}_{\B},E^{n}_{\B})+\kappa\left(\nabla \times E^{n}_{\B},\nabla \times E^{n}_{\B}\right)=(M_{h,2}^{n}+R_{h,2}^{n},E^{n}_{\B}),\\
      \label{eq-error-system-3} &(d_{t}E^{n}_{\theta}, E^n_{\theta})+\sigma\left(\nabla E^{n}_{\theta},\nabla E^n_{\theta}\right)=(M_{h,3}^{n}+R_{h,3}^{n},E^n_{\theta}),
     \end{align}
      where we have abbreviated the convective parts of $R^{n}_{h,1}$ and $R^{n}_{h,3}$ by
$$(R_{C1h}^n,E_{\u}^n):=a_1(\u^n,\u^n, E_{\u}^n)-a_1(\mathbf{U}^{n-1}, \mathbf{U}^n, E_{\u}^n),\qquad
(R_{C2h}^n,E_{\theta}^n):=a_2(\u^n,\theta^n, E_{\theta}^n)-a_2(\mathbf{U}^{n-1}, \Theta^n, E_{\theta}^n).$$
       Using $2b(b-a)=b^2-a^2+(b-a)^2$ and \eqref{eq-strong-monotonicity},
       \begin{align}
       & (d_{t}E^{n}_{\u},E^{n}_{\u})+\big(a_{0}(\mathbf U^{n})-a_{0}(\mathcal P_{h}\u^{n}),\nabla E^{n}_{\u}\big)\ge
        \frac{1}{2}\left[d_{t}\|E^{n}_{\u}\|_{\L^{2}}^{2}+\Delta t\|d_{t}E^{n}_{\u}\|_{\L^{2}}^{2}\right]+\nu\,2^{2-p}{\|\nabla E^{n}_{\u}\|^{p}_{\L^{p}}},\\
      &  (d_{t}E^{n}_{\B}, E^{n}_{\B})+\kappa\left(\nabla \times E^{n}_{\B},\nabla \times E^{n}_{\B}\right)\ge \frac{1}{2}\left[d_{t}\|E^{n}_{\B}\|_{\L^{2}}^{2}+\Delta t\|d_{t}E^{n}_{\B}\|_{\L^{2}}^{2}\right]+\kappa\|\nabla \times E^{n}_{\B}\|_{\L^{2}}^{2},\\
     & (d_{t}E^{n}_{\theta}, E^n_{\theta})+\sigma\left(\nabla E^{n}_{\theta},\nabla E^n_{\theta}\right)\ge \frac{1}{2}\left[d_{t}\|E^{n}_{\theta}\|_{L^{2}}^{2}+\Delta t\|d_{t}E^{n}_{\theta}\|_{L^{2}}^{2}\right]+\sigma\|\nabla E^{n}_{\theta}\|_{L^{2}}^{2}.
    \end{align}
    \noindent
Since $a_{0}$ is nonlinear,
$a_{0}(\mathbf U^{n})-a_{0}(\mathcal P_h\u^{n})\neq a_{0}(E^n_{\u})$.
The lower bound follows from \eqref{eq-strong-monotonicity} with
$\boldsymbol\xi=\nabla\mathbf U^n$ and
$\boldsymbol\eta=\nabla\mathcal P_h\u^n$. The term
$(a_{0}(\u^{n})-a_{0}(\mathcal P_h\u^{n}),\nabla E^n_{\u})$ in $R^n_{h,1}$ is instead a
consistency error.

\medskip
\noindent
\textbf{Step 2: the temporal consistency terms.}\,
For $M^n_{h,1}$,
    \begin{align}
        \nonumber(M_{h,1}^{n},E^{n}_{\u})&=(\partial_{t}\u^{n}-d_{t}\mathcal{P}_{h}\u^{n},E^{n}_{\u}) =(\partial_{t}\u^{n}-d_{t}\u^{n},E^{n}_{\u})+(d_{t}\u^{n}-d_{t}\mathcal{P}_{h}\u^{n},E^{n}_{\u}).
    \end{align}
\noindent
The first summand is the standard backward-Euler truncation error,
\[
\|\partial_{t}\u^{n}-d_{t}\u^{n}\|_{\L^{2}}
=\Big\|\frac{1}{\Delta t}\int_{t_{n-1}}^{t_{n}}(t-t_{n-1})\,\partial_{tt}\u\,dt\Big\|_{\L^{2}}
\le(\Delta t)^{\frac12}\|\partial_{tt}\u\|_{L^{2}(t_{n-1},t_{n};\mathbf L^{2})} .
\]
Since $\mathcal P_h$ is linear and time independent, it commutes with $d_t$ and the time
integral:
\begin{equation}\label{eq-Mh1-integral}
\|d_{t}(I-\mathcal P_{h})\u^{n}\|_{\L^{2}}
=\Big\|\frac{1}{\Delta t}\int_{t_{n-1}}^{t_{n}}(I-\mathcal P_{h})\partial_{t}\u\,dt\Big\|_{\L^{2}}
\le \frac{C_{e}h^{l+1}}{\Delta t}\int_{t_{n-1}}^{t_{n}}\|\partial_{t}\u\|_{1+l,2}\,dt
\le \frac{C_{e}h^{l+1}}{(\Delta t)^{\frac12}}\|\partial_{t}\u\|_{L^{2}(t_{n-1},t_{n};\mathbf W^{1+l,2})},
\end{equation}
By Cauchy-Schwarz in time,
    \begin{align}
       \nonumber (M_{h,1}^{n},E^{n}_{\u})&=(\partial_{t}\u^{n}-d_{t}\mathcal{P}_{h}\u^{n},E^{n}_{\u})\\
      \nonumber  &=(\partial_{t}\u^{n}-d_{t}\u^{n},E^{n}_{\u})+(d_{t}\u^{n}-d_{t}\mathcal{P}_{h}\u^{n},E^{n}_{\u})\\
        &\le C\left\{(\Delta t)^{\frac{1}{2}}\|\partial_{tt}\u \|_{L^{2}(t_{n-1},t_{n};\mathbf{L}^{2})}+\frac{C_{e}h^{l+1}}{(\Delta t)^{\frac12}}\|\partial_{t}\u\|_{L^{2}(t_{n-1},t_{n};\mathbf W^{1+l,2})}\right\}\|E^{n}_{\u}\|_{\L^{2}}.
    \end{align}
    Similarly, 
    \begin{align}
       &(M_{h,2}^{n},E^{n}_{\B})\le C\left\{(\Delta t)^{\frac{1}{2}}\|\partial_{tt}\B \|_{L^{2}(t_{n-1},t_{n};\mathbf{L}^{2})}+\frac{C_{e}h^{l}}{(\Delta t)^{\frac12}}\|\partial_{t}\B\|_{L^{2}(t_{n-1},t_{n};\mathbf W^{l,2})}\right\}\|E^{n}_{\B}\|_{\L^{2}},\\
       & (M_{h,3}^{n},E^{n}_{\theta})\le C\left\{(\Delta t)^{\frac{1}{2}}\|\partial_{tt}\theta \|_{L^{2}(t_{n-1},t_{n};{L}^{2})}+\frac{C_{e}h^{l+1}}{(\Delta t)^{\frac12}}\|\partial_{t}\theta\|_{L^{2}(t_{n-1},t_{n};{W}^{1+l,2})}\right\}\|E^{n}_{\theta}\|_{L^{2}}.
    \end{align}
\noindent
Thus \eqref{eq-Mh1-integral} produces $(\Delta t)^{-1/2}$ rather than
$(\Delta t)^{-1}$; after Young's inequality, its weighted sum is
$h^{2(l+1)}\|\partial_t\u\|^2_{L^2(0,T;\cdot)}$, uniformly in $\Delta t$.

\medskip
\noindent
\textbf{Step 3: the convective terms.}\,
{By \eqref{eq-exactly-divfree}, the relevant discrete velocities are pointwise
divergence-free. Hence, for solenoidal $\mathbf a\in\mathcal X_0$,
$\mathbf b\in\mathbf W^{1,p}$, and $\mathbf c\in\mathcal X_0$,
\begin{equation}\label{eq-a1-ibp}
\int_{\mathbb D}(\mathbf a\cdot\nabla)\mathbf c\cdot\mathbf b\,dx
=-\int_{\mathbb D}(\mathbf a\cdot\nabla)\mathbf b\cdot\mathbf c\,dx,
\qquad\text{so}\qquad
a_{1}(\mathbf a,\mathbf b,\mathbf c)=\int_{\mathbb D}(\mathbf a\cdot\nabla)\mathbf b\cdot\mathbf c\,dx .
\end{equation}
Identity \eqref{eq-a1-ibp} moves derivatives onto the smooth factor, so
$E^n_{\u}$ remains in $\mathbf L^2$. Without exact divergence-freeness a divergence remainder
would produce $\|\nabla E^n_{\u}\|_{\mathbf L^2}$ and, after $(p,p')$-Young,
the undesirable term $\|E^{n-1}_{\u}\|^{p/(p-1)}_{\mathbf L^2}$.}

    For $R_{C1h}^{n}$, adding and subtracting the intermediate quantities gives
    \begin{align}
       \nonumber (R_{C1h}^{n},E^{n}_{\u})&=a_{1}(\u^{n}, \u^{n}, E^{n}_{\u})-a_{1}(\mathbf{U}^{n-1}, \mathbf{U}^{n}, E^{n}_{\u})\\\nonumber
        &=\underbrace{a_{1}(\u^{n}-\mathcal{P}_{h}\u^{n-1},\u^{n}, E^{n}_{\u})}_{\Lambda_{1,n}}+\underbrace{a_{1}(\mathcal{P}_{h}\u^{n-1},\u^{n}-\mathcal{P}_{h}\u^{n-1},E^{n}_{\u})}_{\Lambda_{3,n}}\\ \nonumber&\qquad+\underbrace{a_{1}(\mathcal{P}_{h}\u^{n-1},\mathcal{P}_{h}\u^{n-1}-\mathcal{P}_{h}\u^{n},E^{n}_{\u})}_{\Lambda_{4,n}}+\underbrace{a_{1}(\mathcal{P}_{h}\u^{n-1},\mathcal{P}_{h}\u^{n}-\mathbf{U}^{n},E^{n}_{\u})}_{\Lambda_{5,n}}\\
        &\qquad +\underbrace{a_{1}(\mathcal{P}_{h}\u^{n-1}-\mathbf{U}^{n-1},\mathcal{P}_{h}\u^{n},E^{n}_{\u})}_{\Lambda_{6,n}}-\underbrace{a_{1}(\mathcal{P}_{h}\u^{n-1}-\mathbf{U}^{n-1},\mathcal{P}_{h}\u^{n}-\mathbf{U}^{n},E^{n}_{\u})}_{\Lambda_{7,n}}. \end{align}
{\noindent
Since $\mathcal P_{h}\u^{n}-\mathbf U^{n}=-E^{n}_{\u}$ and $a_{1}(\mathbf a,\mathbf c,\mathbf c)=0$,}
    $$\Lambda_{5,n}=\Lambda_{7,n}=0.$$
    For the remaining terms, \eqref{eq-a1-ibp},
Lemma~\ref{eq-error-estimate bounds}, Assumption~\ref{Assumption}, and
\eqref{eq-error approximation} give
\begin{align}
\nonumber\Lambda_{1,n}
&=\int_{\mathbb D}\big((\u^{n}-\mathcal P_{h}\u^{n-1})\cdot\nabla\big)\u^{n}\cdot E^{n}_{\u}\,dx
\ \le\ \|\u^{n}-\mathcal P_{h}\u^{n-1}\|_{\L^{2}}\,\|\nabla\u^{n}\|_{\L^{\infty}}\,\|E^{n}_{\u}\|_{\L^{2}}\\
\nonumber&\le C_{r}\Big(\|\u^{n}-\u^{n-1}\|_{\L^{2}}+\|\u^{n-1}-\mathcal P_{h}\u^{n-1}\|_{\L^{2}}\Big)\|E^{n}_{\u}\|_{\L^{2}}\\
&\le C_{r}\Big\{(\Delta t)^{\frac12}\|\partial_{t}\u\|_{L^{2}(t_{n-1},t_{n};\mathbf L^{2})}+C_{e}h^{l+1}\|\u^{n-1}\|_{l+1,2}\Big\}\|E^{n}_{\u}\|_{\L^{2}} .
\end{align}
Similarly, using also the $H^1$-stability of $\mathcal P_h$,
\begin{align}
\Lambda_{3,n}&\le\|\mathcal P_{h}\u^{n-1}\|_{\L^{\infty}}\|\nabla(\u^{n}-\mathcal P_{h}\u^{n-1})\|_{\L^{2}}\|E^{n}_{\u}\|_{\L^{2}}
\le C_{r}\Big\{(\Delta t)^{\frac12}\|\partial_{t}\nabla\u\|_{L^{2}(t_{n-1},t_{n};\mathbf L^{2})}+C_{e}h^{l}\|\u^{n-1}\|_{l+1,2}\Big\}\|E^{n}_{\u}\|_{\L^{2}},\\
\Lambda_{4,n}&\le\|\mathcal P_{h}\u^{n-1}\|_{\L^{\infty}}\|\nabla\mathcal P_{h}(\u^{n-1}-\u^{n})\|_{\L^{2}}\|E^{n}_{\u}\|_{\L^{2}}
\le C_{r}(\Delta t)^{\frac12}\|\partial_{t}\nabla\u\|_{L^{2}(t_{n-1},t_{n};\mathbf L^{2})}\|E^{n}_{\u}\|_{\L^{2}} .
\end{align}
Finally,
\begin{equation}\label{eq-Lambda6}
\Lambda_{6,n}=-\int_{\mathbb D}(E^{n-1}_{\u}\cdot\nabla)\mathcal P_{h}\u^{n}\cdot E^{n}_{\u}\,dx
\ \le\ \|E^{n-1}_{\u}\|_{\L^{2}}\,\|\nabla\mathcal P_{h}\u^{n}\|_{\L^{\infty}}\,\|E^{n}_{\u}\|_{\L^{2}}
\ \le\ C_{r}\|E^{n-1}_{\u}\|_{\L^{2}}\|E^{n}_{\u}\|_{\L^{2}},
\end{equation}
Using \eqref{eq-Linfty-projections}, all factors remain quadratic in the
$\mathbf L^2$-errors. Hence Young's inequality gives
\begin{align}\label{eq-RC1h-bound}
\nonumber (R_{C1h}^{n},E^{n}_{\u})&\le \varepsilon\|E^{n}_{\u}\|^{2}_{\L^{2}}+C_{\varepsilon,r}\|E^{n-1}_{\u}\|^{2}_{\L^{2}}\\
&\quad+C_{\varepsilon,r}\Big\{\Delta t\|\partial_{t}\u\|^{2}_{L^{2}(t_{n-1},t_{n};\mathbf L^{2})}
+\Delta t\|\partial_{t}\nabla\u\|^{2}_{L^{2}(t_{n-1},t_{n};\mathbf L^{2})}
+h^{2l}\big(\|\u^{n}\|^{2}_{l+1,2}+\|\u^{n-1}\|^{2}_{l+1,2}\big)\Big\}.
\end{align}

\medskip
\noindent
\textbf{Step 4: 
the Coriolis force term.} We will deal the term as follows

\begin{align}
2\left|
\big(
\boldsymbol{\omega}_{r}\times
(\u^n-\mathcal P_h\u^n),
E_\u^n
\big)
\right|
&\le
2|\boldsymbol{\omega}_{r}|
\|\u^n-\mathcal P_h\u^n\|_{\mathbf L^2}
\|E_\u^n\|_{\mathbf L^2} \le
\varepsilon\|E_\u^n\|_{\mathbf L^2}^2
+
C_\varepsilon|\boldsymbol{\omega}_{r}|^2
h^{2(l+1)}
\|\u^n\|_{1+l,2}^2.
\end{align}

\medskip
\noindent
\textbf{Step 5: the $p$-Laplacian consistency term.}\,
   We now estimate the difference between the flux evaluated at the exact solution and at
   its Stokes projection, which is the term of $R^{n}_{h,1}$ carrying the nonlinearity:
    
    \begin{align}
      \nonumber  \left|(a_{0}(\u^{n}),\nabla E^{n}_{\u})-(a_{0}(\mathcal{P}_{h}\u^{n}),\nabla E^{n}_{\u})\right| & =\nu\left|\int_{\mathbb{D}}|\nabla\u^{n}|^{p-2}\nabla\u^{n}\cdot\nabla E^{n}_{\u} dx-\int_{\mathbb{D}}|\nabla\mathcal{P}_{h}\u^{n}|^{p-2}\nabla\mathcal{P}_{h}\u^{n}\cdot\nabla E^{n}_{\u} dx\right|\\\nonumber
     &=\nu\int_{\mathbb{D}}\left|\left(|\nabla\u^{n}|^{p-2}\nabla \u^{n}-|\nabla\mathcal{P}_{h}\u^{n}|^{p-2}\nabla\mathcal{P}_{h}\u^{n}\right)\right|\cdot|\nabla E^{n}_{\u}| dx.
     \end{align}
    Using
    $$\left||\xi|^{p-2}\xi-|\eta|^{p-2}\eta\right|\le C_1|\xi-\eta|(|\xi|+|\eta|)^{p-2},$$
together with H\"older's and Young's inequalities,
     \begin{align}
   \nonumber  &\le \nu\int_{\mathbb{D}}(|\nabla\u^{n}-\nabla\mathcal{P}_{h}\u^{n}|)(|\nabla\u^{n}|+|\nabla\mathcal{P}_{h}\u^{n}|)^{p-2}\cdot|\nabla E^{n}_{\u}|dx\\\nonumber
     &\le\nu \left(\int_{\mathbb{D}}|\nabla E^{n}_{\u}|^p \right)^{\frac{1}{p}}\left(\int_{\mathbb{D}}(|\nabla(\u^{n}-\mathcal{P}_{h}\u^{n})|)^{\frac{p}{p-1}}(|\nabla\u^{n}|+|\nabla\mathcal{P}_{h}\u^{n}|)^{\frac{p(p-2)}{p-1}}dx\right)^{\frac{p-1}{p}}\\\nonumber
    &\le\nu\|\nabla E^{n}_{\u}\|_{\L^p}\left(\int_{\mathbb{D}}|\nabla(\u^{n}-\mathcal{P}_{h}\u^{n})|^{p}dx\right)^{\frac{p-1}{p(p-1)}}\left(\int_{\mathbb{D}}(|\nabla\u^{n}|+|\nabla\mathcal{P}_{h}\u^{n}|)^{\frac{p(p-2)}{p-1}\frac{p-1}{p-2}} dx\right)^{\frac{p-1}{p}\cdot\frac{p-2}{p-1}}\\ \nonumber
   &\le\nu\|\nabla E^{n}_{\u}\|_{\L^p}\|\nabla(\u^{n}-\mathcal{P}_{h}\u^n)\|_{\L^p}\||\nabla \u^n|+|\nabla\mathcal{P}_{h}\u^n|\|^{p-2}_{\L^p}
   \\ \nonumber
   &\le \nu\left( \varepsilon_{0}^{p}\frac{\|\nabla E^{n}_{\u}\|^{p}_{\L^p}}{p}+\frac{p-1}{p} \left\{\|\nabla(\u^{n}-\mathcal{P}_{h}\u^{n} ) \|^{\frac{p}{p-1}}_{\L^p} \right\} \left\{ \||\nabla \u^n|+|\nabla\mathcal{P}_{h}\u^n|\|^{\frac{(p-2)p}{p-1}}_{\L^p} \right\} \right)\\
   & \le \nu\left( \varepsilon_{0}^{p}\frac{\|\nabla E^{n}_{\u}\|^{p}_{\L^p}}{p}+\frac{p-1}{p}\left\{ h^{\frac{lp}{p-1}}\|\u^{n}\|^{\frac{p}{p-1}}_{l+1,p}\||\nabla \u^n|+|\nabla\mathcal{P}_{h}\u^n|\|^{\frac{(p-2)p}{p-1}}_{\L^p} \right\}\right).
      \end{align}
\noindent
Using \eqref{eq-error approximation Lp}, choose $\varepsilon_0>0$ so that
$\varepsilon_0^p\le\tfrac{p}{4}2^{2-p}$; the first term is then absorbed by
$\nu2^{2-p}\|\nabla E^n_{\u}\|^p_{\L^p}$. This is the only use of $(p,p')$-Young on the
consistency error, producing the term $h^{lp/(p-1)}$ and hence the spatial rate in
Theorem~\ref{Main theorem of error}.

\medskip
\noindent
\textbf{Step 6: the electromagnetic and buoyancy couplings.}\,
Collecting the electromagnetic terms from \eqref{eq-error-system-1} and
\eqref{eq-error-system-2},
     \begin{align}
       \nonumber&\left|\left(\mathbf{b}^{n} \times (\nabla \times \mathbf{b}^{n}),E^{n}_{\u}\right )-\left(\mathfrak{B}^{n-1} \times (\nabla \times \mathfrak{B}^n),E^{n}_{\u}\right)+\left(\mathbf{U}^{n}\times\mathfrak{B}^{n-1},\nabla \times E^{n}_{\B}\right)-\left(\u^{n}\times\B^{n},\nabla \times E^{n}_{\B}\right)\right|\\ \nonumber
        &=\bigg|\left([\mathbf{b}^{n}-\mathfrak{B}^{n-1}]\times \nabla \times \mathbf{b}^{n},E^{n}_{\u}\right )+\left(\u^{n}\times[\mathfrak{B}^{n-1}-\B^{n}],\nabla \times E^{n}_{\B}\right)\\\nonumber&\qquad+\left([\mathcal{F}_{h}\B^{n-1}+E^{n-1}_{\B}]\times\nabla \times[\B^{n}-\mathcal{F}_{h}\B^{n}],E^{n}_{\u}\right)-\left([\u^{n}-\mathcal{P}_{h}\u^{n}]\times[\mathcal{F}_{h}\B^{n-1}+E^{n-1}_{\B}],\nabla \times E^{n}_{\B}\right)\bigg|\\\nonumber
        &\le C\bigg[\|\nabla\times\B^{n}\|_{\L^{3}}\|\mathbf{b}^{n}-\mathfrak{B}^{n-1}\|_{\L^{2}}\|E^{n}_{\u}\|_{\L^{6}}+\|\u^{n}\|_{\L^{\infty}}\|\mathbf{b}^{n}-\mathfrak{B}^{n-1}\|_{\L^{2}}\|\nabla \times E^{n}_{\B}\|_{\L^{2}}\\ \nonumber&\qquad+\|\mathcal{F}_{h}\B^{n-1}\|_{\L^{3}}\|\nabla \times(\B^{n}-\mathcal{F}_{h}\B^{n})\|_{\L^{2}}\|E^{n}_{\u}\|_{\L^{6}}+\|\nabla \times(\B^{n}-\mathcal{F}_{h}\B^{n})\|_{\L^{3}}\|E^{n-1}_{\B}\|_{\L^{6}}\|E^{n}_{\u}\|_{\L^{2}}\\&\qquad+\|\u^{n}-\mathcal{P}_{h}\u^{n}\|_{\L^{6}}\|\mathcal{F}_{h}\B^{n-1}\|_{\L^{3}}\|\nabla \times E^{n}_{\B}\|_{\L^{2}}+\|\u^{n}-\mathcal{P}_{h}\u^{n}\|_{\L^{\infty}}\|E^{n-1}_{\B}\|_{\L^{2}}\|\nabla\times E^{n}_{\B}\|_{\L^{2}}\bigg].
        \end{align}
\noindent
Using \eqref{eq-Linfty-projections} and Assumption~\ref{Assumption}, the error factors are
kept in $\mathbf L^2$. In particular, the first and third products are bounded by
$\|\nabla\times\B^{n}\|_{\L^{\infty}}\|\B^{n}-\mathfrak B^{n-1}\|_{\L^{2}}\|E^{n}_{\u}\|_{\L^{2}}$
and
$\|\mathcal F_{h}\B^{n-1}\|_{\L^{\infty}}\|\nabla\times(\B^{n}-\mathcal F_{h}\B^{n})\|_{\L^{2}}\|E^{n}_{\u}\|_{\L^{2}}$,
while
$\|E^{n-1}_{\B}\|_{\L^{6}}\le C\|\nabla\times E^{n-1}_{\B}\|_{\L^{2}}$.\
\noindent
Inserting the triangle inequality
\begin{align*}
\|\B^{n}-\mathfrak{B}^{n-1}\|_{\L^{2}} &\le\|\B^{n}-\B^{n-1}\|_{\L^{2}}+\|\B^{n-1}-\mathcal{F}_{h}\B^{n-1}\|_{\L^{2}}+\|E^{n-1}_{\B}\|_{\L^{2}}
\\ &\le(\Delta t)^{\frac12}\|\partial_{t}\B\|_{L^{2}(t_{n-1},t_{n};\mathbf L^{2})}+C_{e}h^{l}\|\B^{n-1}\|_{l,2}+\|E^{n-1}_{\B}\|_{\L^{2}},
\end{align*}
together with \eqref{eq-error approximation} and \eqref{eq-Linfty-projections}, and collecting
according to whether a term multiplies $\|E^{n}_{\u}\|_{\L^{2}}$ or
$\|\nabla\times E^{n}_{\B}\|_{\L^{2}}$,
  \begin{align}\label{eq-EM-collected}
     \nonumber   &\le C_{r}\Big\{(\Delta t)^{\frac12}\|\partial_{t}\B\|_{L^{2}(t_{n-1},t_{n};\mathbf L^{2})}
     + C_{e}h^{l}\big(\|\B^{n-1}\|_{l,2}+\|\nabla\times\B^{n}\|_{l,2}\big)
     +\|E^{n-1}_{\B}\|_{\L^{2}}+\|\nabla\times E^{n-1}_{\B}\|_{\L^{2}}\Big\}\|E^{n}_{\u}\|_{\L^{2}}\\
     &\qquad+C_{r}\Big\{(\Delta t)^{\frac12}\|\partial_{t}\B\|_{L^{2}(t_{n-1},t_{n};\mathbf L^{2})}
     + C_{e}h^{l}\big(\|\B^{n-1}\|_{l,2}+\|\u^{n}\|_{l+1,2}\big)
     +\|E^{n-1}_{\B}\|_{\L^{2}}\Big\}\|\nabla\times E^{n}_{\B}\|_{\L^{2}} .
     \end{align}
\noindent
Hence, for any $\varepsilon>0$, the quadratic Young inequality gives
     \begin{align}\label{eq-EM-young}
    \nonumber & \le \varepsilon\|E^{n}_{\u}\|^{2}_{\L^{2}}+\frac{\kappa}{8}\|\nabla\times E^{n}_{\B}\|^{2}_{\L^{2}}
    +C_{\varepsilon,r}\Big\{\|E^{n-1}_{\B}\|^{2}_{\L^{2}}+\|\nabla\times E^{n-1}_{\B}\|^{2}_{\L^{2}}\Big\}\\
    &\qquad+C_{\varepsilon,r}\Big\{\Delta t\|\partial_{t}\B\|^{2}_{L^{2}(t_{n-1},t_{n};\mathbf L^{2})}
    +h^{2l}\big(\|\B^{n-1}\|^{2}_{l,2}+\|\nabla\times\B^{n}\|^{2}_{l,2}+\|\u^{n}\|^{2}_{l+1,2}\big)\Big\} .
    \end{align}
{\noindent
Because \eqref{eq-Linfty-projections} keeps the velocity error in $\mathbf L^2$, no
$(p,p')$-Young inequality is needed here; otherwise the temporal rate would deteriorate to
$\mathcal O((\Delta t)^{p'})$. The previous-level term is absorbed after summation through
$\Delta t\sum_{n=1}^{m}\|\nabla\times E^{n-1}_{\B}\|^{2}_{\L^{2}}\le
\Delta t\|\nabla\times E^{0}_{\B}\|^{2}_{\L^{2}}+
\Delta t\sum_{n=1}^{m}\|\nabla\times E^{n}_{\B}\|^{2}_{\L^{2}}$.}

For the buoyancy term in \eqref{error-u-truncation},
\begin{align}
    \nonumber(\Theta^{n-1}e_{3},E^n_{\u})-\left(\theta^{n-1}e_{3},E^n_{\u}\right)&=(\Theta^{n-1}e_{3}-\theta^{n-1}e_{3}, E^n_{\u})\\
   \nonumber &\le C\|\Theta^{n-1}-\theta^{n-1}\|_{L^2}\|E^n_{\u}\|_{\L^2}\\
    \nonumber&\le C\|\Theta^{n-1}-\mathcal{R}_h\theta^{n-1}\|_{L^2}\|E^n_{\u}\|_{\L^2}+C\|\mathcal{R}_h\theta^{n-1}-\theta^{n-1}\|_{{L^2}}\|E^n_{\u}\|_{{\L^2}}\\
    &\le C \|{E^{n-1}_{\theta}}\|_{L^2}\|E^n_{\u}\|_{\L^2}+{C_{e}h^{l}\|\theta^{n-1}\|_{1+l,2}}\|E^n_{\u}\|_{\L^2}.
\end{align}
\noindent
Here $E^{n-1}_{\theta}$ appears because the buoyancy force is evaluated at the previous time
level. The $L^2$-$\mathbf L^2$ pairing preserves the factor $h^l$. For the reciprocal
coupling in \eqref{error-theta-trucation},
\begin{align}
   (\mathbf{U}^{n}\cdot e_3,E^n_{\theta})-(\u^{n}\cdot e_3,E^n_{\theta})&\le C \|E^n_{\theta}\|_{L^2}\|E^n_{\u}\|_{\L^2}+C_{e}h^{l+1}\|\u^{n}\|_{1+l,2}\|E^n_{\theta}\|_{L^2}.
\end{align}
 \noindent
For the temperature convection term $R_{C2h}^{n}$,
\begin{align}
    \nonumber(R_{C2h}^n,E^n_{\theta})&=a_2(\u^n,\theta^n,E^n_{\theta})-a_2(\mathbf{U}^{n-1},\Theta^n,E^n_{\theta})\\
   \nonumber&=\underbrace{ a_2(\u^n-\mathcal{P}_h\u^{n-1},\theta^n, E^n_{\theta})}_{\eta_{1,n}}+\underbrace{a_2(\mathcal{P}_h\u^{n-1}-\mathcal{P}_h\u^{n-1},\theta^n, E^{n}_{\theta})}_{\eta_{2,n}}+\underbrace{a_2(\mathcal{P}_{h}\u^{n-1},\theta^n-\mathcal{R}_h\theta^n,E^n_\theta)}_{\eta_{3,n}}\\
  \nonumber &\qquad+\underbrace{a_2(\mathcal{P}_{h}\u^{n-1},\mathcal{R}_h\theta^n-\mathcal{R}_h\theta^n,E^n_\theta)}_{\eta_{4,n}}+\underbrace{a_2(\mathcal{P}_{h}\u^{n-1},\mathcal{R}_h\theta^n-\Theta^n,E^n_\theta)}_{\eta_{5,n}}+\underbrace{a_2(\mathcal{P}_{h}\u^{n-1}-\mathbf{U}^{n-1},\mathcal{R}_h\theta^n,E^n_\theta)}_{\eta_{6,n}}\\
  &\qquad+\underbrace{a_2(\mathcal{P}_{h}\u^{n-1}-\mathbf{U}^{n-1},\mathcal{R}_h\theta^n-\Theta^n,E^n_\theta)}_{\eta_{7,n}}.
\end{align}
The terms $\eta_{2,n}$ and $\eta_{4,n}$ vanish identically, while
$\eta_{5,n}$ and $\eta_{7,n}$ vanish by
$a_{2}(\u,\phi,\phi)=0$ and
$\mathcal R_h\theta^n-\Theta^n=-E^n_{\theta}$:
$$\eta_{2,n}=\eta_{4,n}=\eta_{5,n}=\eta_{7,n}=0.$$

Using \eqref{eq-error approximation} and Lemma~\ref{eq-error-estimate bounds},
\begin{align}
   \nonumber \eta_{1,n} =  a_2(\u^n-\mathcal{P}_h\u^{n-1},\theta^n, E^n_{\theta}) &\le \|\u^n-\mathcal{P}_h\u^{n-1}\|_{\L^2}\|{\nabla\theta^{n}}\|_{{\mathbf L}^\infty}{\|E^n_{\theta}\|_{L^2}}\\
      \nonumber&\le C_r\left(\|\u^{n}-\u^{n-1}\|_{\mathbf{L}^{2}}+\|\u^{n-1}-\mathcal{P}_{h}\u^{n-1}\|_{\mathbf{L}^{2}}\right)\|E^{n}_{\theta}\|_{L^{2}}\\
        &\le C_r\left\{(\Delta t)^{1/2}\|\partial_{t}\u \|_{L^{2}(t_{n-1},t_{n};\mathbf{L}^{2})}+C_eh^{l+1}\|\u^{n-1}\|_{l+1,2}\right\}\|E^{n}_{\theta}\|_{{L}^{2}}.
   \end{align}
\noindent
As in the velocity case, the derivative is moved onto $\theta^n$, and
Assumption~\ref{Assumption} keeps the error factor in $L^2$.
   Similarly,
   \begin{align}
       \nonumber\eta_{3,n} =a_2(\mathcal{P}_{h}\u^{n-1},\theta^n-\mathcal{R}_h\theta^n,E^n_\theta) &\le\|\mathcal{P}_{h}\u^{n-1}\|_{\L^\infty}\|\nabla(\theta^n-\mathcal{R}_h\theta^n)\|_{L^2}\|E^n_\theta\|_{L^2} \le C_r h^{l}\|\theta^n\|_{1+l,2}\| E^n_{\theta}\|_{L^2}.
   \end{align}
   \begin{align}
      \eta_{6,n}&=a_2(\mathcal{P}_{h}\u^{n-1}-\mathbf{U}^{n-1},\mathcal{R}_h\theta^n,E^n_\theta)
      =-\int_{\mathbb D}\big(E^{n-1}_{\u}\cdot\nabla\mathcal R_{h}\theta^{n}\big)E^{n}_{\theta}\,dx \le C_r \|E^{n-1}_{\u}\|_{\L^2}\|E^n_{\theta}\|_{L^2},
   \end{align}
\noindent
using \eqref{eq-Linfty-projections}.
\noindent
   Combining these estimates gives
\begin{align}\label{eq-RC2h-bound}
(R_{C2h}^n,E^n_{\theta})&\le \varepsilon\|E^{n}_{\theta}\|^{2}_{L^{2}}+C_{\varepsilon,r}\|E^{n-1}_{\u}\|^{2}_{\L^{2}} +C_{\varepsilon,r}\Big\{\Delta t\|\partial_{t}\u\|^{2}_{L^{2}(t_{n-1},t_{n};\mathbf L^{2})}+h^{2l}\big(\|\u^{n-1}\|^{2}_{l+1,2}+\|\theta^{n}\|^{2}_{1+l,2}\big)\Big\}.
\end{align}
The remaining projection terms are estimated by \eqref{eq-error approximation} and
Young's inequality:
\begin{align}
    \kappa\left(\nabla\times(\B^{n}-\mathcal{F}_{h}\B^{n}),\nabla\times E^n_{\B}\right)&\le C\|\nabla\times(\B^{n}-\mathcal{F}_{h}\B^{n})\|_{\L^2}\|\nabla\times E^n_{\B}\|_{\L^2} \le \frac{\kappa}{8}\|\nabla\times E^n_{\B}\|^{2}_{\L^2}+C h^{2l}\|\nabla\times \B^n\|^{2}_{l,2},
\end{align}
\begin{align}
    \sigma\left(\nabla(\theta^n-\mathcal{R}_h\theta^n),\nabla E^n_{\theta}\right)\le \frac{\sigma}{8}\|\nabla E^n_{\theta}\|^{2}_{L^2}+C h^{2l}\|\theta^n\|^{2}_{1+l,2}.
\end{align}

\medskip
\noindent
\textbf{Step 7: collection and discrete Gr\"onwall.}\,
Combining the preceding estimates in
\eqref{eq-error-system-1}-\eqref{eq-error-system-3} and choosing
$\varepsilon,\varepsilon_0$ sufficiently small to absorb the dissipative terms gives
    \begin{align}\label{eq-error-gronwall}
       \nonumber &d_{t}\|E^{n}_{\u}\|^{2}_{\L^2}+d_t\|E^{n}_{\B}\|^{2}_{\L^2}+d_t\|E^{n}_{\theta}\|^{2}_{L^2}
       +(\Delta t)\left[\|d_{t}E^{n}_{\u}\|^{2}_{\L^2}+\|d_{t}E^{n}_{\B}\|^{2}_{\L^2}+\|d_{t}E^{n}_{\theta}\|^{2}_{L^2}\right]\\
       \nonumber&\qquad+\nu\,2^{2-p}\|\nabla E^{n}_{\u}\|^{p}_{\L^p}+\frac{\kappa}{2}\|\nabla\times E^{n}_{\B}\|^{2}_{\L^2}+\frac{\sigma}{2}\|\nabla E^{n}_{\theta}\|^{2}_{L^2}\\
       &\le c\left(\|E^{n}_{\u}\|^{2}_{\L^2}+\|E^{n-1}_{\u}\|^{2}_{\L^2}+\|E^{n}_{\B}\|^2_{\L^2}+\|E^{n-1}_{\B}\|^2_{\L^2}+\|E^{n}_{\theta}\|^{2}_{L^2}+{\|E^{n-1}_{\theta}\|^{2}_{L^2}}\right)
       +c\,\|\nabla\times E^{n-1}_{\B}\|^{2}_{\L^2}+\lambda_{n},
    \end{align}
\noindent
where all error terms on the right are quadratic and
    \begin{align}\label{eq-zeta_n}
        \nonumber\lambda_{n}&=C(\Delta t)\left\{\|\partial_{tt}\u\|^{2}_{L^2(t_{n-1},t_n;\L^2)}+\|\partial_{tt}\B\|^{2}_{L^2(t_{n-1},t_n;\L^2)}+\|\partial_{tt}\theta\|^{2}_{L^2(t_{n-1},t_n;L^2)}\right\}\\
        \nonumber&\qquad+C(\Delta t)\left\{\|\partial_{t}\u\|^{2}_{L^2(t_{n-1},t_n;\mathbf H^{1})}+\|\partial_{t}\B\|^{2}_{L^2(t_{n-1},t_n;\L^{2})}\right\}\\
        \nonumber &\qquad+\frac{C}{\Delta t}\left\{h^{2(l+1)}\|\partial_{t}\u\|^{2}_{L^{2}(t_{n-1},t_{n};\mathbf W^{1+l,2})}+h^{2l}\|\partial_{t}\B\|^{2}_{L^{2}(t_{n-1},t_{n};\mathbf W^{l,2})}+h^{2(l+1)}\|\partial_{t}\theta\|^{2}_{L^{2}(t_{n-1},t_{n};W^{1+l,2})}\right\}\\
        \nonumber&\qquad+ Ch^{2l}\left\{\|\u^{n}\|^{2}_{l+1,2}+\|\u^{n-1}\|^{2}_{l+1,2}+\|\B^{n}\|^2_{l,2}+\|\B^{n-1}\|^2_{l,2}+\|\nabla\times\B^{n}\|^{2}_{l,2}+\|\theta^{n}\|^{2}_{1+l,2}+{\|\theta^{n-1}\|^{2}_{1+l,2}}\right\} \\ &\qquad+Ch^{\frac{lp}{p-1}}\,\|\u^{n}\|^{\frac{p}{p-1}}_{l+1,p}\,\big\||\nabla \u^n|+|\nabla\mathcal{P}_{h}\u^n|\big\|^{\frac{(p-2)p}{p-1}}_{\L^p}.
    \end{align}
\noindent
Multiplying by $\Delta t$ and summing over $n=1,\dots,m$, the first two groups are
$O((\Delta t)^2)$, the projection-time terms are $O(h^{2(l+1)}+h^{2l})$, and the last two
groups are $O(h^{2l}+h^{lp/(p-1)})$. Since
$p'=\frac{p}{p-1}\le2$ and $h\le1$, Assumption~\ref{Assumption} yields
\begin{equation}\label{eq-lambda-sum}
\Delta t\sum_{n=1}^{m}\lambda_n\ \le\ C^{*}\left\{(\Delta t)^{2}+h^{\frac{lp}{p-1}}\right\}.
\end{equation}
The integral representation \eqref{eq-Mh1-integral} is precisely what removes the
undesired term $h^{2l}/\Delta t$.

\noindent
Finally, multiplying \eqref{eq-error-gronwall} by $\Delta t$, summing over $n=1,\dots,m$ and
using the index shift
$\Delta t\sum_{n=1}^{m}\|\nabla\times E^{n-1}_{\B}\|^{2}_{\L^{2}}
\le\Delta t\|\nabla\times E^{0}_{\B}\|^{2}_{\L^{2}}+\Delta t\sum_{n=1}^{m}\|\nabla\times E^{n}_{\B}\|^{2}_{\L^{2}}$
to absorb the term $c\,\|\nabla\times E^{n-1}_{\B}\|^{2}_{\L^{2}}$ into
$\frac{\kappa}{2}\|\nabla\times E^{n}_{\B}\|^{2}_{\L^{2}}$ (which requires $c\,\Delta t\le\kappa/4$,
i.e.\ $\Delta t\le\Delta t_{0}$ with $\Delta t_{0}$ independent of $h$), we obtain
\begin{align*}
&\|E^{m}_{\u}\|_{\L^2}^2+\| E_{\B}^m\|_{\L^2}^2+\|E^{m}_{\theta}\|_{L^2}^2+\Delta t
\sum_{n=1}^{m}\Big\{\frac{\nu}{2}2^{2-p}\|\nabla E^n_{\u}\|_{\L^p}^p+\frac{\kappa}{2}\|\nabla\times E^n_{\B}\|_{\L^2}^2+\frac{\sigma}{2}\|\nabla E^{n}_{\theta}\|_{L^2}^2\\
&\qquad\qquad +(\Delta t)\|d_{t}E^{n}_{\u}\|^{2}_{\L^2}+(\Delta t)\|d_{t}E^{n}_{\B}\|^{2}_{\L^2}+(\Delta t)\|d_{t}E^{n}_{\theta}\|^{2}_{L^2}\Big\}\\
&\qquad\le C h^{2l}+c\,\Delta t\sum_{n=1}^{m}\Big(\|E_{\u}^n\|^{2}_{\L^2}+\|E_{\B}^n\|^{2}_{\L^2}+\|E_{\theta}^n\|^{2}_{L^2}\Big)+\Delta t\sum_{n=1}^{m}\lambda_n ,
\end{align*}
Here $Ch^{2l}$ collects the initial errors. The discrete Gr\"onwall lemma
\cite[Lemma~4.9]{ding2022convergence}, together with \eqref{eq-lambda-sum}, gives
\eqref{eq-main-error}. Finally, \eqref{eq-main-error-total} follows from the triangle
inequality and \eqref{eq-error approximation}.
   \end{proof}

\vspace{0.2cm}

\section{Numerical experiments}\label{sec-numerics}


We present four three-dimensional experiments illustrating the discrete dynamics, the energy
behaviour of Theorem~\ref{bounds}, and the temporal and space-time convergence predicted by
Theorem~\ref{Main theorem of error}.

\subsection{Common setting}\label{subsec-setting}

Throughout, $\mathbb{D}=(0,1)^{3}$ is discretised by a uniform tetrahedral mesh, with
$T=1$ and
\[
\nu=1,\qquad \sigma=1,\qquad \kappa=1, \qquad \boldsymbol{\omega}_r=(0,0,1) .
\]
We use continuous $P_{2}$ elements for velocity and temperature, continuous $P_{1}$ elements
for pressure, and second-order N\'ed\'elec edge elements for the magnetic induction. Hence
$k=2$ and $l=\min\{2,s\}$, with $s$ from Assumption~\ref{Assumption}.

\subsection{Qualitative behaviour of the discrete solution}\label{subsec-qualitative}

\begin{example}[Temporal profiles of $\u$, $\B$ and $\theta$]\label{ex-temporal}
For $p=3$, we track the fields on the plane $z=0.25$ over $t\in[0,1]$. The exact velocity is
defined as the curl of a scalar potential and is therefore pointwise divergence-free:
\[
\mathbf{u}=\nabla\times\left(0,0,\psi\right),
\qquad
\psi(x,y,z,t)=\sin^{2}(\pi x)\sin^{2}(\pi y)\sin^{2}(\pi z)\sin(\omega t+\pi x),
\]
that is,
\[
\mathbf{u}(x,y,z,t)=
\begin{pmatrix}
\pi \sin^{2}(\pi x)\,
\sin(2\pi y)\,
\sin^{2}(\pi z)\,
\sin(\omega t+\pi x)
\\[1.2ex]
\begin{aligned}
&-\pi \sin(2\pi x)\,\sin^{2}(\pi y)\,\sin^{2}(\pi z)\,\sin(\omega t+\pi x)\\
&\qquad-\pi \sin^{2}(\pi x)\,\sin^{2}(\pi y)\,\sin^{2}(\pi z)\,\cos(\omega t+\pi x)
\end{aligned}
\\[1.2ex]
0
\end{pmatrix}.
\]
with $\omega=7$; the double zeros of the $\sin^{2}$ factors ensure that $\u=0$ on
$\partial\mathbb D$. The magnetic field is likewise defined as a curl,
$$
\mathbf{b}(x,y,z,t)=
\begin{pmatrix}
\pi \sin^{2}(\pi x)\sin(2\pi y)\sin^{2}(\pi z)
\sin\!\left(2\pi\bigl(x\cos(10t)+y\sin(10t)\bigr)\right)
\\
\qquad
+2\pi\sin(10t)\sin^{2}(\pi x)\sin^{2}(\pi y)\sin^{2}(\pi z)
\cos\!\left(2\pi\bigl(x\cos(10t)+y\sin(10t)\bigr)\right)
\\[1.5ex]
-\pi\sin(2\pi x)\sin^{2}(\pi y)\sin^{2}(\pi z)
\sin\!\left(2\pi\bigl(x\cos(10t)+y\sin(10t)\bigr)\right)
\\
\qquad
-2\pi\cos(10t)\sin^{2}(\pi x)\sin^{2}(\pi y)\sin^{2}(\pi z)
\cos\!\left(2\pi\bigl(x\cos(10t)+y\sin(10t)\bigr)\right)
\\[1.5ex]
0
\end{pmatrix}.
$$
The temperature is
$$\theta=
\sin(\pi z)\Big[2\sin(\pi x)\sin(\pi y)\cos(4t)
+
\frac{3}{2}\sin(3\pi x)\sin(2\pi y)\sin(7t)\Big].
$$

Figures~\ref{fig:solutions_different_times}-\ref{fig:theta_solutions_different_times} show
the computed velocity, magnetic field and temperature. The velocity forms counter-rotating
cells, the magnetic field follows the rotating phase
$2\pi(x\cos 10t+y\sin 10t)$, and the temperature changes from a centred hot spot to a
cellular pattern and then to a cold core. The convergence rates are examined in
Sections~\ref{subsec-temporal-conv} and \ref{subsec-spacetime-conv}.

\begin{figure}[!htbp]
    \centering

    \begin{subfigure}[b]{0.3\textwidth}
        \centering
        \includegraphics[width=\textwidth]{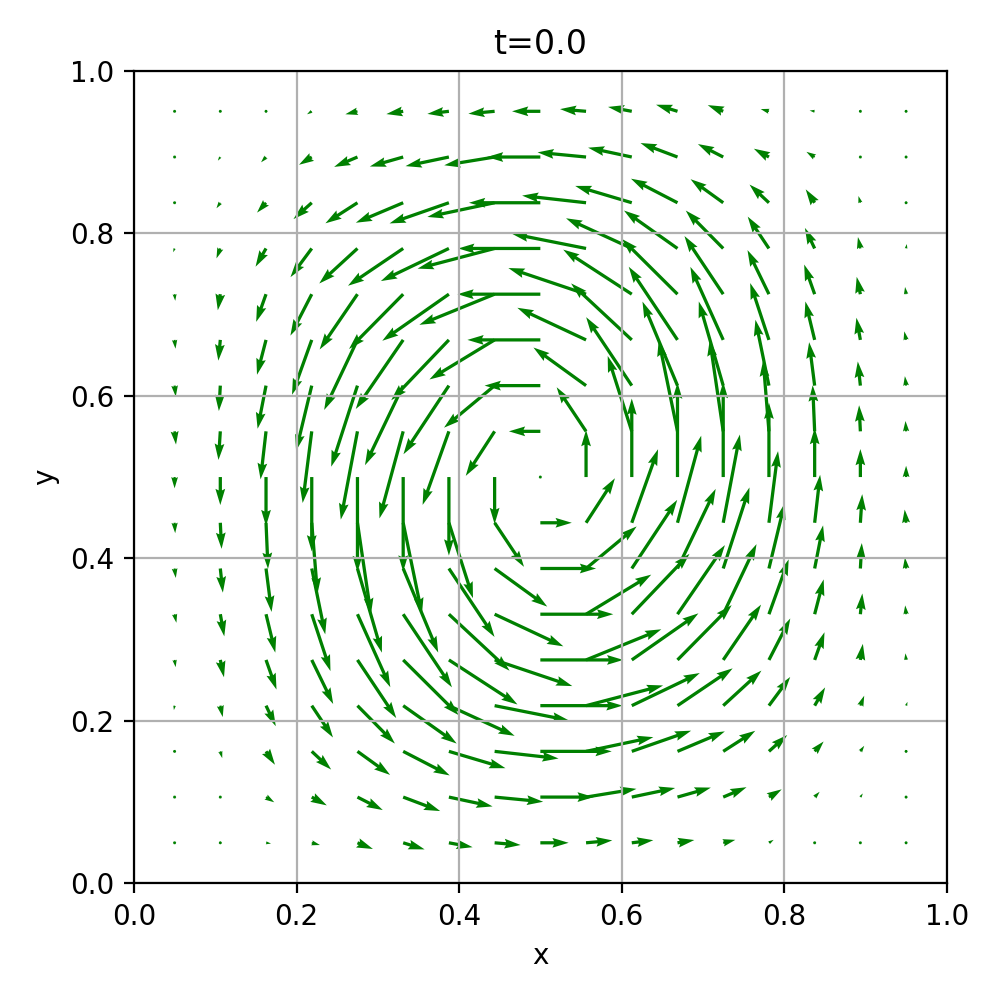}
        \caption{$t=0.0$}
        \label{fig:tu00}
    \end{subfigure}
    \hfill
    \begin{subfigure}[b]{0.3\textwidth}
        \centering
        \includegraphics[width=\textwidth]{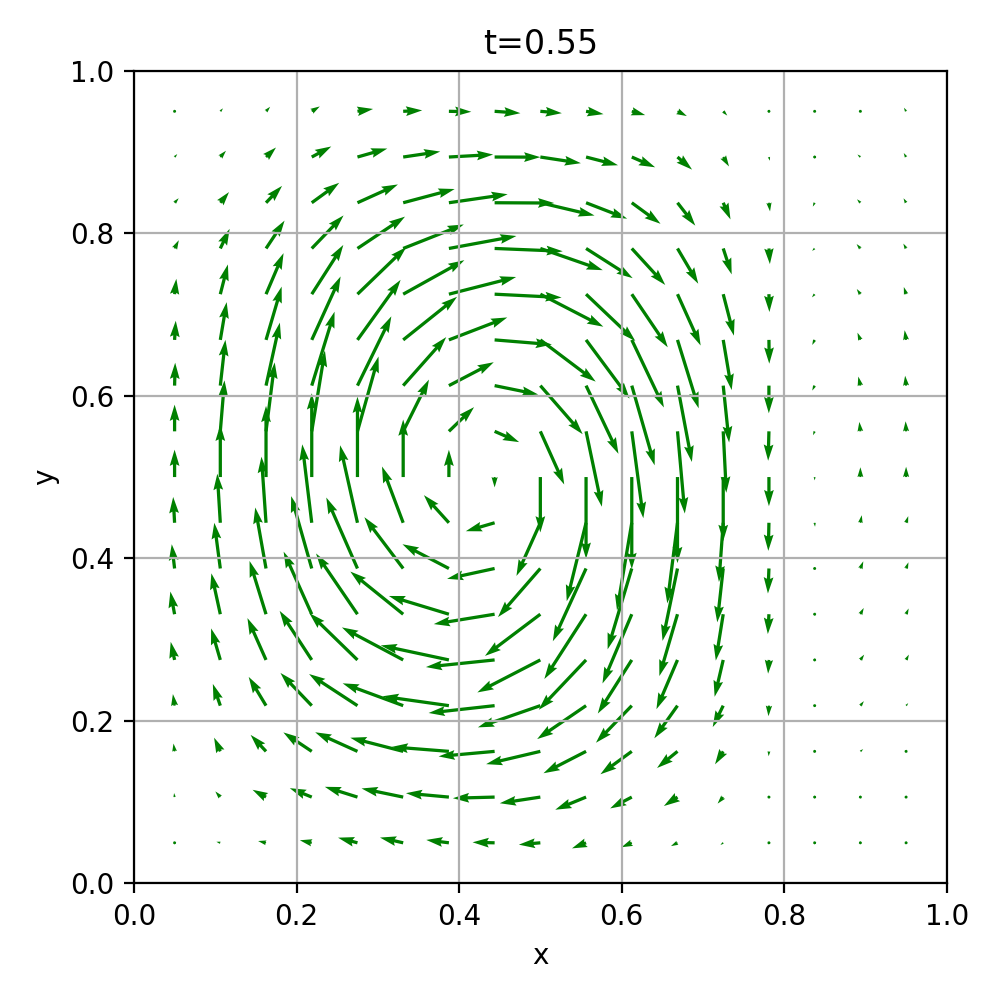}
        \caption{$t=0.55$}
        \label{fig:t055}
    \end{subfigure}
    \hfill
    \begin{subfigure}[b]{0.3\textwidth}
        \centering
        \includegraphics[width=\textwidth]{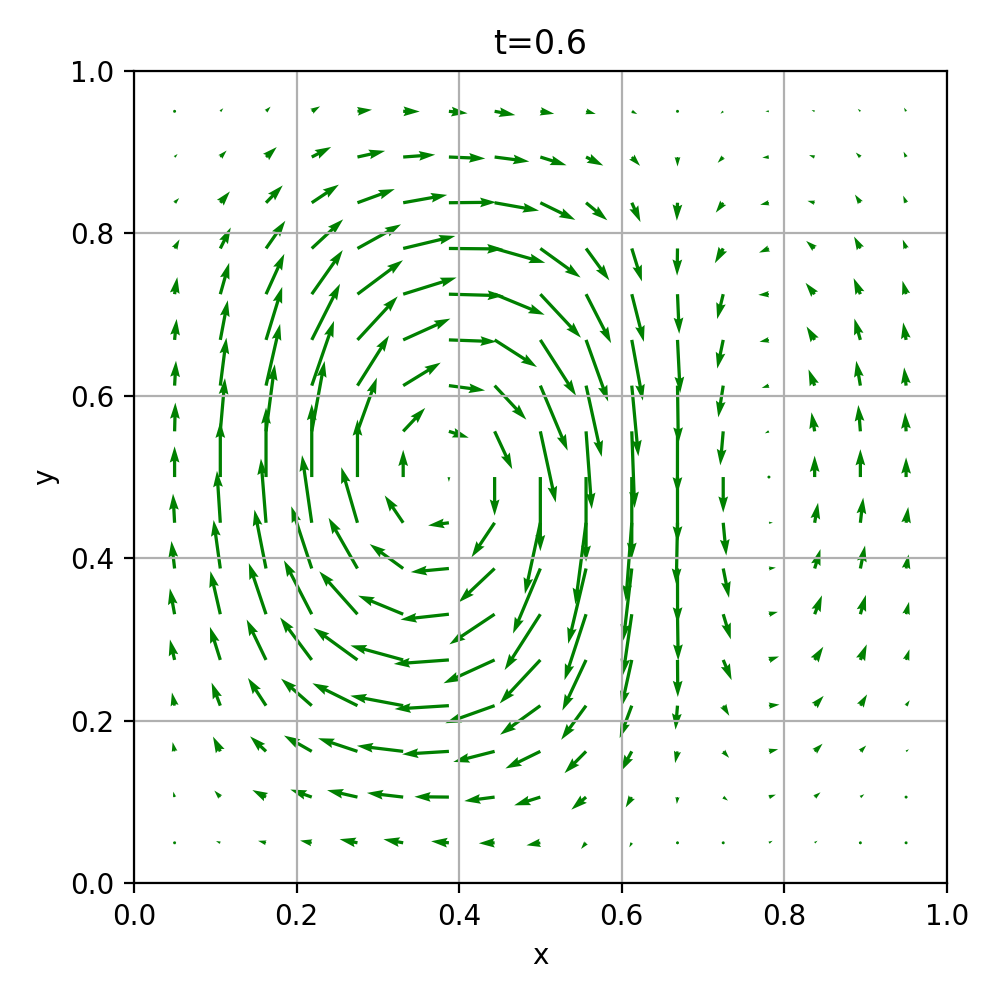}
        \caption{$t=0.6$}
        \label{fig:t06}
    \end{subfigure}

    \vspace{0.4cm}

    \begin{subfigure}[b]{0.3\textwidth}
        \centering
        \includegraphics[width=\textwidth]{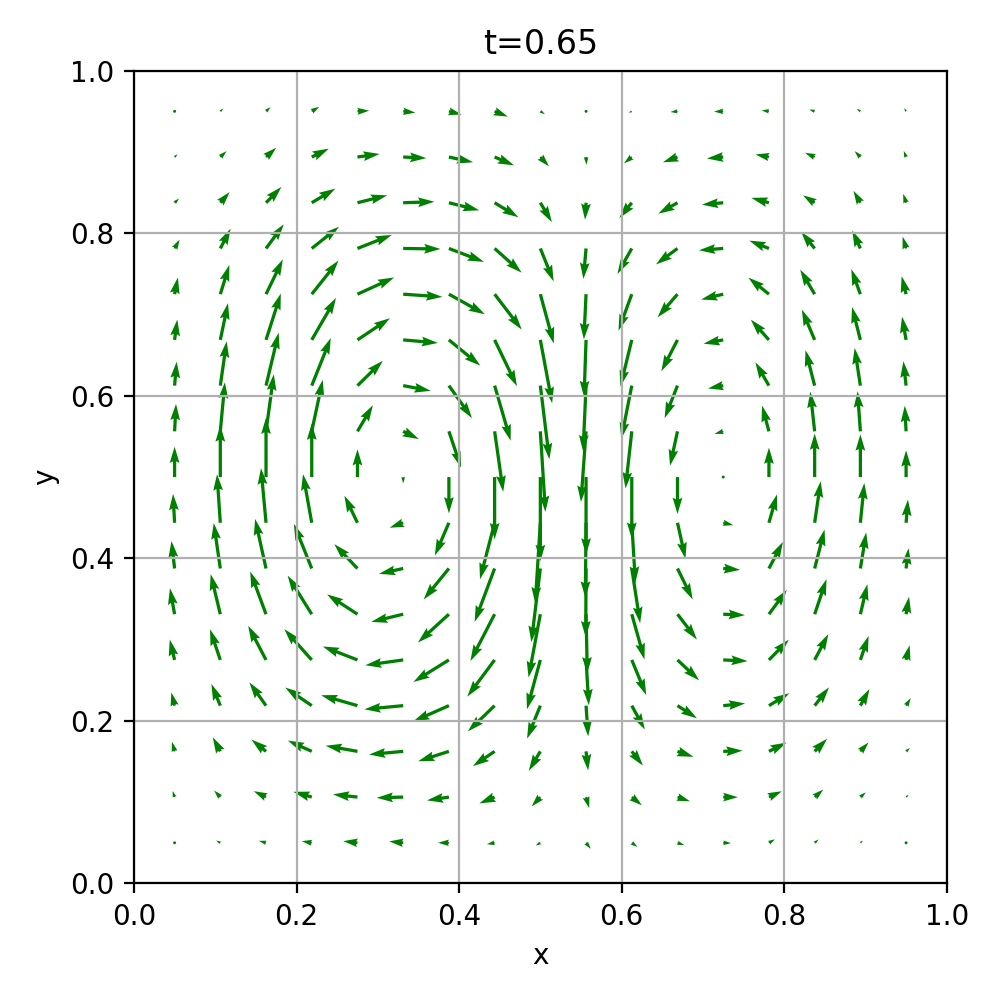}
        \caption{$t=0.65$}
        \label{fig:t065}
    \end{subfigure}
    \hfill
    \begin{subfigure}[b]{0.3\textwidth}
        \centering
        \includegraphics[width=\textwidth]{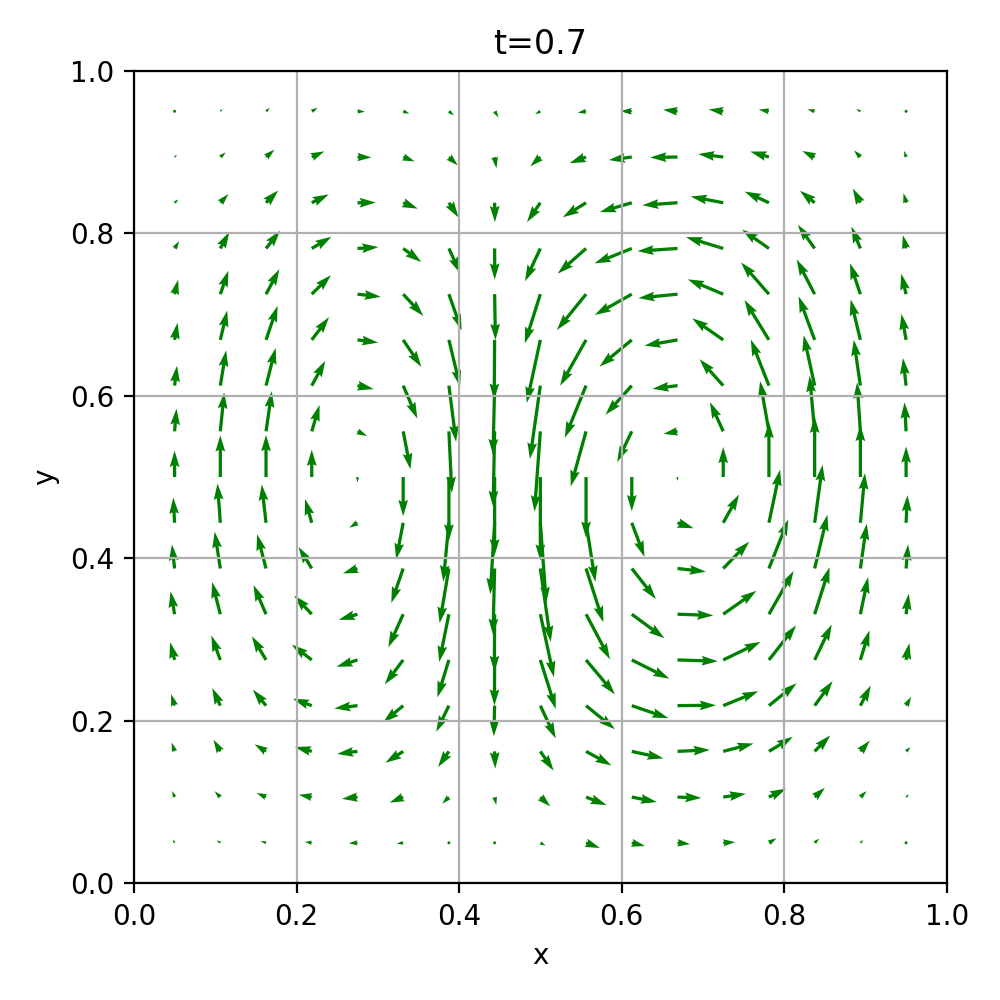}
        \caption{$t=0.7$}
        \label{fig:t07}
    \end{subfigure}
    \hfill
    \begin{subfigure}[b]{0.3\textwidth}
        \centering
        \includegraphics[width=\textwidth]{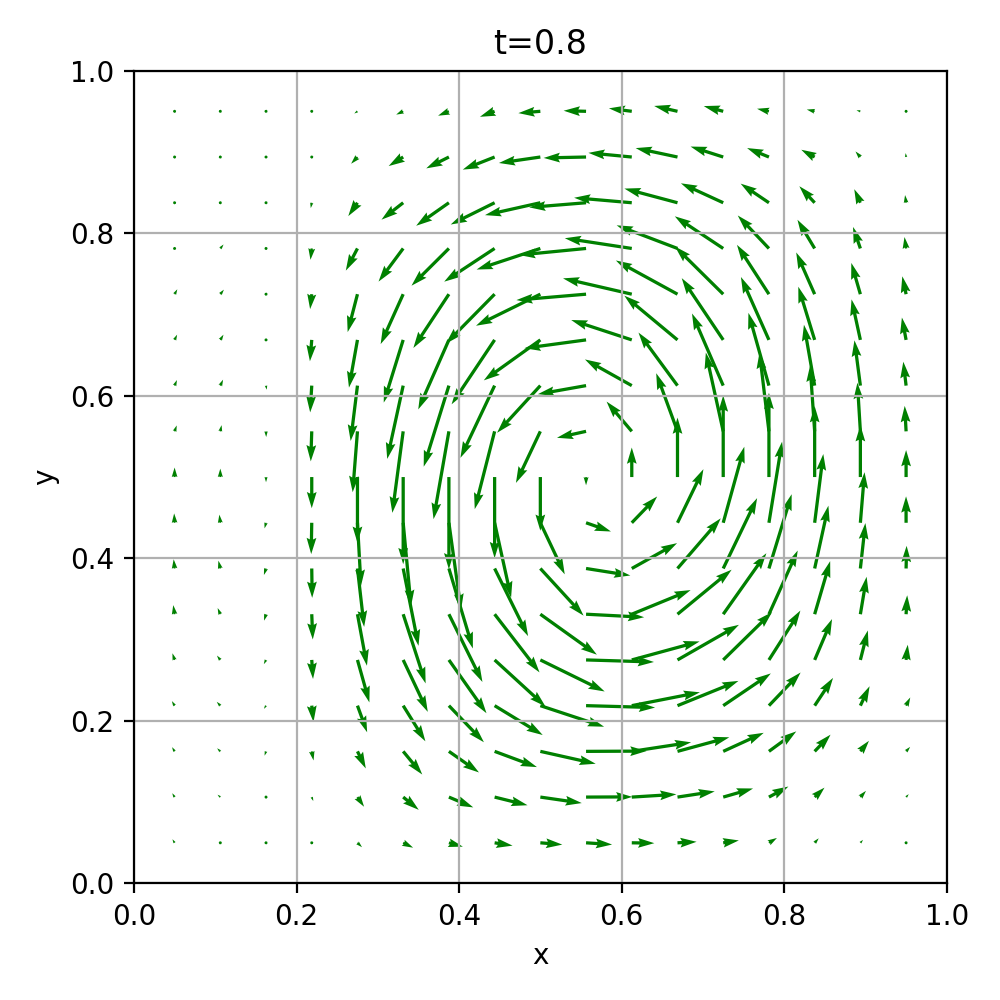}
        \caption{$t=0.8$}
        \label{fig:t08}
    \end{subfigure}
  \vspace{0.4cm}

    \begin{subfigure}[b]{0.3\textwidth}
        \centering
        \includegraphics[width=\textwidth]{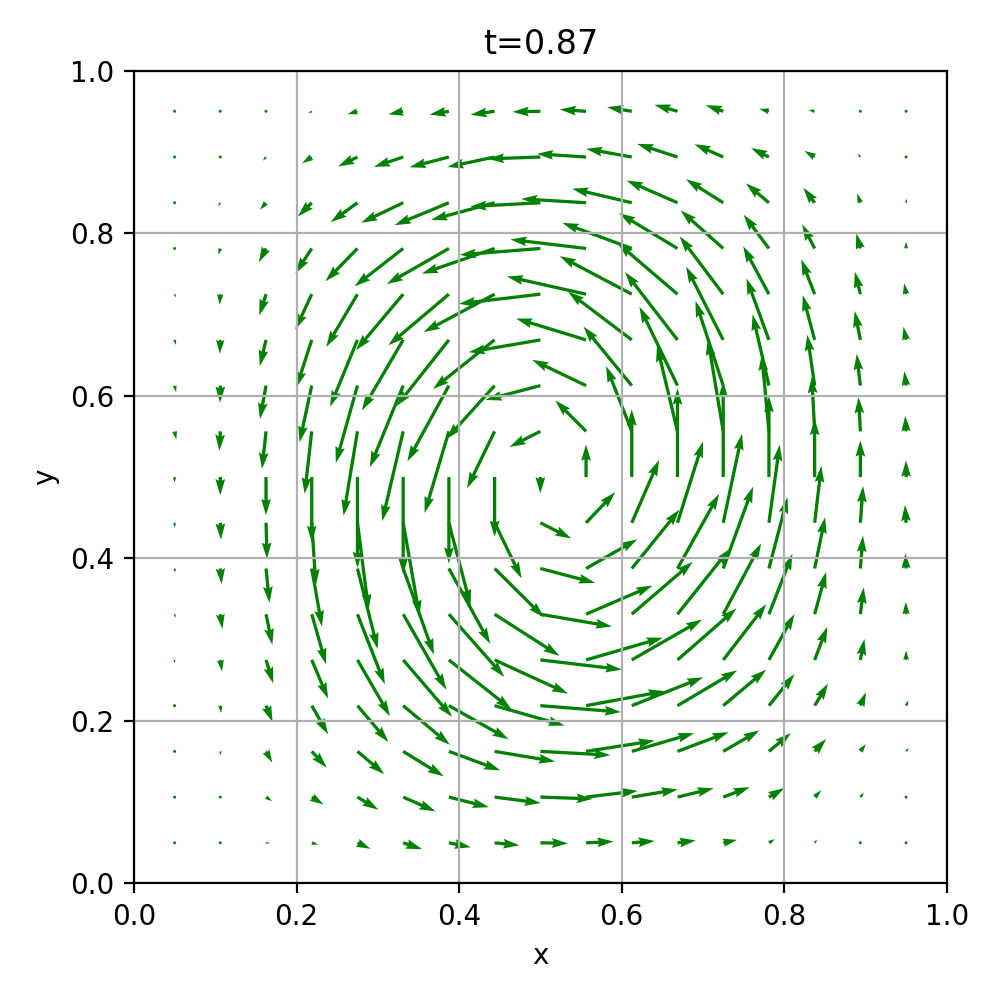}
        \caption{$t=0.87$}
        \label{fig:t087}
    \end{subfigure}
    \hfill
      \begin{subfigure}[b]{0.3\textwidth}
        \centering
        \includegraphics[width=\textwidth]{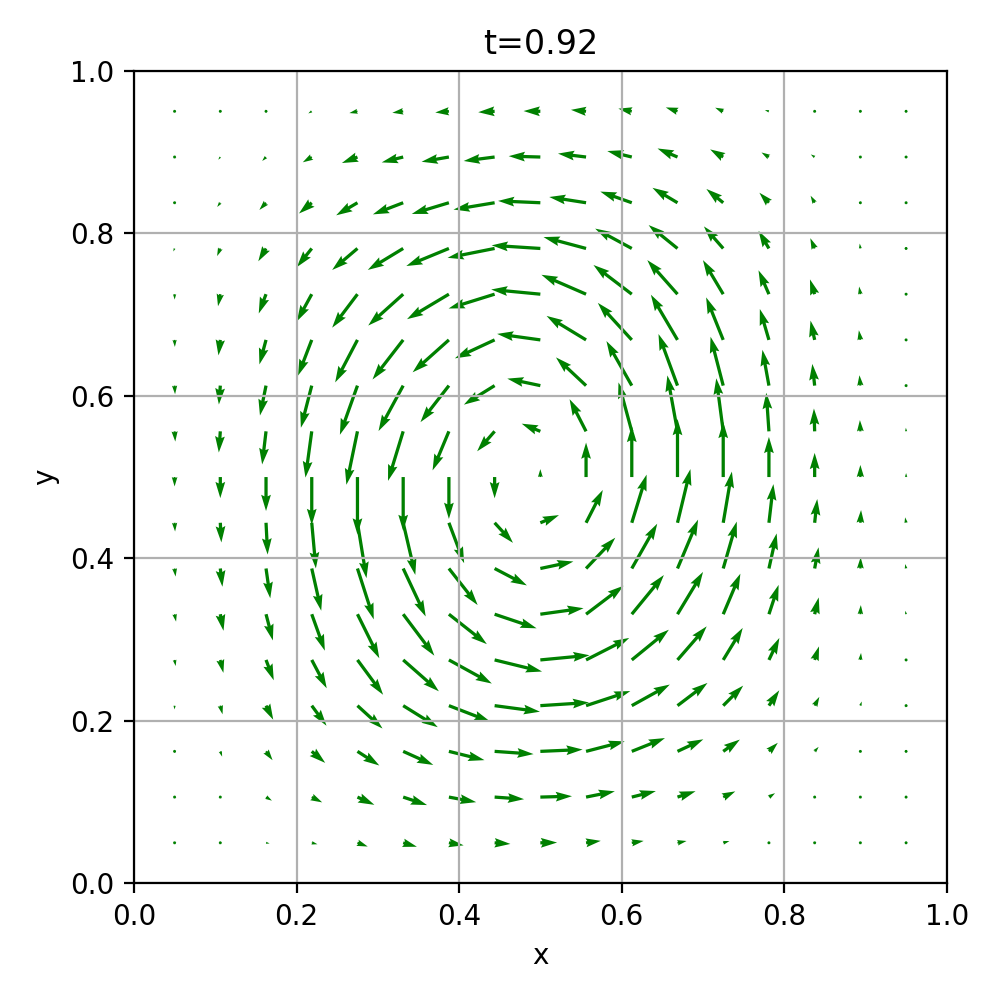}
        \caption{$t=0.92$}
        \label{fig:t092}
    \end{subfigure}
    \hfill
    \begin{subfigure}[b]{0.3\textwidth}
        \centering
        \includegraphics[width=\textwidth]{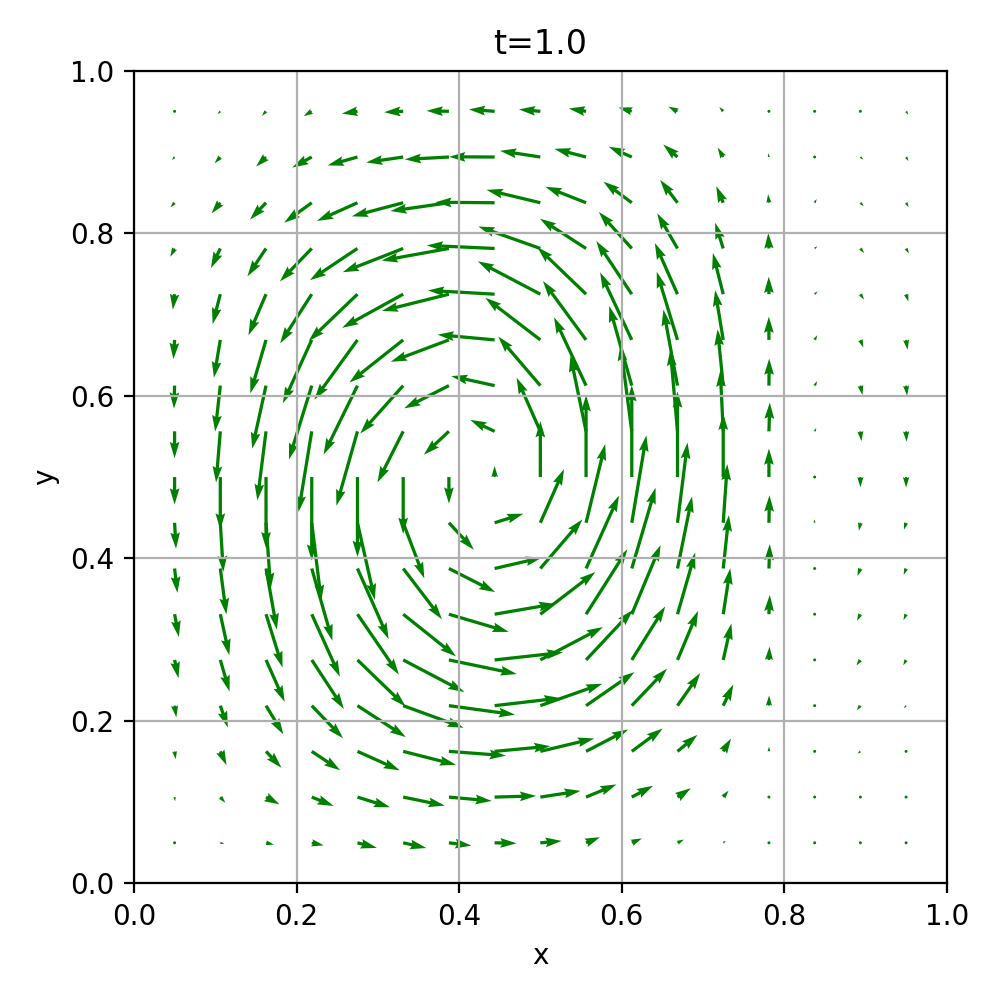}
        \caption{$t=1.0$}
        \label{fig:t10}
    \end{subfigure}
    \caption{Temporal evolution of the computed velocity field $\u$ (shown in green) at the plane $z=0.25$ for different time instances.}
    \label{fig:solutions_different_times}
\end{figure}

\begin{figure}[!htbp]
    \centering

    \begin{subfigure}[b]{0.3\textwidth}
        \centering
        \includegraphics[width=\textwidth]{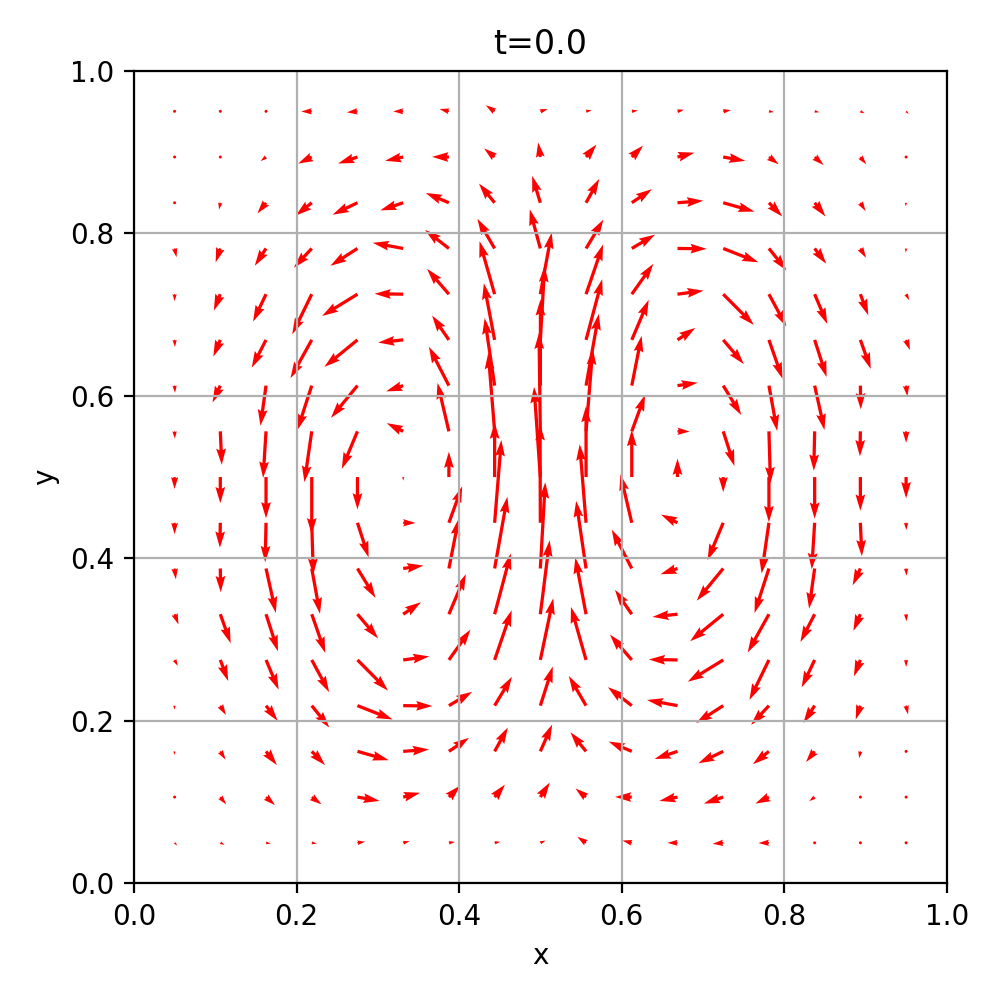}
        \caption{$t=0.0$}
        \label{fig:Bt00}
    \end{subfigure}
    \hfill
    \begin{subfigure}[b]{0.3\textwidth}
        \centering
        \includegraphics[width=\textwidth]{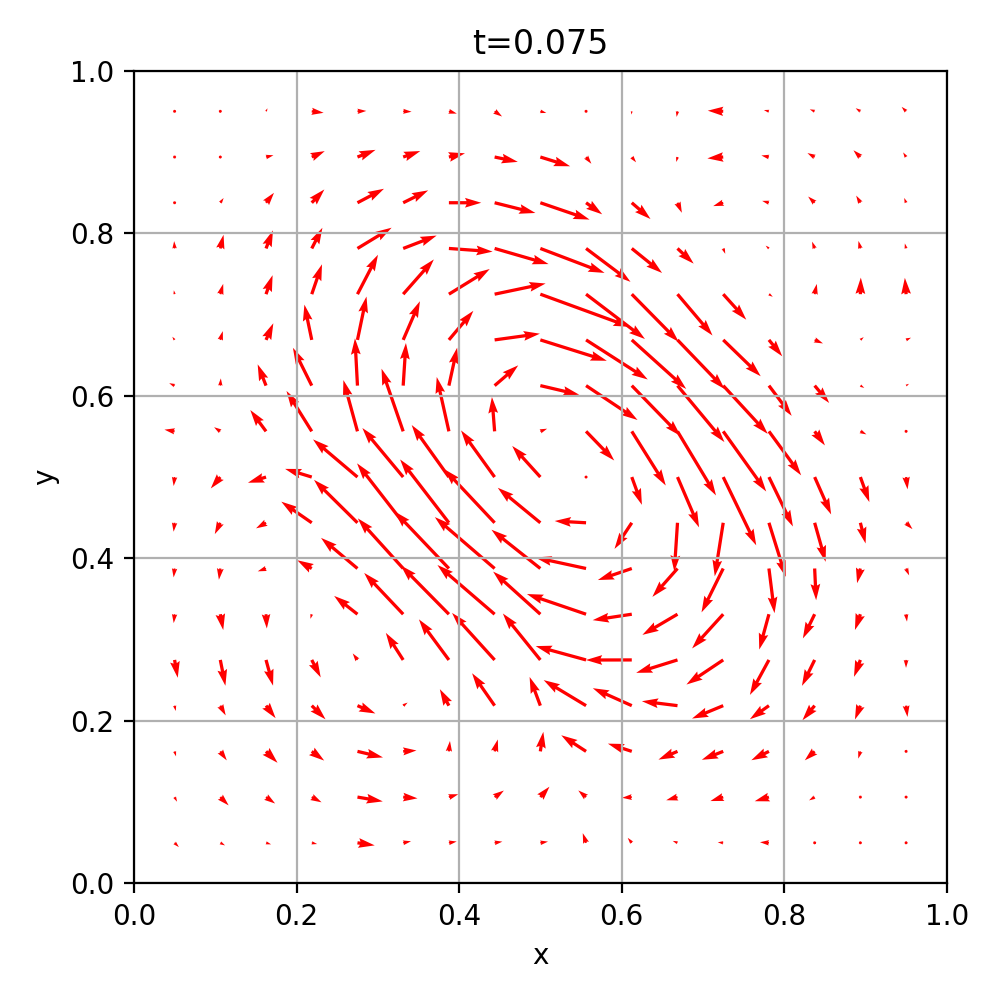}
        \caption{$t=0.075$}
        \label{fig:Bt0075}
    \end{subfigure}
    \hfill
    \begin{subfigure}[b]{0.3\textwidth}
        \centering
        \includegraphics[width=\textwidth]{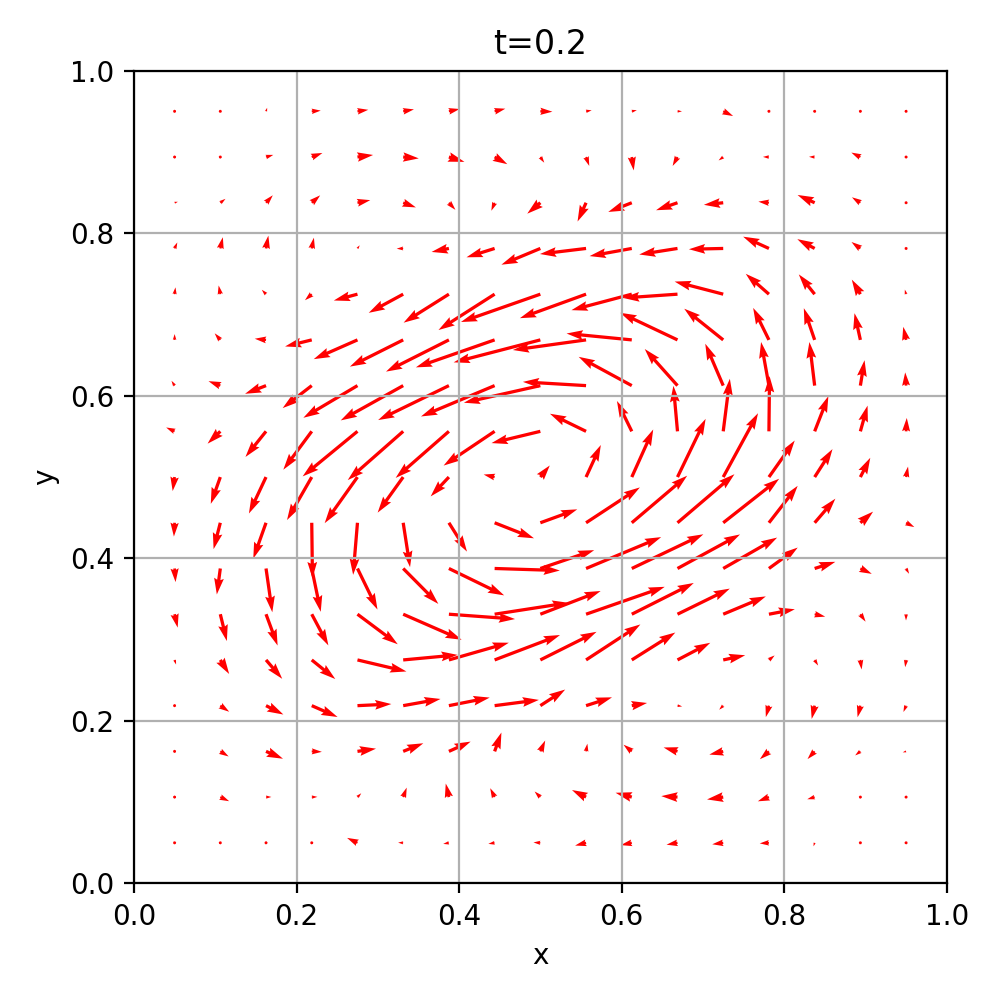}
        \caption{$t=0.2$}
        \label{fig:Bt02}
    \end{subfigure}

    \vspace{0.4cm}

    \begin{subfigure}[b]{0.3\textwidth}
        \centering
        \includegraphics[width=\textwidth]{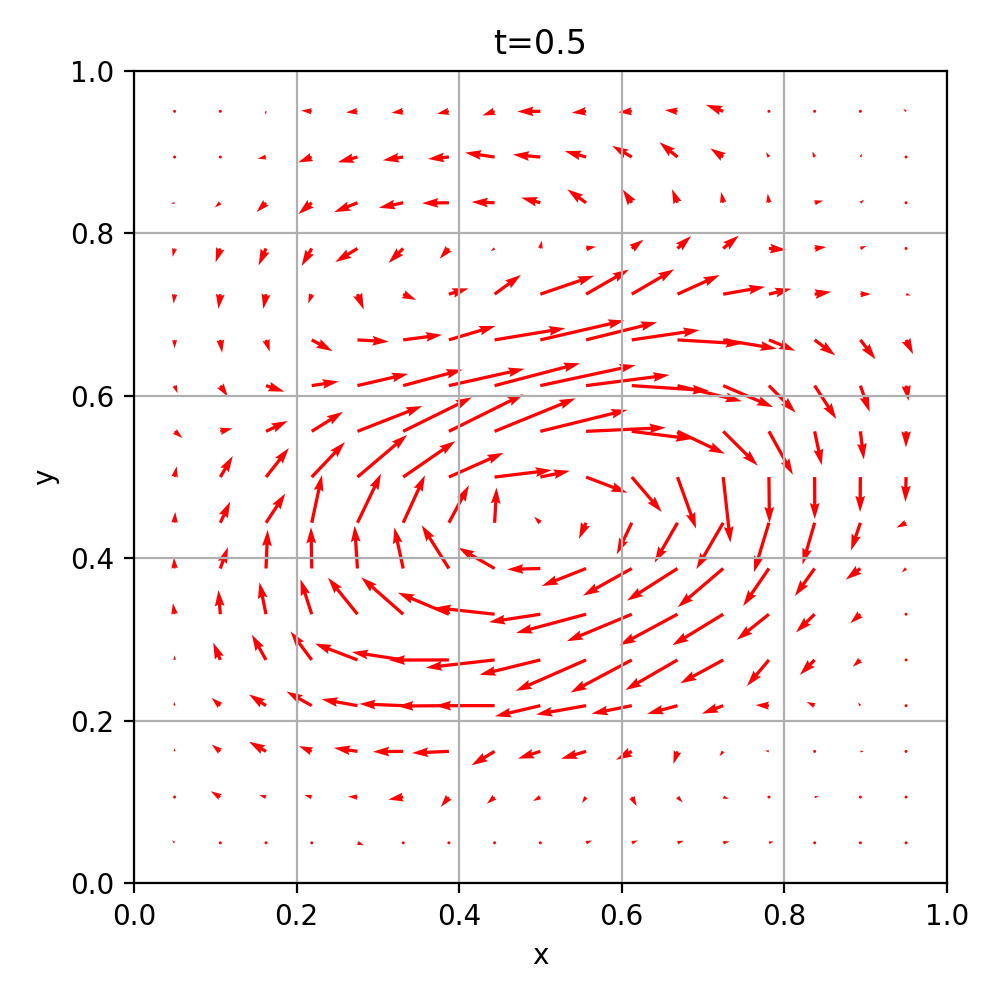}
        \caption{$t=0.5$}
        \label{fig:Bt05}
    \end{subfigure}
    \hfill
    \begin{subfigure}[b]{0.3\textwidth}
        \centering
        \includegraphics[width=\textwidth]{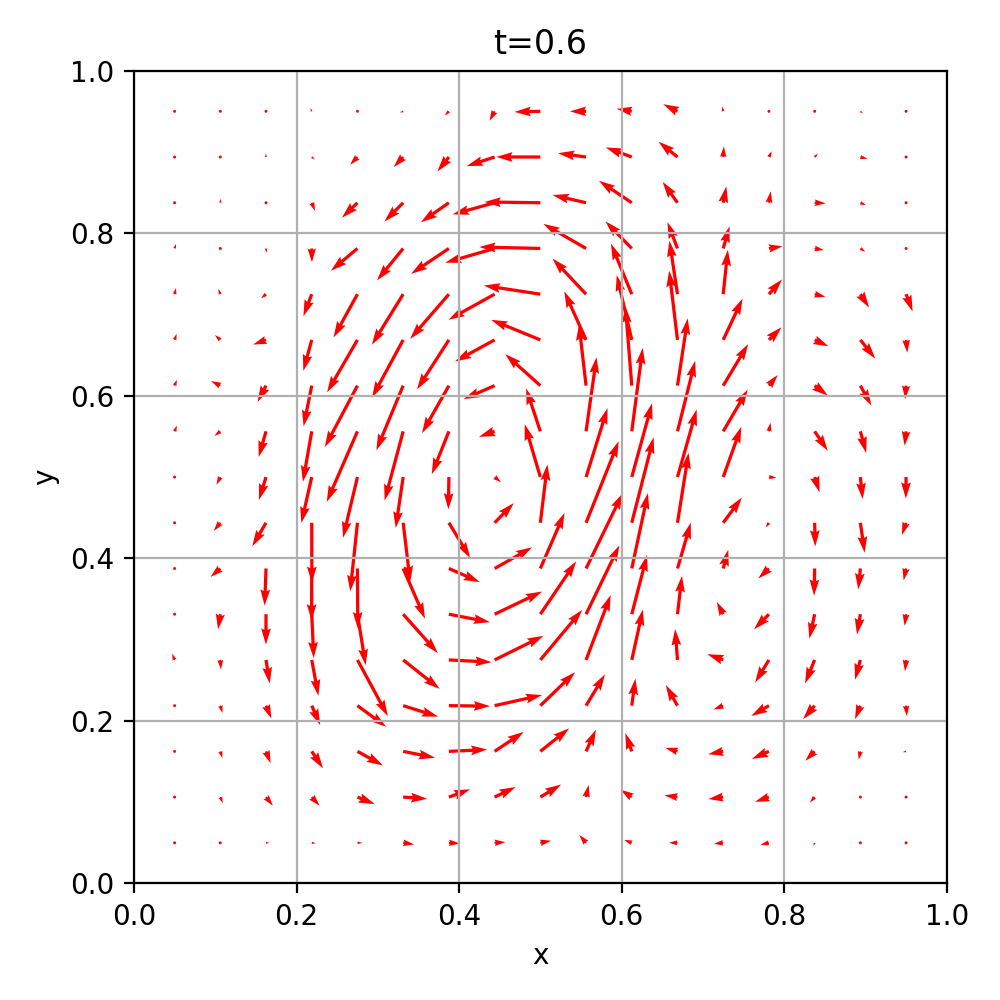}
        \caption{$t=0.6$}
        \label{fig:Bt06}
    \end{subfigure}
    \hfill
    \begin{subfigure}[b]{0.3\textwidth}
        \centering
        \includegraphics[width=\textwidth]{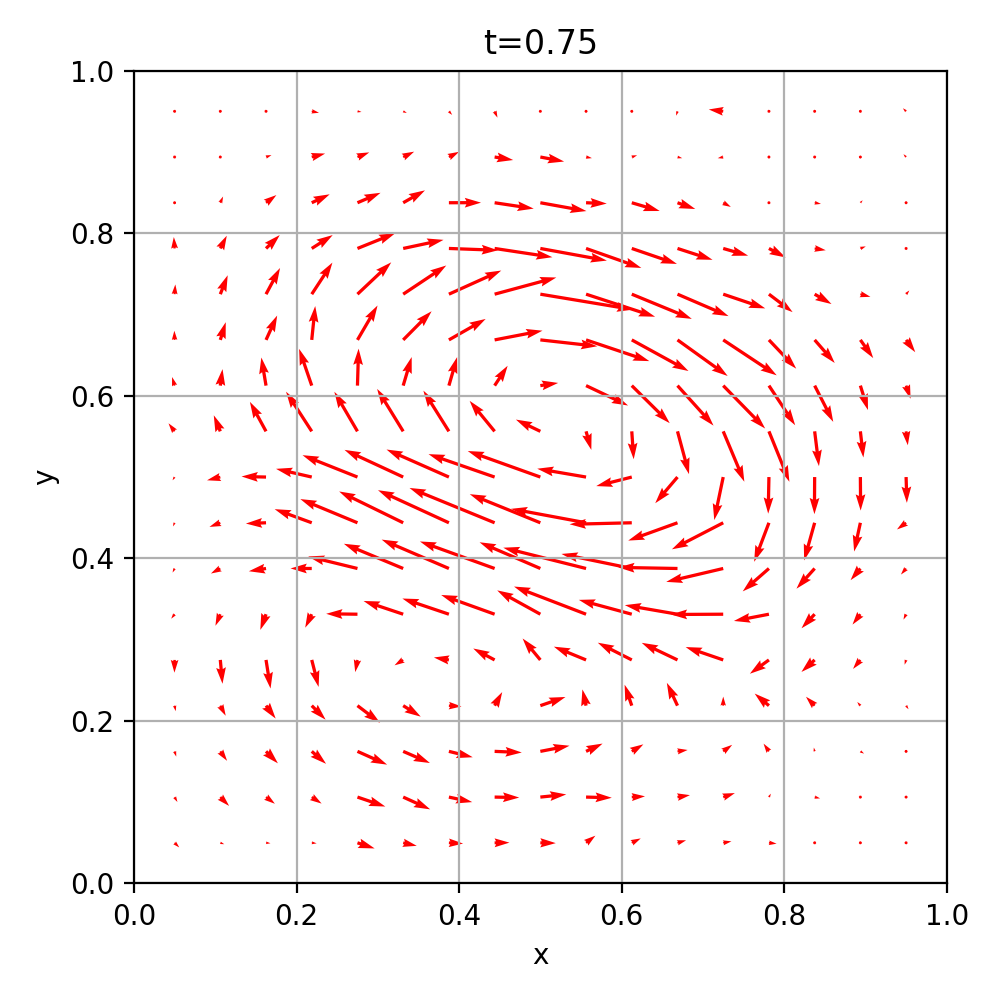}
        \caption{$t=0.75$}
        \label{fig:Bt075}
    \end{subfigure}
    \begin{subfigure}[b]{0.3\textwidth}
        \centering
        \includegraphics[width=\textwidth]{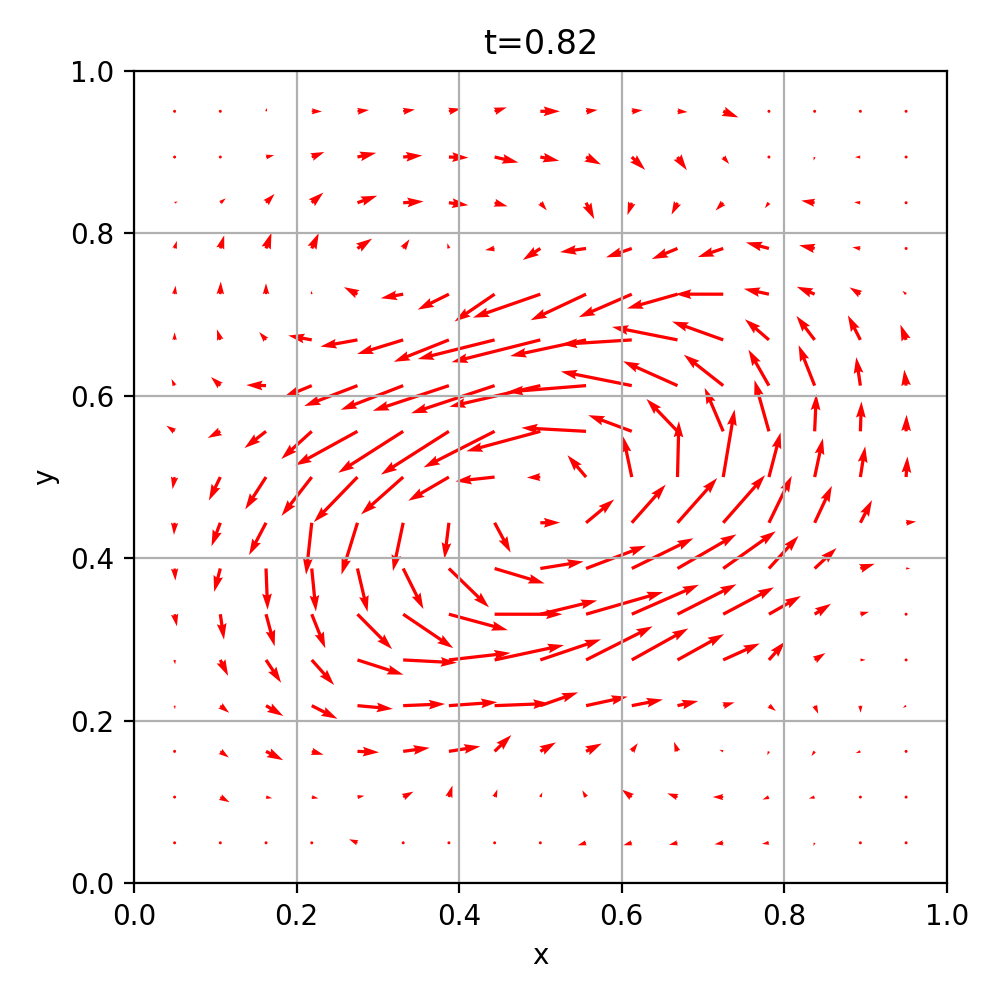}
        \caption{$t=0.82$}
        \label{fig:Bt082}
    \end{subfigure}
    \hfill
    \begin{subfigure}[b]{0.3\textwidth}
        \centering
        \includegraphics[width=\textwidth]{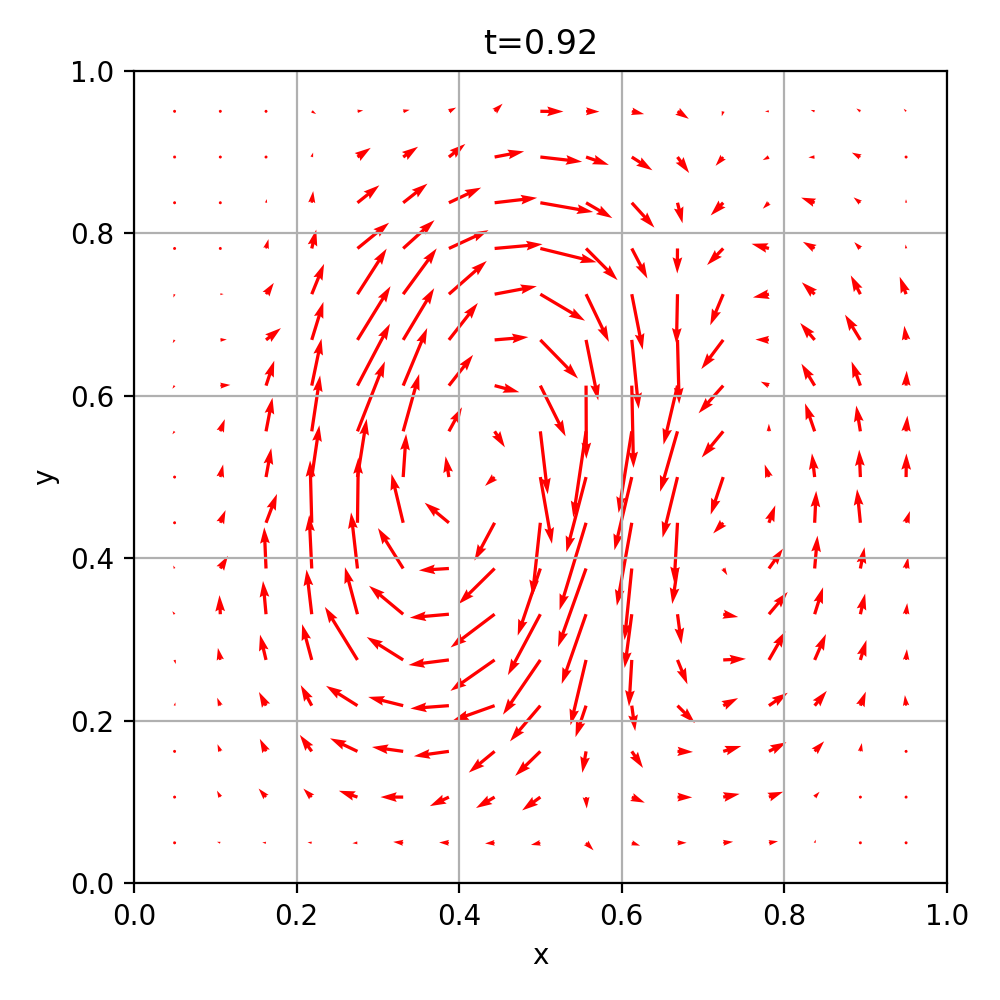}
        \caption{$t=0.92$}
        \label{fig:Bt092}
    \end{subfigure}
    \hfill
    \begin{subfigure}[b]{0.3\textwidth}
        \centering
        \includegraphics[width=\textwidth]{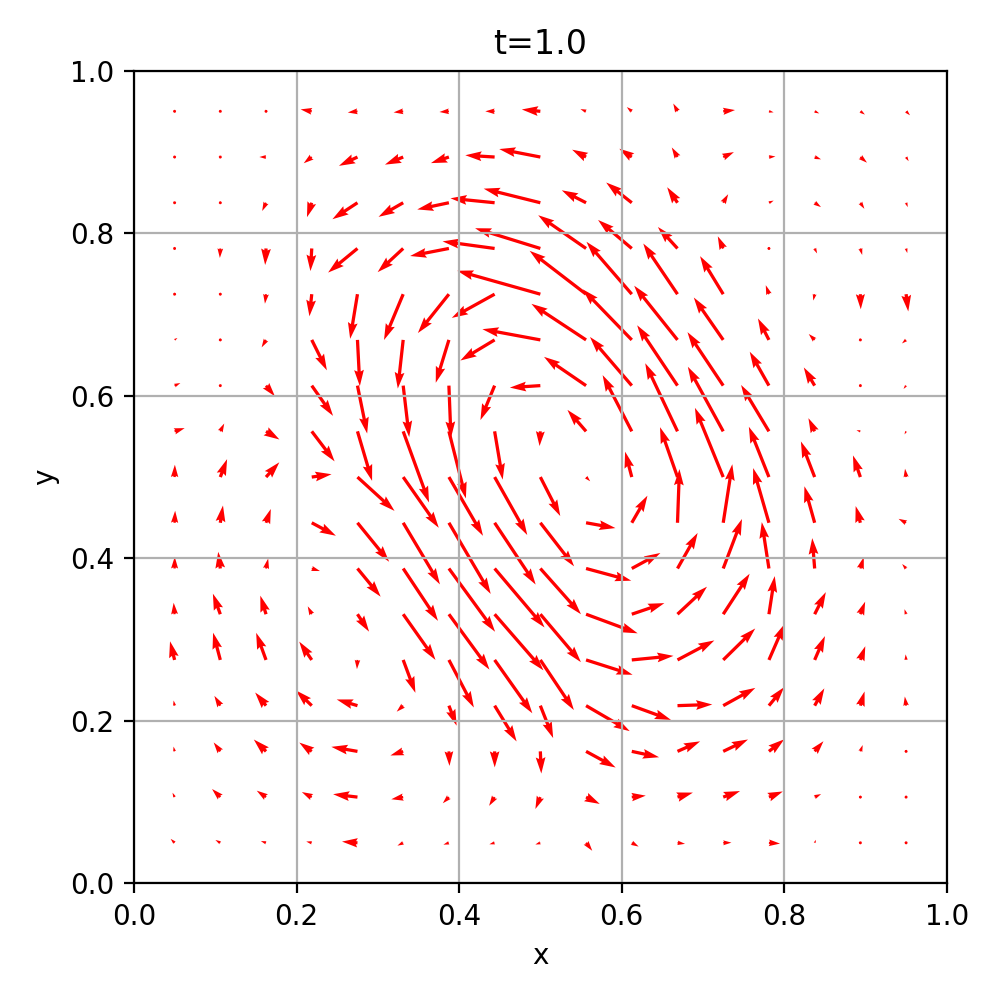}
        \caption{$t=1.0$}
        \label{fig:Bt10}
    \end{subfigure}

    \caption{Temporal evolution of the computed magnetic field $\B$ (shown in red) at the plane $z=0.25$ for different time instances.}
    \label{fig:B_solutions_different_times}
\end{figure}
\begin{figure}[!htbp]
    \centering

    \begin{subfigure}[b]{0.3\textwidth}
        \centering
        \includegraphics[width=\textwidth]{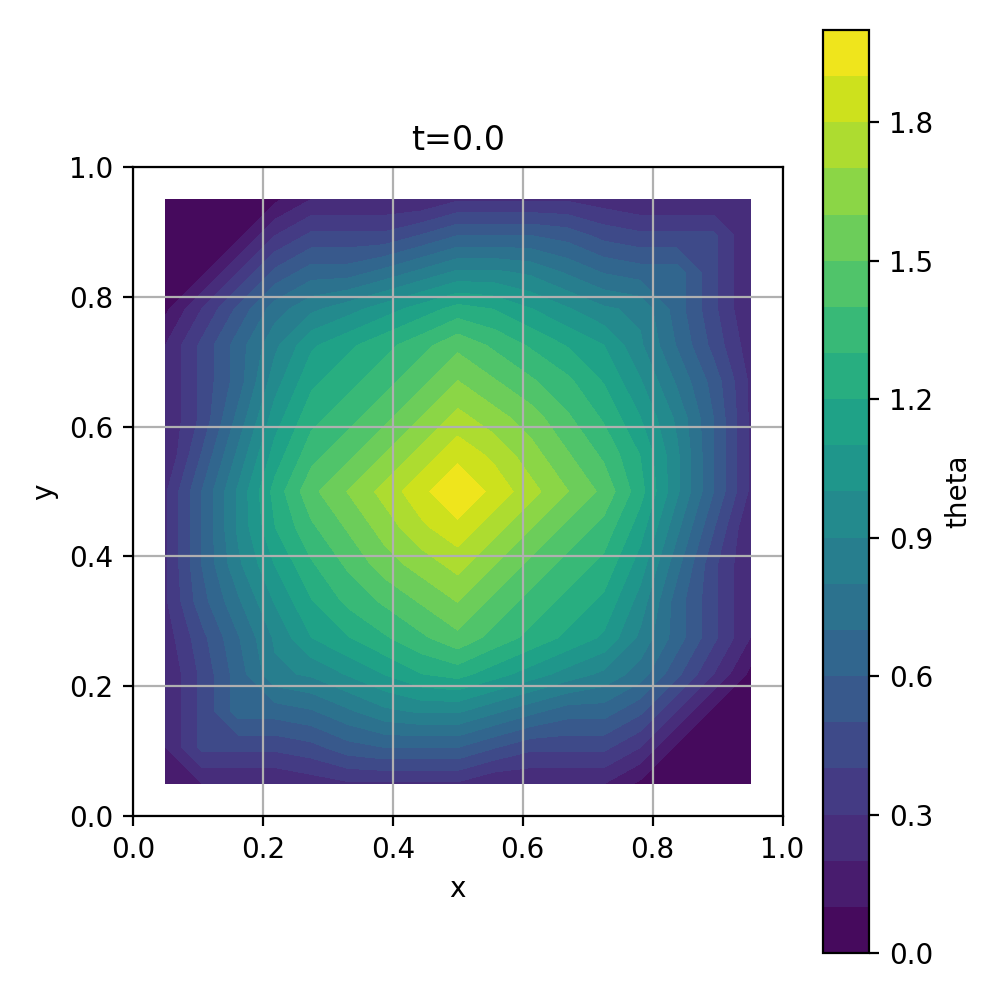}
        \caption{$t=0.0$}
        \label{fig:Ct00}
    \end{subfigure}
    \hfill
    \begin{subfigure}[b]{0.3\textwidth}
        \centering
        \includegraphics[width=\textwidth]{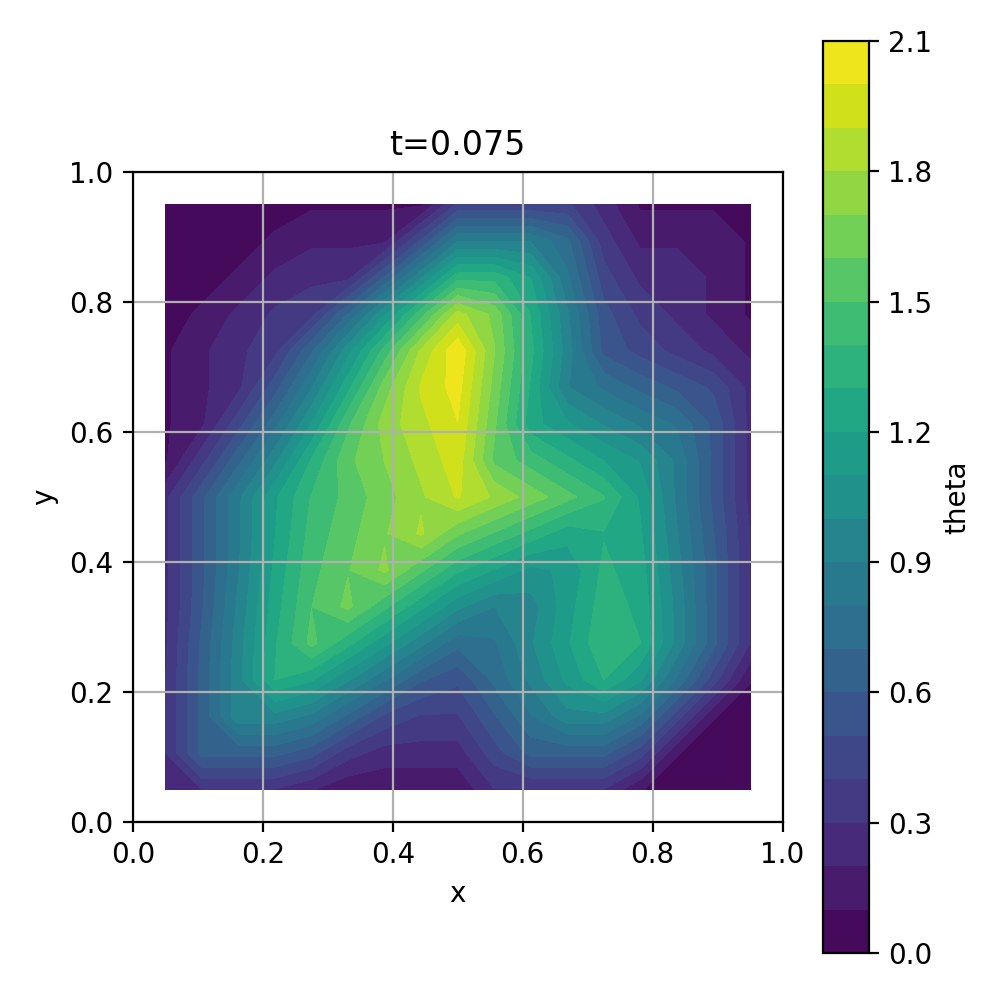}
        \caption{$t=0.075$}
        \label{fig:Ct0075}
    \end{subfigure}
    \hfill
    \begin{subfigure}[b]{0.3\textwidth}
        \centering
        \includegraphics[width=\textwidth]{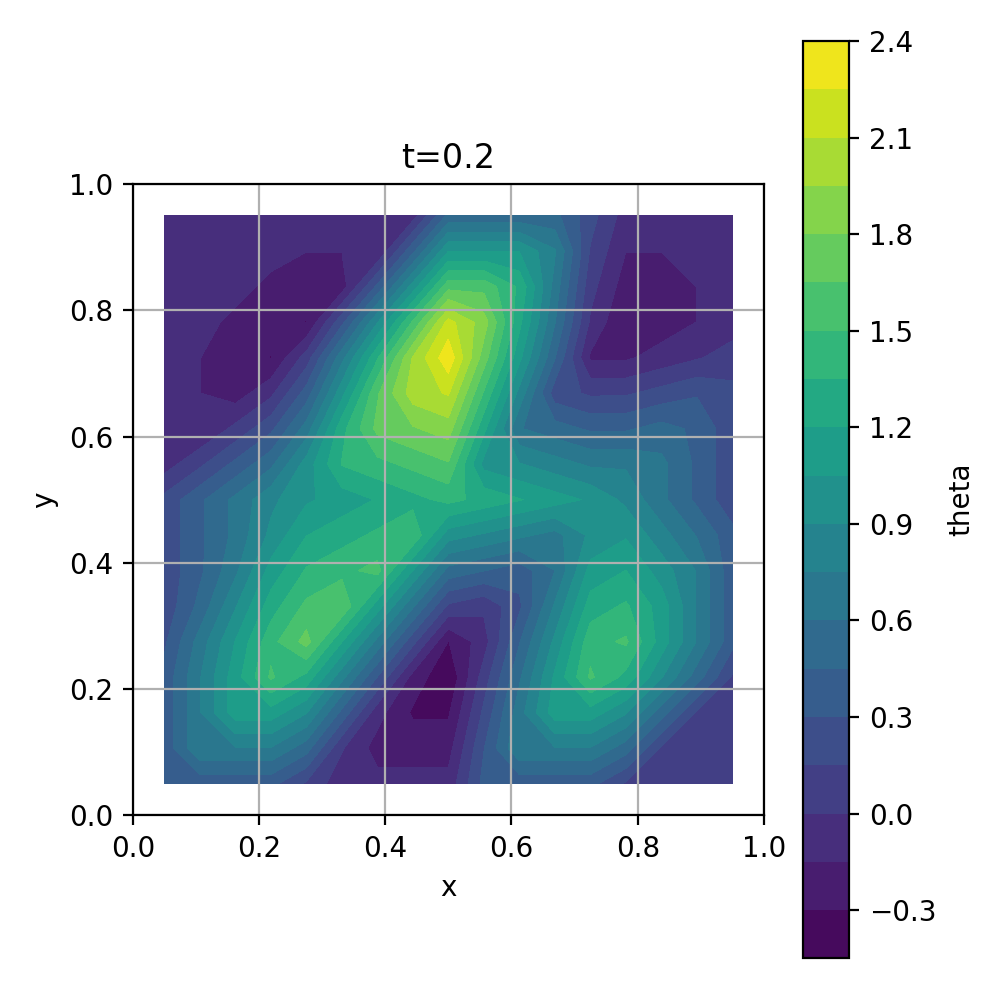}
        \caption{$t=0.2$}
        \label{fig:Ct02}
    \end{subfigure}

    \vspace{0.4cm}

    \begin{subfigure}[b]{0.3\textwidth}
        \centering
        \includegraphics[width=\textwidth]{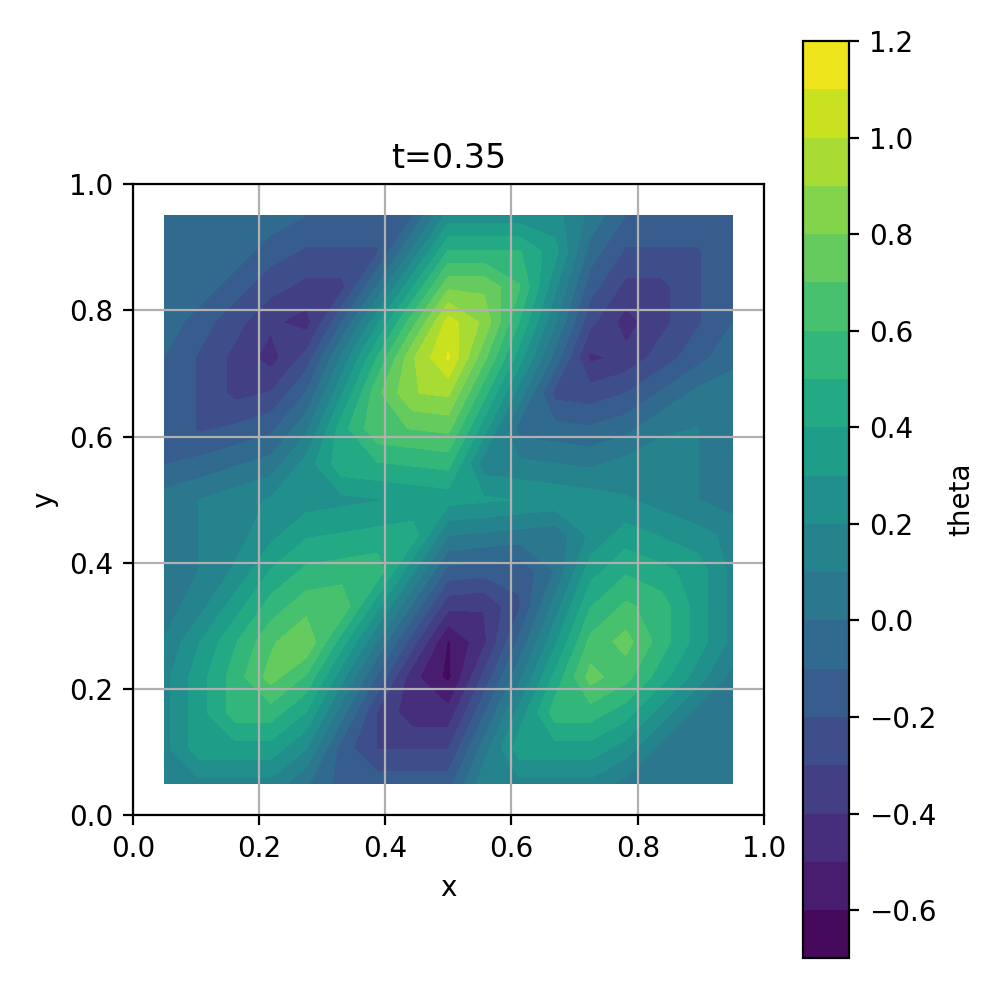}
        \caption{$t=0.35$}
        \label{fig:Ct035}
    \end{subfigure}
   \hfill
     \begin{subfigure}[b]{0.3\textwidth}
        \centering
        \includegraphics[width=\textwidth]{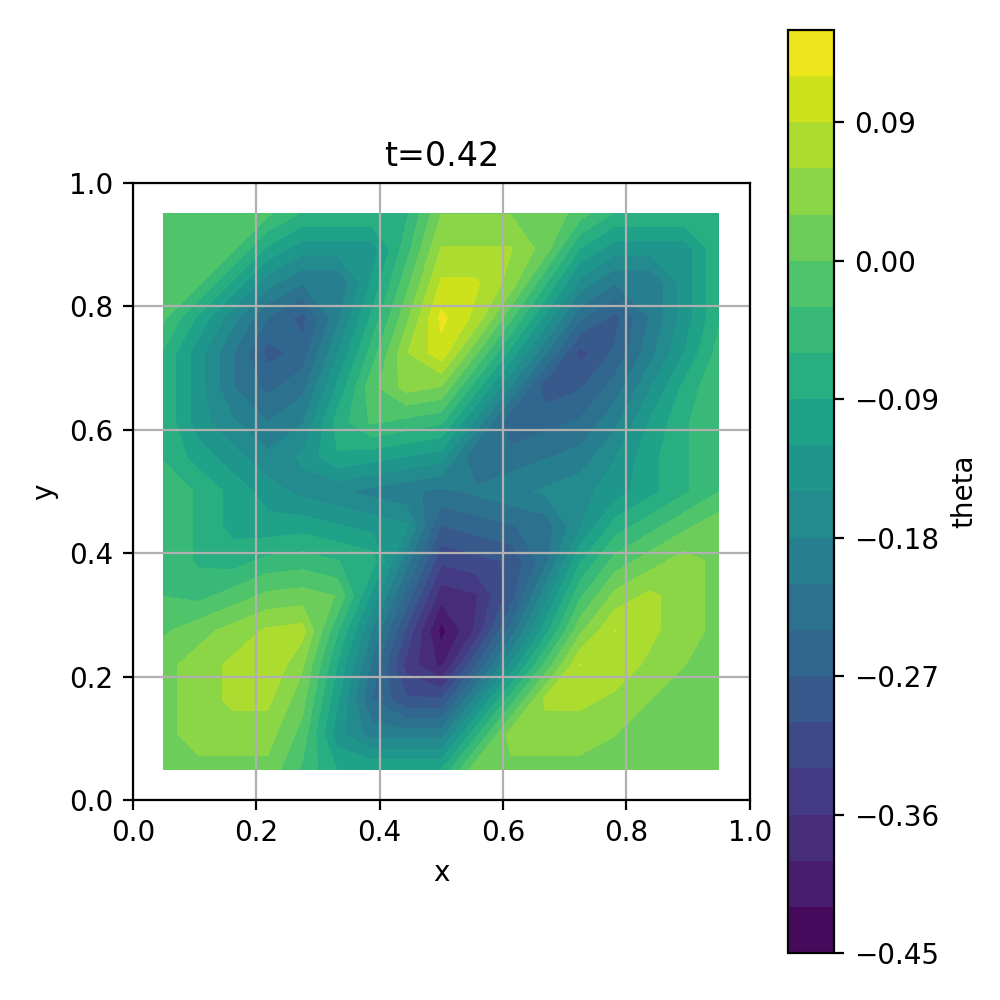}
        \caption{$t=0.42$}
        \label{fig:Ct042}
    \end{subfigure}
   \hfill
    \begin{subfigure}[b]{0.3\textwidth}
        \centering
        \includegraphics[width=\textwidth]{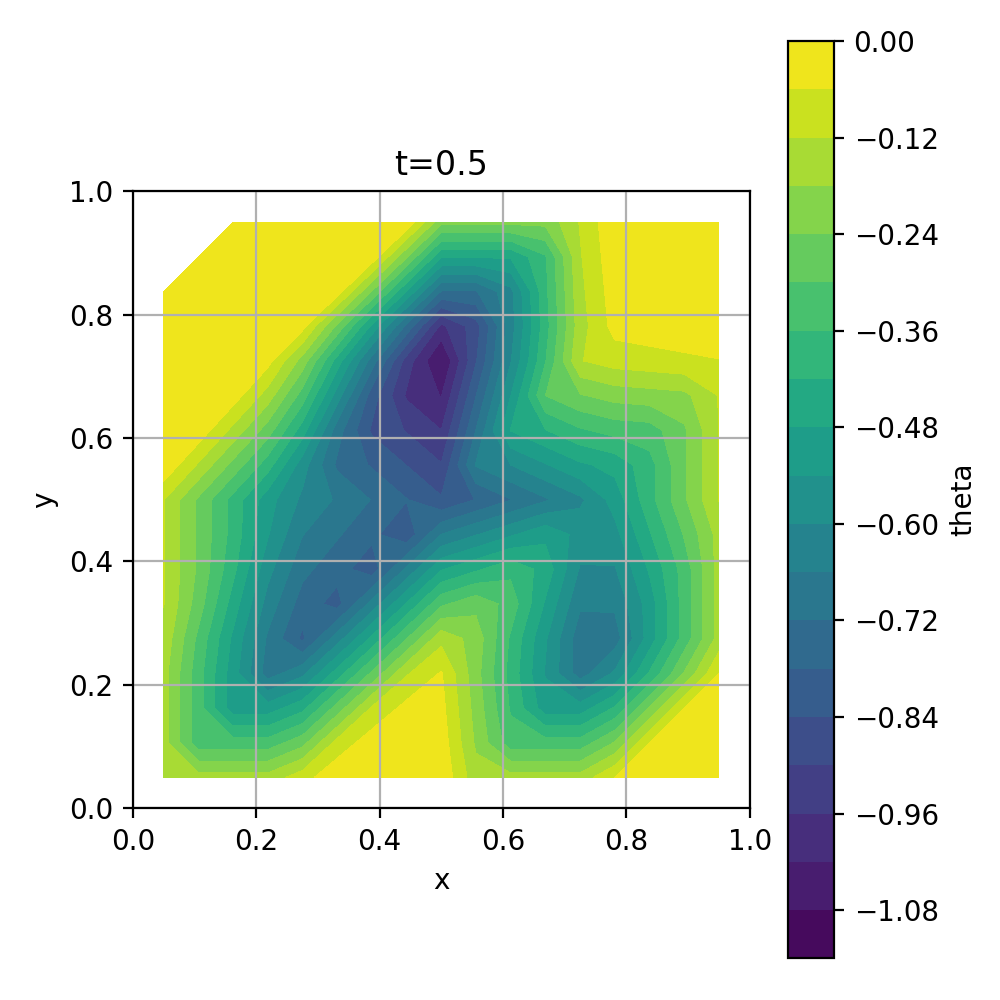}
        \caption{$t=0.5$}
        \label{fig:Ct05}
    \end{subfigure}
    
    \vspace{0.4cm}

    \begin{subfigure}[b]{0.3\textwidth}
        \centering
        \includegraphics[width=\textwidth]{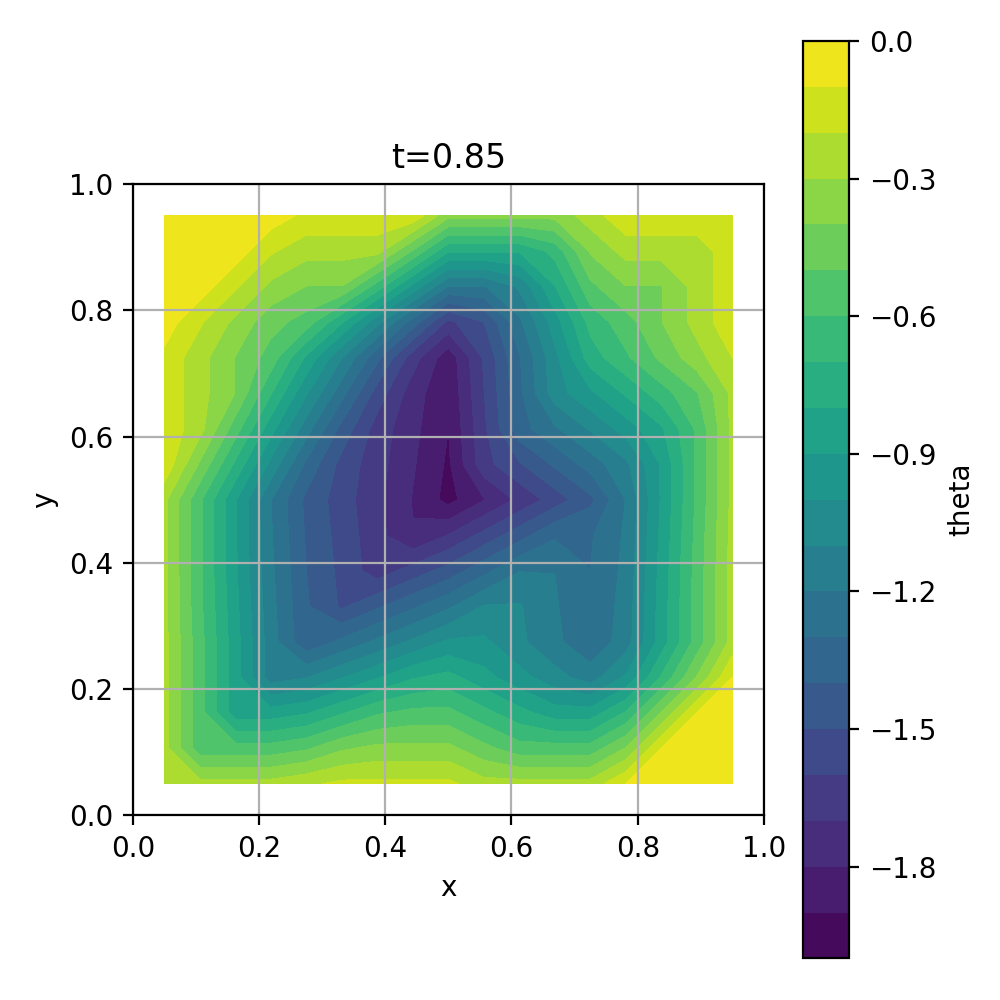}
        \caption{$t=0.85$}
        \label{fig:Ct085}
    \end{subfigure}
      \hfill
     \begin{subfigure}[b]{0.3\textwidth}
        \centering
        \includegraphics[width=\textwidth]{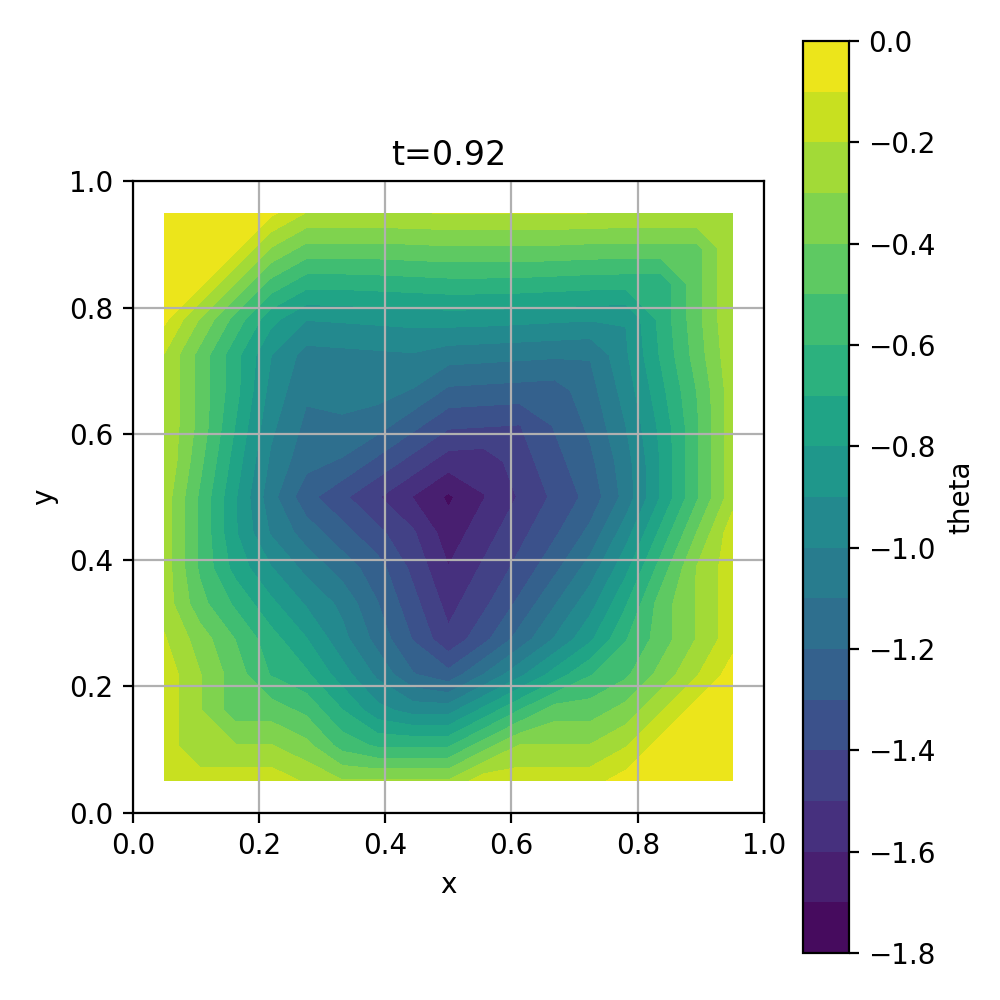}
        \caption{$t=0.92$}
        \label{fig:Ct092}
    \end{subfigure}
      \hfill
     \begin{subfigure}[b]{0.3\textwidth}
        \centering
        \includegraphics[width=\textwidth]{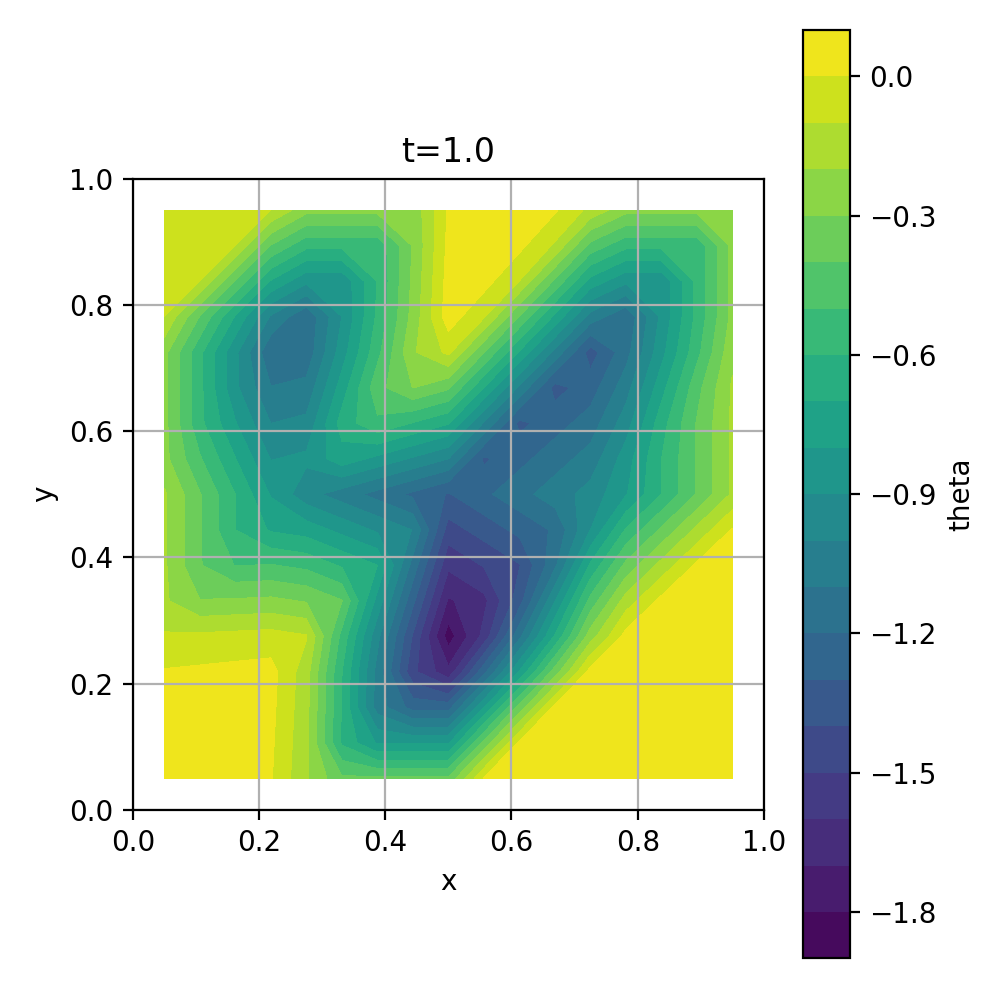}
        \caption{$t=1.0$}
        \label{fig:Ct10}
    \end{subfigure}
    \caption{Temporal evolution of the computed temperature $\theta$ at the plane $z=0.25$ for different time instances; the colour bar is rescaled in each panel.}
    \label{fig:theta_solutions_different_times}
\end{figure}
\end{example}

\subsection{Discrete energy and stability}\label{subsec-stability-exp}

\begin{example}[Discrete energy]\label{ex-stability}
This experiment illustrates the a~priori bounds of Theorem~\ref{bounds}. The exact solution is
\[
\u=\nabla\times\left(0,0,\psi\right),
\qquad
\psi(x,y,z,t)=g(x)\,g(y)\,g(z)\,e^{-t},
\qquad
g(s)=s^{2}(s-1)^{2},
\]
that is,
\[
\u=\Bigl(2x^{2}(x-1)^{2}y(y-1)(2y-1)z^{2}(z-1)^{2}e^{-t},\
-2x(x-1)(2x-1)y^{2}(y-1)^{2}z^{2}(z-1)^{2}e^{-t},\ 0\Bigr),
\]
with pressure $P^{\ast}=(2x-1)(2y-1)(1+t)$ and magnetic field
\[
\B=
\left(
\sin(t)\sin(\pi x)\cos(\pi y),
\,
-\sin(t)\cos(\pi x)\sin(\pi y),
\,
0
\right).
\]
with temperature $\theta=e^{-t}\sin(\pi x)\sin(\pi y)\sin(\pi z)$.

We take $p=3$ and fix $h=0.433$ (four subdivisions per coordinate direction), since only the
time evolution of the discrete norms is considered. Figure~\ref{fig:stability-norms} displays
$\|\B^{m}_{h}\|_{\L^{2}}$, $\|\nabla\times\B^{m}_{h}\|_{\L^{2}}$,
$\|\u^{m}_{h}\|_{\L^{2}}$, $\|\nabla\u^{m}_{h}\|_{\L^{2}}$,
$\|\u^{m}_{h}\|_{\L^{\infty}}$, $\|\theta^{m}_{h}\|_{L^{2}}$, and
$\|\nabla\theta^{m}_{h}\|_{L^{2}}$. All remain bounded and decay on $[0,1]$, consistent
with Theorem~\ref{bounds} and with mesh-independent stability for $p=3$.
   \begin{figure}[!ht]
    \centering
\includegraphics[width=0.58\textwidth]{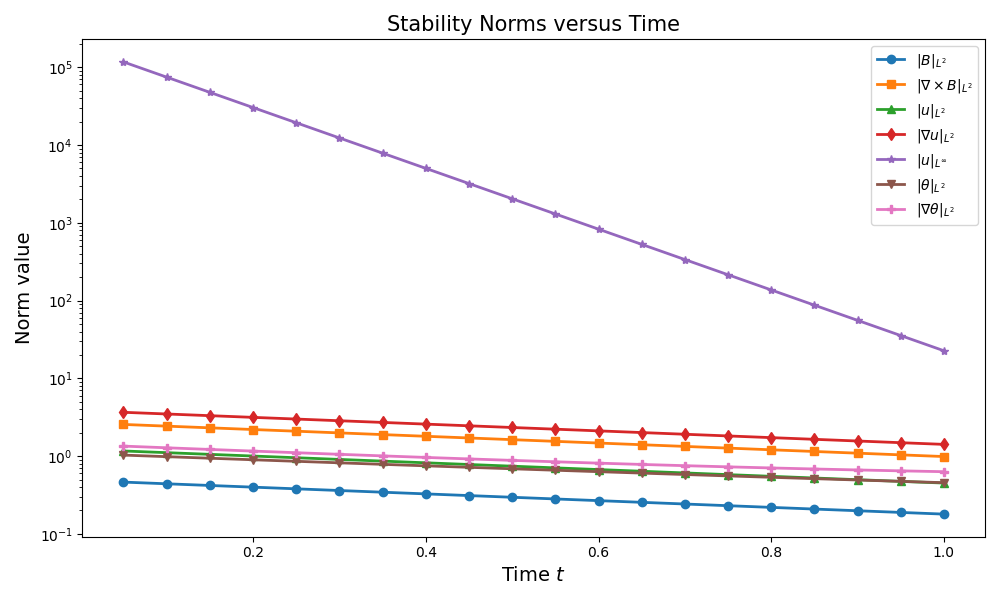}
    \caption{Discrete norms of $\u^{m}_{h},\B^{m}_{h},\theta^{m}_{h}$ versus time for $p=3$
    at $h=0.433$ (logarithmic vertical scale). All quantities entering
    \eqref{eq-bounds-1}-\eqref{eq-bounds-2} remain bounded and decay monotonically.}
   \label{fig:stability-norms}
\end{figure}
\end{example}

\subsection{Temporal convergence}\label{subsec-temporal-conv}

\begin{example}[Convergence in time]\label{ex-time-conv}
Here the spatial mesh is fixed and only $\Delta t$ is refined, so the observed slopes measure
the temporal discretisation error. The exact velocity, pressure and magnetic field are
{
\[
\mathbf{u}
=
2e^{-t}
\begin{pmatrix}
\sin^{2}(\pi x)\,\sin(\pi y)\cos(\pi y)\,\sin^{2}(\pi z)
\\[1.2ex]
-\sin(\pi x)\cos(\pi x)\,\sin^{2}(\pi y)\,\sin^{2}(\pi z)
\\[1.2ex]
0
\end{pmatrix}.
\]
with pressure
\[
P^{\ast}=e^{-t}\sin(\pi x)\sin(\pi y)\sin(\pi z),
\]
and magnetic field
\[
\B=
\left(
e^{-t}\sin(\pi x)\cos(\pi y)\cos(\pi z),
\,
-e^{-t}\cos(\pi x)\sin(\pi y)\cos(\pi z),
\,0
\right).
\]}
and the temperature is $\theta=e^{-t}\sin(\pi x)\sin(\pi y)\sin(\pi z)$.

We take $p=3$, fix $h=0.0866$ ($20$ subdivisions per coordinate direction), and successively
halve $\Delta t$. Errors are recorded at the final time in
$\mathbf L^{2},\mathbf H^{1},\mathbf L^{5},\mathbf W^{1,5}$ for $\u$,
$\mathbf L^{2}$ for $\B$ and $\nabla\times\B$, $L^{2}$ for $P^{\ast}$, and
$L^{2},H^{1}$ for $\theta$. Figure~\ref{fig:time-convergence} shows slopes close to one,
consistent with the $\mathcal O(\Delta t)$ rate obtained from \eqref{eq-main-error} and the
backward-Euler scheme. For $p=3$, $\mathbf W^{1,3}$ is the natural velocity norm; the
$\mathbf L^{5}$ and $\mathbf W^{1,5}$ curves are included as additional diagnostics.
\begin{figure}[!ht]
    \centering
    \includegraphics[width=0.6\textwidth]{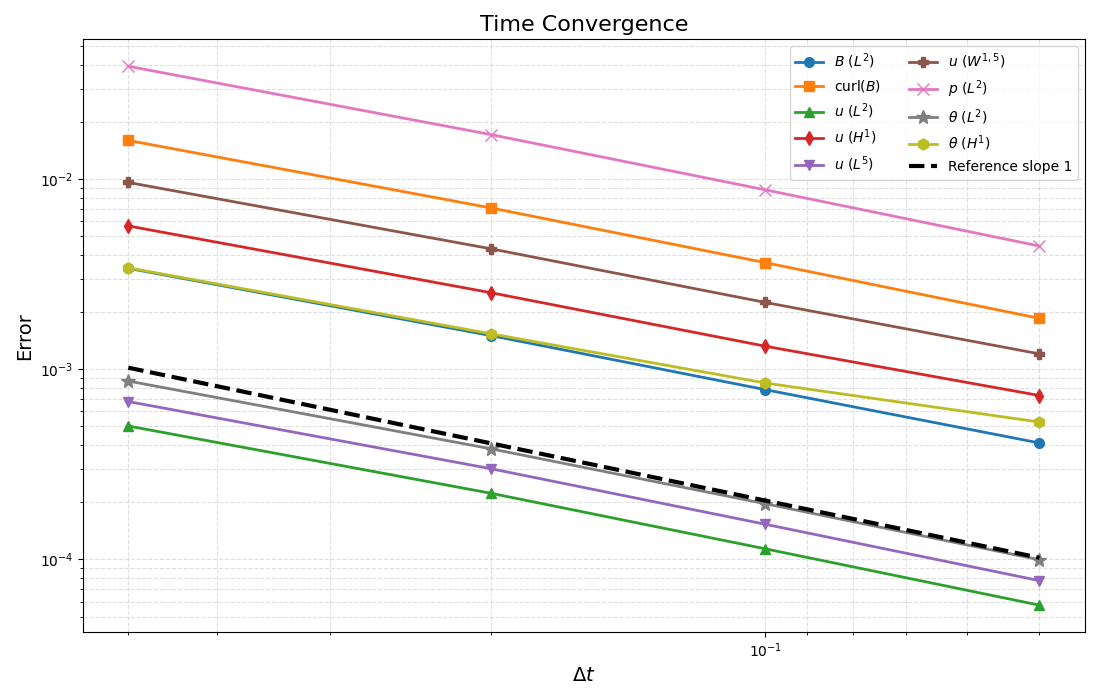}
    \caption{Errors at the final time versus $\Delta t$ for $p=3$ at fixed $h=0.0866$. All
    quantities converge with slope approximately one, in agreement with the first-order
    temporal accuracy implied by Theorem~\ref{Main theorem of error}.}
    \label{fig:time-convergence}
\end{figure}
\end{example}

\subsection{Space-time convergence}\label{subsec-spacetime-conv}

\begin{example}[Convergence under simultaneous refinement]\label{ex-spacetime-conv}
We refine $h$ and $\Delta t$ simultaneously for the shear-thickening case $p=5$. With
$g(s)=s^{3}(s-1)^{3}$, the exact velocity is
\[
\mathbf{u}(x,y,z,t)=
\begin{pmatrix}
e^{-t}g(x)g'(y)g(z)\\[2mm]
-e^{-t}g'(x)g(y)g(z)\\[2mm]
0
\end{pmatrix}.
\]
\noindent
The exact magnetic field is
\[
\mathbf{b}(x,y,z,t)=
\begin{pmatrix}
e^{-t}g(x)g'(y)g(z)\\[2mm]
-e^{-t}g'(x)g(y)g(z)\\[2mm]
0
\end{pmatrix}.
\]
\noindent
with $P^{\ast}=0$ and $\theta=e^{-t}\sin(\pi x)\sin(\pi y)\sin(\pi z)$. Both $\u$ and
$\B$ are divergence-free and vanish on $\partial\mathbb D$. Since $\B=\u$, the induction
coupling $\u\times\B$ vanishes, isolating the $p$-Laplace nonlinearity; the full coupling is
tested in Example~\ref{ex-temporal}.
\\ \noindent
The refinement is
\[
(h,\Delta t)
=
(0.866,0.25),
\,
(0.433,0.0625),
\,
(0.2165,0.015625),
\]
that is, $\Delta t\sim h^{2}$.
\\ \noindent
Tables~\ref{tab:space-time-B}-\ref{tab:space-time-theta} report the final-time errors and
experimental rates. Here $\mathcal E(\cdot)$ denotes the cumulative space-time norms defined
in the captions, distinct from the projected errors
$E^{n}_{\u},E^{n}_{\B},E^{n}_{\theta}$ of Section~\ref{sec-error}.

The rates are measured with respect to $h$ under the coupled refinement
$\Delta t\sim h^{2}$; thus a rate near two in $h$ corresponds to first order in time and
should not be interpreted as second-order temporal accuracy. The observed spatial rates are
higher than the guaranteed exponent $\frac{lp}{2(p-1)}$ in
\eqref{eq-main-error-total}, which is expected for the smooth manufactured solutions used
here and indicates that the theoretical bound need not be sharp in this regime.

\begin{table}[!ht]
\centering
\renewcommand{\arraystretch}{1.2}

\begin{tabular}{|c|c|cc|cc|}
\hline
$h$ & $\Delta t$
& \multicolumn{2}{c|}{$\B$ in $\L^2$}
& \multicolumn{2}{c|}{$\mathcal E(\B)$} \\
\cline{3-6}
& & Error & Rate & Error & Rate \\
\hline
0.866000 & 0.250000 & $3.377912\times10^{-3}$ & -     & $1.001601\times10^{-2}$ & --     \\
0.433000 & 0.062500 & $8.920679\times10^{-4}$ & 1.9209 & $2.633456\times10^{-3}$ & 1.9273 \\
0.216500 & 0.015625 & $2.263062\times10^{-4}$ & 1.9789 & $6.643965\times10^{-4}$ & 1.9868 \\
\hline
\end{tabular}
\caption{\textbf{Errors and convergence rates for the magnetic field $\B$ and $\mathcal E(\B)=\left(\Delta t\sum_{n=1}^{m}\kappa\left\|\nabla\times(\B^{n}-\mathfrak{B}^{n})\right\|^{2}_{\L^{2}}\right)^{\frac{1}{2}}$ for $p=2$ under simultaneous refinement of $h$ and $\Delta t$ according to $\Delta t\sim h^{2}$}.}
\label{tab:space-time-B}
\end{table}

\vspace{0.2cm}

\begin{table}[!ht]
\centering
\renewcommand{\arraystretch}{1.2}

\begin{tabular}{|c|c|cc|cc|}
\hline
$h$ & $\Delta t$
& \multicolumn{2}{c|}{$\u$ in $\L^2$}
& \multicolumn{2}{c|}{$\mathcal E(\u)$} \\
\cline{3-6}
& & Error & Rate & Error & Rate \\
\hline
0.866000 & 0.250000
& $6.072698\times10^{-4}$ & --
& $7.425291\times10^{-3}$ & -- \\

0.433000 & 0.062500
& $9.661759\times10^{-5}$ &{2.6520}
& $1.813942\times10^{-4}$ & 1.9945 \\

0.216500 & 0.015625
& $1.981730\times10^{-5}$ & 2.2855
& $4.508806\times10^{-4}$ & 2.0083 \\
\hline
\end{tabular}
\caption{\textbf{Errors and convergence rates for the velocity field $\u$ in $\L^2$ and $\mathcal E(\u)=\left(\Delta t\sum_{n=1}^{m}\nu\|\nabla(\u^{n}-\mathbf{U}^{n})\|^{p}_{\L^p}\right)^{\frac{1}{p}}$ for $p=2$ under simultaneous refinement of $h$ and $\Delta t$according to $\Delta t\sim h^{2}$.}}
\label{tab:space-time-u}
\end{table}

\begin{table}[!ht]
\centering
\renewcommand{\arraystretch}{1.2}

\begin{tabular}{|c|c|cc|cc|}
\hline
$h$ & $\Delta t$
& \multicolumn{2}{c|}{$\mathcal E(\u)$}
& \multicolumn{2}{c|}{{$P^{\ast}$ in $L^2$}} \\
\cline{3-6}
& & Error & Rate & Error & Rate \\
\hline
0.866000 & 0.250000
& $9.487001\times10^{-3}$ & --
& $1.870904\times10^{-2}$ & -- \\

0.433000 & 0.062500
& $2.368073\times10^{-3}$ & 2.0022
& $5.843580\times10^{-3}$ & 1.6788 \\

0.216500 & 0.015625
& $5.884803\times10^{-4}$ & 2.0086
& $1.416383\times10^{-3}$ & 2.0446 \\
\hline
\end{tabular}
\caption{\textbf{Errors and convergence rates for the pressure $P^{\ast}$ and $\mathcal E(\u)$ for $p=5$. The pressure is not an unknown of the scheme \eqref{eq-FEM-1}-\eqref{eq-FEM-3}, which is posed on the discretely divergence-free space $\mathcal X_{0h}$; it is recovered a posteriori from the momentum residual by means of the discrete inf-sup condition of Lemma~\ref{lemma:inf-sup}, and no error estimate for it is claimed in Theorem~\ref{Main theorem of error}.}}
\label{tab:space-time-pressure}
\end{table}
\begin{table}[!ht]
\centering
\renewcommand{\arraystretch}{1.2}

\begin{tabular}{|c|c|cc|cc|}
\hline
$h$ & $\Delta t$
& \multicolumn{2}{c|}{$\theta$ in $L^2$}
& \multicolumn{2}{c|}{$\mathcal E(\theta)$} \\
\cline{3-6}
& & Error & Rate & Error & Rate \\
\hline
0.866000 & 0.250000
& $1.146766\times10^{-3}$ & --
& $1.642463\times10^{-2}$ & -- \\

0.433000 & 0.062500
& $1.921255\times10^{-4}$ & 2.5774
& $4.140296\times10^{-3}$ & 1.9881 \\

0.216500 & 0.015625
& $4.006077\times10^{-5}$ & 2.2618
& $1.037405\times10^{-3}$ & 1.9968 \\
\hline
\end{tabular}
\caption{\textbf{Errors and convergence rates for the temperature $\theta$ in $L^2$ and  $\mathcal E(\theta)=\left(\Delta t\sum_{n=1}^{m}\sigma\left\|\nabla(\theta^{n}-\Theta^{n})\right\|^{2}_{{L}^{2}}\right)^{\frac{1}{2}}$ for $p=2$ under simultaneous refinement of $h$ and $\Delta t$ according to $\Delta t\sim h^{2}$}.}
\label{tab:space-time-theta}
\end{table}
\end{example}

\medskip

\section{Conclusion and future directions}\label{sec-conclusion}

We have analysed a fully discrete mixed finite element approximation of the
three-dimensional rotating incompressible MHD-Boussinesq system with $p$-Laplace viscosity,
$p\ge2$. The method combines a semi-implicit backward Euler discretisation with
conforming finite elements for the velocity and temperature and
$\mathbf H(\operatorname{curl})$-conforming N\'ed\'elec edge elements for the magnetic
induction. The scheme is well posed at every time level for arbitrary $h>0$ and
$\Delta t>0$; see Theorem~\ref{thm-existence} and Lemma~\ref{lem-uniqueness}.
Moreover, Theorem~\ref{bounds} yields a mesh-independent energy estimate, with no
coupling condition between $\Delta t$ and $h$. A key structural feature is the exact
cancellation between the Lorentz force and the induction coupling induced by the
semi-implicit time staggering.

As $h,\Delta t\to0$, the discrete solutions converge, along a subsequence, to a limit
solution of the continuous system; see Theorem~\ref{Thm-Conv-scheme}. Since the
$p$-Laplacian does not provide a suitable uniform estimate for the discrete time
derivative of the velocity, strong compactness is obtained through a time-translate
estimate and Simon's criterion. The weak limit of the nonlinear diffusion flux is then
identified by Minty's method using the energy $\limsup$ inequality
\eqref{eq-limsup-key}. Under Assumption~\ref{Assumption} and for a pointwise
divergence-free velocity pair, we further obtain the error estimate
\[
\mathcal O\!\left(\Delta t+h^{\frac{lp}{2(p-1)}}\right)
\]
in $L^\infty(0,T;\mathbf L^2)$; see Theorem~\ref{Main theorem of error}. For $p=2$,
this reduces to the classical rate $\mathcal O(\Delta t+h^l)$.

Several directions remain open. First, identifying the $p$-Laplacian flux outside the
range \eqref{eq-p-threshold} appears to require additional tools, such as
Lipschitz-truncation techniques
\cite{DieningRuzickaWolf2010,DieningKreuzerSuli2013,SuliTscherpel2020}; see
Remark~\ref{rem-p-threshold}. Second, the spatial exponent
$\frac{lp}{p-1}$ arises from estimating the nonlinear consistency error in the
$\mathbf L^p$ scale, while the experiments of Section~\ref{sec-numerics} suggest that
this rate may not be sharp for smooth solutions. A natural approach to improving it is
through the quasi-norm framework of Barrett and Liu \cite{BarrettLiu1993} and Diening
and Ruzicka \cite{DieningRuzicka2007}. Finally, the pointwise divergence-free condition
\eqref{eq-exactly-divfree}, used in Section~\ref{sec-error} to obtain a purely quadratic
Gr\"onwall argument, may be relaxed by a more refined treatment of the divergence
remainder in the convective terms.

\bigskip

\unsection{Declarations}

\smallskip
\noindent\textbf{Acknowledgements.} The authors gratefully acknowledge Dr.~Kim Ngan Le for many helpful discussions, insightful comments, and valuable suggestions during the development of this work.

\smallskip
\noindent\textbf{Funding.} This work was supported by the Department of Science and Technology (DST), Government of India, through the INSPIRE Faculty Award, grant
no.\ IFA21-MA 158.

\medskip
\noindent\textbf{Conflict of interest.} The authors have no competing interests to declare
that are relevant to the content of this article.

\medskip
\noindent\textbf{Data availability.}
The numerical data supporting the findings of this study were generated through the finite element simulations described in Section~\ref{sec-numerics}. The relevant results are reported in the manuscript, and additional computational data are available from the corresponding author upon reasonable request.


\bigskip

\bibliographystyle{plainnat}
\bibliography{sample}

\begin{thebibliography}{52}
\providecommand{\natexlab}[1]{#1}
\providecommand{\url}[1]{\texttt{#1}}
\expandafter\ifx\csname urlstyle\endcsname\relax
  \providecommand{\doi}[1]{doi: #1}\else
  \providecommand{\doi}{doi: \begingroup \urlstyle{rm}\Url}\fi

\bibitem[Amrouche et~al.(1998)Amrouche, Bernardi, Dauge, and Girault]{AmroucheBernardiDaugeGirault1998}
C.~Amrouche, C.~Bernardi, M.~Dauge, and V.~Girault.
\newblock Vector potentials in three-dimensional non-smooth domains.
\newblock \emph{Mathematical Methods in the Applied Sciences}, 21:\penalty0 823--864, 1998.

\bibitem[Banas and Prohl(2010)]{BanasProhl2010}
L.~Banas and A.~Prohl.
\newblock Convergent finite element discretization of the multi-fluid nonstationary incompressible magnetohydrodynamics equations.
\newblock \emph{Math. Comp.}, 79:\penalty0 1957--1999, 2010.

\bibitem[Barrett and Liu(1993)]{BarrettLiu1993}
J.~W. Barrett and W.~B. Liu.
\newblock Finite element approximation of the {$p$}-{L}aplacian.
\newblock \emph{Math. Comp.}, 61:\penalty0 523--537, 1993.

\bibitem[Barrett and Liu(1994)]{BarrettLiu1994}
J.~W. Barrett and W.~B. Liu.
\newblock Finite element approximation of the parabolic $p$-laplacian.
\newblock \emph{SIAM Journal on Numerical Analysis}, 31\penalty0 (2):\penalty0 413--428, 1994.

\bibitem[Berselli and Ruzika(2022)]{BerselliRuzicka2022}
L.~C. Berselli and M.~Ruzika.
\newblock Space‐time discretization for nonlinear parabolic systems with $p$‐structure.
\newblock \emph{IMA Journal of Numerical Analysis}, 42\penalty0 (1):\penalty0 260--299, 2022.

\bibitem[Berselli et~al.(2015)Berselli, Diening, and Ruzicka]{BerselliDieningRuzicka2015}
L.~C. Berselli, L.~Diening, and M.~Ruzicka.
\newblock Optimal error estimate for semi-implicit space-time discretization for the equations describing incompressible generalized newtonian fluids.
\newblock \emph{IMA Journal of Numerical Analysis}, 35\penalty0 (2):\penalty0 680--697, 2015.

\bibitem[Bian(2016)]{bian2016initial}
D.~Bian.
\newblock Initial boundary value problem for two-dimensional viscous boussinesq equations for mhd convection.
\newblock \emph{Discrete and Continuous Dynamical Systems - Series S}, 9\penalty0 (6):\penalty0 1591--1611, 2016.

\bibitem[Bian and Gui(2016)]{bian2016boussinesq}
D.~Bian and G.~Gui.
\newblock On 2-d boussinesq equations for mhd convection with stratification effects.
\newblock \emph{Journal of Differential Equations}, 261\penalty0 (3):\penalty0 1669--1711, 2016.

\bibitem[Bian and Liu(2017)]{bian2017initial}
D.~Bian and J.~Liu.
\newblock Initial--boundary value problem to 2d boussinesq equations for mhd convection with stratification effects.
\newblock \emph{Journal of Differential Equations}, 263\penalty0 (12):\penalty0 8074--8101, 2017.

\bibitem[Breit et~al.(2021)Breit, Diening, Storn, and Wichmann]{BreitDieningStornWichmann2021}
D.~Breit, L.~Diening, J.~Storn, and J.~Wichmann.
\newblock The parabolic $p$‐laplacian with fractional differentiability.
\newblock \emph{IMA Journal of Numerical Analysis}, 41\penalty0 (3):\penalty0 2110--2138, 2021.

\bibitem[Cheng et~al.(2009)Cheng, Takahashi, and Forghani]{Cheng2009}
Z.~Cheng, N.~Takahashi, and B.~Forghani.
\newblock \emph{Electromagnetic and Thermal Field Modeling and Application in Electrical Engineering}.
\newblock Science Press, Beijing, 2009.

\bibitem[Costabel and Dauge(1998)]{CostabelDauge1998}
M.~Costabel and M.~Dauge.
\newblock Singularities of maxwell's equations on polyhedral domains.
\newblock In \emph{Analysis, Numerics and Applications of Differential and Integral Equations (Stuttgart, 1996)}, volume 379 of \emph{Pitman Research Notes in Mathematics Series}, pages 69--76. Longman, Harlow, 1998.

\bibitem[Costabel and Dauge(2000)]{CostabelDauge2000}
M.~Costabel and M.~Dauge.
\newblock Singularities of electromagnetic fields in polyhedral domains.
\newblock \emph{Archive for Rational Mechanics and Analysis}, 151:\penalty0 221--276, 2000.

\bibitem[Dauge(1989)]{Dauge1989}
M.~Dauge.
\newblock Stationary stokes and navier-stokes systems on two- or three-dimensional domains with corners. i. linearized equations.
\newblock \emph{SIAM Journal on Mathematical Analysis}, 20\penalty0 (1):\penalty0 74--97, 1989.

\bibitem[Diening and Ruzicka(2007)]{DieningRuzicka2007}
L.~Diening and M.~Ruzicka.
\newblock Interpolation operators in {O}rlicz--{S}obolev spaces.
\newblock \emph{Numer. Math.}, 107:\penalty0 107--129, 2007.

\bibitem[Diening et~al.(2007)Diening, Ebmeyer, and Ruzika]{DieningEbmeyerRuzicka2007}
L.~Diening, C.~Ebmeyer, and M.~Ruzika.
\newblock Optimal convergence for the implicit space-time discretization of parabolic systems with $p$-structure.
\newblock \emph{SIAM Journal on Numerical Analysis}, 45\penalty0 (2):\penalty0 457--472, 2007.

\bibitem[Diening et~al.(2010)Diening, Ruzicka, and Wolf]{DieningRuzickaWolf2010}
L.~Diening, M.~Ruzicka, and J.~Wolf.
\newblock Existence of weak solutions for unsteady motions of generalized {N}ewtonian fluids.
\newblock \emph{Ann. Sc. Norm. Super. Pisa Cl. Sci. (5)}, 9:\penalty0 1--46, 2010.

\bibitem[Diening et~al.(2013)Diening, Kreuzer, and S{\"u}li]{DieningKreuzerSuli2013}
L.~Diening, C.~Kreuzer, and E.~S{\"u}li.
\newblock Finite element approximation of steady flows of incompressible fluids with implicit power-law-like rheology.
\newblock \emph{SIAM J. Numer. Anal.}, 51:\penalty0 984--1015, 2013.

\bibitem[Ding and Li(2025)]{DING2025116470}
Q.~Ding and M.~Li.
\newblock Convergence analysis of finite element method for incompressible magnetohydrodynamics system with variable density.
\newblock In \emph{Proceedings of the Conference on Computational and Applied Mathematics}, volume 462, page 116470, 2025.
\newblock \doi{https://doi.org/10.1016/j.cam.2024.116470}.

\bibitem[Ding et~al.(2022)Ding, Long, and Mao]{ding2022convergence}
Q.~Ding, X.~Long, and S.~Mao.
\newblock Convergence analysis of a fully discrete finite element method for thermally coupled incompressible mhd problems with temperature-dependent coefficients.
\newblock \emph{ESAIM: Mathematical Modelling and Numerical Analysis}, 56\penalty0 (3):\penalty0 969--1005, 2022.

\bibitem[Droniou et~al.(2022)Droniou, Goldys, and Le]{DroniouGoldysLe2022}
J.~Droniou, B.~Goldys, and K.-N. Le.
\newblock Design and convergence analysis of numerical methods for stochastic evolution equations with leray-lions operator.
\newblock \emph{IMA Journal of Numerical Analysis}, 42\penalty0 (2):\penalty0 1143--1179, 2022.

\bibitem[Eckhard and Wróblewska-Kamińska(2013)]{EmmrichWroblewska2013}
E.~Eckhard and A.~Wróblewska-Kamińska.
\newblock Convergence of a full discretization of quasi-linear parabolic equations in isotropic and anisotropic orlicz spaces.
\newblock \emph{SIAM Journal on Numerical Analysis}, 51\penalty0 (3):\penalty0 1163--1184, 2013.

\bibitem[Eckstein and Ruzicka(2018)]{EcksteinRuzicka2018}
S.~Eckstein and M.~Ruzicka.
\newblock On the full space-time discretization of the generalized navier-stokes equations: the dirichlet case.
\newblock \emph{SIAM Journal on Numerical Analysis}, 56\penalty0 (4):\penalty0 2234--2261, 2018.

\bibitem[Gerbeau et~al.(2006)Gerbeau, Bris, and Leli{\`e}vre]{GerbeauLeBrisLelievre2006}
J.-F. Gerbeau, C.~Le Bris, and T.~Leli{\`e}vre.
\newblock \emph{Mathematical Methods for the Magnetohydrodynamics of Liquid Metals}.
\newblock Numerical Mathematics and Scientific Computation. Oxford University Press, Oxford, 2006.

\bibitem[Girault and Raviart(1986)]{GiraultRaviart1986}
V.~Girault and P.-A. Raviart.
\newblock \emph{Finite Element Methods for Navier--Stokes Equations: Theory and Algorithms}, volume~5 of \emph{Springer Series in Computational Mathematics}.
\newblock Springer-Verlag, Berlin, 1986.

\bibitem[Guzman and Neilan(2014)]{GuzmanNeilan2014}
J.~Guzman and M.~Neilan.
\newblock Conforming and divergence-free {S}tokes elements in three dimensions.
\newblock \emph{IMA J. Numer. Anal.}, 34:\penalty0 1489--1508, 2014.

\bibitem[H.(2007)]{Hansen2007}
Eskil H.
\newblock Galerkin runge–kutta discretizations of nonlinear parabolic equations.
\newblock \emph{Journal of Computational and Applied Mathematics}, 205\penalty0 (2):\penalty0 882--890, 2007.

\bibitem[Hashizume(2006)]{hashizume2006numerical}
H.~Hashizume.
\newblock Numerical and experimental research to solve mhd problems in liquid blankets system.
\newblock \emph{Fusion Engineering and Design}, 81\penalty0 (8–14):\penalty0 1431--1438, 2006.

\bibitem[Heywood and Rannacher(1990)]{HeywoodRannacher1990}
John~G. Heywood and Rolf Rannacher.
\newblock Finite element approximation of the nonstationary {N}avier--{S}tokes problem. {IV}. {E}rror analysis for second-order time discretization.
\newblock \emph{SIAM J. Numer. Anal.}, 27:\penalty0 353--384, 1990.

\bibitem[Hiptmair(2002)]{Hiptmair2002}
R.~Hiptmair.
\newblock Finite elements in computational electromagnetism.
\newblock \emph{Acta Numerica}, 11:\penalty0 237--339, 2002.

\bibitem[Le et~al.(2026)Le, Panda, and Sarkar]{LePandaSarkar2026}
K.~N. Le, A.~A. Panda, and E.~I. Sarkar.
\newblock Stability, convergence, and error analysis of finite element methods for 3d magnetohydrodynamics with p-laplacian viscosity, 2026.
\newblock URL \url{https://arxiv.org/abs/2608.24197}.

\bibitem[Li et~al.(2022)Li, Qiu, and Yang]{LiQiuYang2022}
B.~Li, W.~Qiu, and Z.~Yang.
\newblock A convergent post-processed discontinuous {Galerkin} method for incompressible flow with variable density.
\newblock \emph{J. Sci. Comput.}, 91\penalty0 (2):\penalty0 28, 2022.

\bibitem[Li and Li(2024)]{LiLi2024}
S.~Li and Y.~Li.
\newblock Optimal error analysis of an unconditionally stable {BDF2} finite element approximation for the 3d incompressible {MHD} equations with variable density.
\newblock \emph{J. Comput. Appl. Math.}, 445:\penalty0 18, 2024.

\bibitem[Lindqvist(2006)]{Lindqvist2006}
P.~Lindqvist.
\newblock \emph{Notes on the $p$-{L}aplace Equation}.
\newblock University of Jyv{\"a}skyl{\"a}, Report 102, 2006.

\bibitem[Liu and Walkington(2007)]{LiuWalkington2007}
C.~Liu and N.~J. Walkington.
\newblock Convergence of numerical approximations of the incompressible navier-stokes equations with variable density and viscosity.
\newblock \emph{SIAM Journal on Numerical Analysis}, 45\penalty0 (3):\penalty0 1287--1304, 2007.

\bibitem[Marchioro and Pulvirenti(1994)]{marchioro1994mathematical}
C.~Marchioro and M.~Pulvirenti.
\newblock \emph{Mathematical Theory of Incompressible Nonviscous Fluids}.
\newblock Springer-Verlag, New York, 1994.

\bibitem[Meir(1995)]{Meir1995}
A.~J. Meir.
\newblock Thermally coupled, stationary, incompressible mhd flow; existence, uniqueness, and finite element approximation.
\newblock \emph{Numerical Methods for Partial Differential Equations}, 11:\penalty0 311--337, 1995.

\bibitem[Meir and Schmidt(2001)]{MeirSchmidt2001}
A.~J. Meir and P.~G. Schmidt.
\newblock On electromagnetically and thermally driven liquid-metal flows.
\newblock In \emph{Proceedings of the Third World Congress of Nonlinear Analysts, Part 5 (Catania, 2000)}, volume~47, pages 3281--3294, 2001.

\bibitem[Monk(2003)]{Monk2003}
P.~Monk.
\newblock \emph{Finite Element Methods for Maxwell's Equations}.
\newblock Numerical Mathematics and Scientific Computation. Oxford University Press, New York, 2003.

\bibitem[Mulone and Rionero(2003)]{mulone2003necessary}
G.~Mulone and S.~Rionero.
\newblock Necessary and sufficient conditions for nonlinear stability in the magnetic bénard problem.
\newblock \emph{Archive for Rational Mechanics and Analysis}, 166\penalty0 (3):\penalty0 197--218, 2003.

\bibitem[Prohl(2008)]{Prohl2008}
A.~Prohl.
\newblock Convergent finite element discretizations of the nonstationary incompressible magnetohydrodynamics system.
\newblock \emph{M2AN Mathematical Modelling and Numerical Analysis}, 42\penalty0 (6):\penalty0 1065--1087, 2008.

\bibitem[Qiu(2020)]{Qiu2020}
H.~Qiu.
\newblock Error analysis of euler semi-implicit scheme for the nonstationary magneto-hydrodynamics problem with temperature dependent parameters.
\newblock \emph{Journal of Scientific Computing}, 85:\penalty0 1--26, 2020.

\bibitem[Ravindran(2019)]{Ravindran2019}
S.~S. Ravindran.
\newblock Partitioned time-stepping scheme for an mhd system with temperature-dependent coefficients.
\newblock \emph{IMA Journal of Numerical Analysis}, 39:\penalty0 1860--1887, 2019.

\bibitem[Simon(1987)]{Simon1987}
J.~Simon.
\newblock Compact sets in the space {$L^p(0,T;B)$}.
\newblock \emph{Ann. Mat. Pura Appl. (4)}, 146:\penalty0 65--96, 1987.

\bibitem[Smoletntsev et~al.(2010)Smoletntsev, Moreau, Buehler, et~al.]{smoletntsev2010mhd}
S.~Smoletntsev, R.~Moreau, L.~Buehler, et~al.
\newblock Mhd thermofluid issues of liquid-metal blankets: phenomena and advances.
\newblock \emph{Fusion Engineering and Design}, 85\penalty0 (7–9):\penalty0 1196--1205, 2010.

\bibitem[S{\"u}li and Tscherpel(2020)]{SuliTscherpel2020}
E.~S{\"u}li and T.~Tscherpel.
\newblock Fully discrete finite element approximation of unsteady flows of implicitly constituted incompressible fluids.
\newblock \emph{IMA J. Numer. Anal.}, 40:\penalty0 801--849, 2020.

\bibitem[Tabata and Tagami(2005)]{TabataTagami2005}
M.~Tabata and D.~Tagami.
\newblock Error estimates of finite element methods for nonstationary thermal convection problems with temperature-dependent coefficients.
\newblock \emph{Numerische Mathematik}, 100:\penalty0 351--372, 2005.

\bibitem[Temam(1977)]{Temam1977}
R.~Temam.
\newblock \emph{Navier--Stokes Equations: Theory and Numerical Analysis}.
\newblock North-Holland, Amsterdam, 1977.

\bibitem[Walkington(2005)]{Walkington2005}
N.~J. Walkington.
\newblock Convergence of the discontinuous galerkin method for discontinuous solutions.
\newblock \emph{SIAM Journal on Numerical Analysis}, 42\penalty0 (5):\penalty0 1801--1817, 2005.

\bibitem[Wang and Li(2024)]{MR4783950}
S.~Wang and Y.~Li.
\newblock Optimal convergent analysis of a linearized {E}uler finite element scheme for the 2{D} incompressible temperature-dependent {MHD}-{B}oussinesq equations.
\newblock \emph{Commun. Nonlinear Sci. Numer. Simul.}, 138:\penalty0 Paper No. 108264, 17, 2024.
\newblock ISSN 1007-5704,1878-7274.

\bibitem[Ye(2021)]{ye2021global}
Z.~Ye.
\newblock Global regularity of 2d temperature-dependent mhd-boussinesq equations with zero thermal diffusivity.
\newblock \emph{Journal of Differential Equations}, 293:\penalty0 447--481, 2021.

\bibitem[Zhang(2005)]{Zhang2005}
S.~Zhang.
\newblock A new family of stable mixed finite elements for the {3D} {S}tokes equations.
\newblock \emph{Math. Comp.}, 74:\penalty0 543--554, 2005.

\end{thebibliography}

\end{document}